\documentclass{amsart}
\usepackage{amssymb, amsmath, mathrsfs, verbatim, url}
\usepackage{mathabx}

\makeatletter
\@namedef{subjclassname@2020}{%
  \textup{2020} Mathematics Subject Classification}
\makeatother

\theoremstyle{plain}
\newtheorem{theorem}{Theorem}[section]
\newtheorem{lemma}[theorem]{Lemma}
\newtheorem{proposition}[theorem]{Proposition}
\newtheorem{corollary}[theorem]{Corollary}

\theoremstyle{definition}
\newtheorem{definition}[theorem]{Definition}

\newcommand{\re}{\upharpoonright}
\newcommand{\diam}{\mathsf{diam}}
\newcommand{\maxi}{\mathsf{max}}
\newcommand{\limi}{\mathsf{lim}}
\newcommand{\id}{\mathsf{id}}

\newcommand{\supi}{\mathsf{sup}}
\newcommand{\bool}{\mathsf{b}}
\newcommand{\bD}{\mathbf{\Delta}}
\newcommand{\bG}{\mathbf{\Gamma}}
\newcommand{\bGc}{\widecheck{\mathbf{\Gamma}}}
\newcommand{\bL}{\mathbf{\Lambda}}
\newcommand{\bLc}{\widecheck{\mathbf{\Lambda}}}
\newcommand{\bP}{\mathbf{\Pi}}
\newcommand{\bPc}{\widecheck{\mathbf{\Pi}}}
\newcommand{\bS}{\mathbf{\Sigma}}
\newcommand{\Borel}{\mathsf{B}}
\newcommand{\sss}{\mathsf{s}}
\newcommand{\wc}{\!\!\downarrow\,}
\newcommand{\Ne}{\mathsf{N}}
\newcommand{\Ha}{\mathsf{Ha}}
\newcommand{\HaS}{\mathsf{Ha}_{\bS}}
\newcommand{\Lo}{\mathsf{Lo}}
\newcommand{\LoS}{\mathsf{Lo}_{\bS}}

\newcommand{\ZF}{\mathsf{ZF}}
\newcommand{\DC}{\mathsf{DC}}
\newcommand{\Det}{\mathsf{Det}}
\newcommand{\Ga}{\mathsf{G}}
\newcommand{\BP}{\mathsf{BP}}
\newcommand{\EW}{\mathsf{EW}}
\newcommand{\NSD}{\mathsf{NSD}}
\newcommand{\NSDS}{\mathsf{NSD}_{\bS}}
\newcommand{\PU}{\mathsf{PU}}

\newcommand{\AD}{\mathsf{AD}}

\newcommand{\SD}{\mathsf{SD}}

\newcommand{\Diff}{\mathsf{D}}
\newcommand{\Wa}{\mathsf{Wa}}

\newcommand{\Aa}{\mathcal{A}}
\newcommand{\BB}{\mathcal{B}}
\newcommand{\CC}{\mathcal{C}}

\newcommand{\HH}{\mathcal{H}}

\newcommand{\II}{\mathcal{I}}

\newcommand{\UU}{\mathcal{U}}
\newcommand{\VV}{\mathcal{V}}
\newcommand{\PP}{\mathcal{P}}

\newcommand{\RRR}{\mathbb{R}}

\begin{document}

\title{Constructing Wadge classes}

\author{Rapha\"{e}l Carroy}
\address{Dipartimento di Matematica ``Giuseppe Peano''
\newline\indent Palazzo Campana, Universit\'{a} di Torino
\newline\indent Via Carlo Alberto 10
\newline\indent 10123 Turin, Italy}
\email{raphael.carroy@unito.it}

\author{Andrea Medini}
\address{Institut f\"{u}r Diskrete Mathematik und Geometrie
\newline\indent Technische Universit\"{a}t Wien
\newline\indent  Wiedner Hauptstra{\ss}e 8–-10/104
\newline\indent 1040 Vienna, Austria}
\email{andrea.medini@tuwien.ac.at}
\urladdr{http://www.dmg.tuwien.ac.at/medini/}

\author{Sandra M\"{u}ller}
\address{Institut f\"{u}r Diskrete Mathematik und Geometrie
\newline\indent Technische Universit\"{a}t Wien
\newline\indent  Wiedner Hauptstra{\ss}e 8–-10/104
\newline\indent 1040 Vienna, Austria}
\email{sandra.mueller@tuwien.ac.at}
\urladdr{http://muellersandra.github.io}



\thanks{The first-listed author acknowledges the support of the FWF grant P 28153 and of the PRIN grant 2017NWTM8R. The second-listed author acknowledges the support of the FWF grant P 30823. The third-listed author (formerly known as Sandra Uhlenbrock) acknowledges the support of the FWF grants P 28157 and V 844, as well as funding from L'OR\'{E}AL Austria, in collaboration with the Austrian UNESCO Commission and in cooperation with the Austrian Academy of Sciences, through the fellowship ``Determinacy and Large Cardinals''. The authors are grateful to Alessandro Andretta for valuable bibliographical help, to Alain Louveau for allowing them to use his unpublished book \cite{louveaub}, and to Salvatore Scamperti for spotting several inaccuracies.}

\date{March 15, 2022}

\begin{abstract}
We show that, assuming the Axiom of Determinacy, every non-selfdual Wadge class can be constructed by starting with those of level $\omega_1$ (that is, the ones that are closed under Borel preimages) and iteratively applying the operations of expansion and separated differences. The proof is essentially due to Louveau, and it yields at the same time a new proof of a theorem of Van Wesep (namely, that every non-selfdual Wadge class can be expressed as the result of a Hausdorff operation applied to the open sets). The exposition is self-contained, except for facts from classical descriptive set theory.
\end{abstract}

\maketitle

\tableofcontents

\section{Introduction}\label{sectionintroduction}

Throughout this article, unless we specify otherwise, we will be working in the theory $\ZF+\DC$, that is, the usual axioms of Zermelo-Fraenkel (without the Axiom of Choice) plus the principle of Dependent Choices (see Section \ref{sectionpreliminaries} for more details). Given a set $Z$, we will denote by $\PP(Z)$ the collection of all subsets of $Z$. By \emph{space}, we will always mean separable metrizable topological space.

Given a space $Z$, we will say that $\bG$ is a \emph{Wadge class} in $Z$ if there exists $A\subseteq Z$ such that
$$
\bG=\{f^{-1}[A]:f:Z\longrightarrow Z\text{ is a continuous function}\}.
$$
Given a set $Z$ and $\bG\subseteq\PP(Z)$, define $\bGc=\{Z\setminus A:A\in\bG\}$. We will say that $\bG$ is \emph{selfdual} if $\bGc=\bG$. Observe that $\{\varnothing\}$ and $\{Z\}$ are non-selfdual Wadge classes whenever $Z$ is non-empty. The systematic study of these classes, founded by William Wadge in his doctoral thesis \cite{wadget} (see also \cite{wadgea}), is known as Wadge theory, and it has become a classical topic in descriptive set theory (see \cite[Section 21.E]{kechris}). Under suitable determinacy assumptions, the collection of all Wadge classes on a zero-dimensional Polish space $Z$, ordered by inclusion, constitutes a well-behaved hierarchy that is similar to, but much finer than the well-known Borel hierarchy (and not limited to sets of low complexity).

In \cite{louveaua}, Alain Louveau gave a complete description of the non-selfdual Borel Wadge classes, by using an iterative process built on five basic operations.
Subsequently, in his unpublished book \cite{louveaub}, he obtained the following theorem, that reduces the number of operations to two (namely, expansion and separated differences).
Since it would not be feasible to introduce all the relevant notions here, for now we will only mention that $\bG^{(\xi)}$ denotes the operation of \emph{expansion}, $\SD_\eta$ denotes the operation of \emph{separated differences}, and $\ell(\bG)$ denotes the \emph{level} of $\bG$ (see Sections \ref{sectionexpansionbasic}, \ref{sectionsdbasic}, and \ref{sectionlevelbasic} respectively).
\begin{theorem}[Louveau]\label{louveaumain}
The collection of all non-selfdual Borel Wadge classes in $\omega^\omega$ is equal to $\Lo$, where $\Lo$ is the smallest collection satisfying the following conditions:
\begin{itemize}
\item $\{\varnothing\}\in\Lo$ and $\{\omega^\omega\}\in\Lo$,
\item $\bG^{(\xi)}\in\Lo$ whenever $\bG\in\Lo$ and $\xi<\omega_1$,
\item $\SD_\eta(\bD,\bG)\in\Lo$, where $\bD=\bigcup_{n\in\omega}(\bL_n\cup\bLc_n)$, whenever $1\leq\eta<\omega_1$, $\bG\in\Lo$ and $\bL_n\in\Lo$ for $n\in\omega$ are such that $\bG\subseteq\bD$ and $\ell(\bL_n)\geq 1$ for each $n$. 
\end{itemize}
\end{theorem}

The final fundamental notion for this article is that of a Hausdorff operation. In Section \ref{sectionhausdorffbasic}, given $D\subseteq\PP(\omega)$, we will show how to simultaneously define a function $\HH_D:\PP(Z)^\omega\longrightarrow\PP(Z)$ for every set $Z$. Functions of this form are known as \emph{Hausdorff operations}. The most basic examples of Hausdorff operation are those obtained by combining the usual set-theoretic operations of union, intersection, and complement (see Proposition \ref{hausdorffsettheoretic}). When $Z$ is a space (as opposed to just a set), we will let
$$
\bG_D(Z)=\{\HH_D(U_0,U_1,\ldots):U_0,U_1,\ldots\in\bS^0_1(Z)\}
$$
denote the class in $Z$ associated to $\HH_D$. Under rather mild assumptions on $Z$, using universal sets, it is not hard to show that each $\bG_D(Z)$ is a non-selfdual Wadge class in $Z$ (see Theorem \ref{addisontheorem}).
In fact, in his doctoral thesis, Robert Van Wesep built on work of Addison, Steel, and Radin to obtain the following result (see \cite[Proposition 5.0.3 and Theorem 5.3.1]{vanwesept}).

\begin{theorem}[Van Wesep]\label{vanwesepmain}
Assume that the Axiom of Determinacy holds. Then the following conditions are equivalent:
\begin{itemize}
\item $\bG$ is a non-selfdual Wadge class in $\omega^\omega$,
\item $\bG=\bG_D(\omega^\omega)$ for some $D\subseteq\PP(\omega)$.
\end{itemize}
\end{theorem}

The purpose of this article is to give a self-contained (except for facts from \cite{kechris}) proof of Theorem \ref{main}, which simultaneously generalizes Theorems \ref{louveaumain} and \ref{vanwesepmain}. There are several ways in which Theorem \ref{main} generalizes the above results. First, the ambient space is an arbitrary uncountable zero-dimensional Polish space instead of $\omega^\omega$. Second, unlike Theorem \ref{louveaumain}, it applies to classes beyond the Borel realm. Third, it gives a level-by-level result, in the sense that to obtain the desired result for classes of a given complexity, only the corresponding determinacy assumption will be required.

In our previous paper \cite{carroymedinimuller} (which has a significant overlap with the present one\footnote{\,In particular, Sections \ref{sectionhausdorffbasic}, \ref{sectionhausdorffclasses}, \ref{sectionhausdorffuniversal}, \ref{sectionexpansionbasic} and \ref{sectionexpansionmain} are almost verbatim the same as the corresponding sections in \cite{carroymedinimuller}.}), we generalized results of van Engelen from Borel spaces to arbitrary spaces (assuming the Axiom of Determinacy). We hope and expect that the results proved here will yield similar applications in the future.

We would like to point out that many of our proofs are essentially the same as those from \cite[Section 7.3]{louveaub}. However, as that is an unpublished manuscript, numerous gaps had to be filled. Most notably, \cite{louveaub} lacks any treatment of relativization (see Sections \ref{sectionrelativizationbasic}, \ref{sectionrelativizationuncountable} and \ref{sectionexpansionrelativization}). Furthermore, for the general case, we will employ ideas of Radin that were not needed in the Borel case (see Section \ref{sectionstretch}).

Finally, we remark that Theorem \ref{main} is in a sense more transparent than Theorem \ref{vanwesepmain}, as it specifies more clearly which Hausdorff operations generate the given Wadge classes. This approach is based once again on unpublished results of Louveau, which are however limited to the Borel realm (see \cite[Corollary 7.3.11 and Theorem 7.3.12]{louveaub}).

\section{Preliminaries and notation}\label{sectionpreliminaries}

Given a function $f:Z\longrightarrow W$, $A\subseteq Z$ and $B\subseteq W$, we will use the notation $f[A]=\{f(x):x\in A\}$ and $f^{-1}[B]=\{x\in Z:f(x)\in B\}$.

\begin{definition}[Wadge]
Let $Z$ be a space, and let $A,B\subseteq Z$. We will write $A\leq B$ if there exists a continuous function $f:Z\longrightarrow Z$ such that $A=f^{-1}[B]$.\footnote{\,Wadge-reduction is usually denoted by $\leq_\mathsf{W}$, which allows to distinguish it from other types of reduction (such as Lipschitz-reduction). Since we will not consider any other type of reduction in this article, we decided to simplify the notation.} In this case, we will say that $A$ is \emph{Wadge-reducible} to $B$, and that $f$ \emph{witnesses} the reduction. We will write $A<B$ if $A\leq B$ and $B\not\leq A$. We will write $A\equiv B$ if $A\leq B$ and $B\leq A$.
\end{definition}

\begin{definition}[Wadge]\label{wadgeclassdefinition}
Let $Z$ be a space. Given $A\subseteq Z$, define
$$
A\wc=\{B\subseteq Z:B\leq A\}.\footnote{\,We point out that $A\wc$ is sometimes denoted by $[A]$ (see for example \cite{carroymedinimuller}, \cite{vanengelenpr}, \cite{vanengelent}, \cite{vanengelena}, and \cite{louveaua}). We decided to avoid this notation, as it conflicts with the notation for the \emph{Wadge degree} of $A$, that is $\{B\subseteq Z:B\equiv A\}$.}
$$
Given $\bG\subseteq\PP(Z)$, we will say that $\bG$ is a \emph{Wadge class} if there exists $A\subseteq Z$ such that $\bG=A\wc$, and that $\bG$ is \emph{continuously closed} if $A\wc\subseteq\bG$ for every $A\in\bG$.
\end{definition}

Both of the above definitions depend of course on the space $Z$. Often, for the sake of clarity, we will specify what the ambient space is by saying, for example, that ``$A\leq B$ in $Z$'' or ``$\bG$ is a Wadge class in $Z$''. We will say that $A\subseteq Z$ is \emph{selfdual} if $A\leq Z\setminus A$ in $Z$. It is easy to check that $A$ is selfdual iff $A\wc$ is selfdual.

Our reference for descriptive set theory is \cite{kechris}. In particular, we assume familiarity with the basic theory of Polish spaces, and their Borel and projective subsets. We use the same notation as in \cite[Section 11]{kechris}. For example, given a space $Z$, the collection of all Borel subsets of $Z$ is denoted by $\Borel(Z)$, while $\mathbf{\Sigma}^0_1(Z)$, $\mathbf{\Pi}^0_1(Z)$ and $\mathbf{\Delta}^0_1(Z)$ denote the collections of all open, closed and clopen subsets of $Z$ respectively. Given spaces $Z$ and $W$, we will say that $j:Z\longrightarrow W$ is an \emph{embedding} if $j:Z\longrightarrow j[Z]$ is a homeomorphism. Recall that the classes $\bS^1_n(Z)$ for $1\leq n<\omega$ can be defined for an arbitrary (that is, not necessarily Polish) space $Z$ by declaring $A\in\bS^1_n(Z)$ if there exist a Polish space $W$ and an embedding $j:Z\longrightarrow W$ such that $j[A]=B\cap j[Z]$ for some $B\in\bS^1_n(W)$.\footnote{\,This is the same definition given in \cite[page 315]{kechris}.} Using the methods of the proof of \cite[Proposition 4.2]{medinizdomskyy}, one can show that $A\in\bS^1_n(Z)$ iff for every Polish space $W$ and embedding $j:Z\longrightarrow W$ there exists $B\in\bS^1_n(W)$ such that $j[A]=B\cap j[Z]$. Our reference for other set-theoretic notions is \cite{jech}.

The classes defined below constitute the so-called \emph{difference hierarchy} (or \emph{small Borel sets}). For a detailed treatment, see \cite[Section 22.E]{kechris} or \cite[Chapter 3]{vanengelent}. Here, we will only mention that the classes $\Diff_\eta(\mathbf{\Sigma}^0_\xi(Z))$ are among the simplest concrete examples of Wadge classes (see Propositions \ref{expansiondifferences} and \ref{expansionhausdorffnonselfdual}).

\begin{definition}[Kuratowski]
Let $Z$ be a space, let $1\leq\eta<\omega_1$, and let $1\leq\xi<\omega_1$. Given a sequence of sets $(A_\mu:\mu<\eta)$, define
$$
\left.
\begin{array}{lcl}
& & \Diff_\eta(A_\mu:\mu<\eta)= \left\{
\begin{array}{ll}
\bigcup\{A_\mu\setminus\bigcup_{\zeta<\mu}A_\zeta:\mu<\eta\text{ and }\mu\text{ is odd}\} & \text{if }\eta\text{ is even,}\\
\bigcup\{A_\mu\setminus\bigcup_{\zeta<\mu}A_\zeta:\mu<\eta\text{ and }\mu\text{ is even}\} & \text{if }\eta\text{ is odd.}
\end{array}
\right.
\end{array}
\right.
$$
Define $\Diff_\eta(\mathbf{\Sigma}^0_\xi(Z))$ by declaring $A\in\Diff_\eta(\mathbf{\Sigma}^0_\xi(Z))$ if there exist $A_\mu\in\mathbf{\Sigma}^0_\xi(Z)$ for $\mu<\eta$ such that $A=\Diff_\eta(A_\mu:\mu<\eta)$.\footnote{\,Notice that requiring that $(A_\mu:\mu<\eta)$ is $\subseteq$-increasing would yield an equivalent definition of $\Diff_\eta(\mathbf{\Sigma}^0_\xi(Z))$.}
\end{definition}

The following two lemmas are useful for proving by induction statements regarding the difference hierarchy. Their straightforward proofs are  mostly left to the reader.

\begin{lemma}\label{diffsuccessor}
Let $Z$ be a space, let $1\leq\eta<\omega_1$, and let $1\leq\xi<\omega_1$. Then the following are equivalent:
\begin{itemize}
\item $A\in\Diff_{\eta+1}(\bS^0_\xi(Z))$,
\item There exist $B\in\bS^0_\xi(Z)$ and $C\in\Diff_\eta(\bS^0_\xi(Z))$ such that $C\subseteq B$ and $A=B\setminus C$.
\end{itemize}
\end{lemma}
\begin{proof}
Proceed by induction on $\eta$.	
\end{proof}

\begin{lemma}\label{difflimit}
Let $Z$ be a space, let $1<\xi<\omega_1$, and let $\eta<\omega_1$ be a limit ordinal. Then the following are equivalent:
\begin{itemize}
\item $A\in\Diff_\eta(\bS^0_\xi(Z))$,
\item There exist $A_n\in\bigcup_{\eta'<\eta}\Diff_{\eta'}(\bS^0_\xi(Z))$ and pairwise disjoint $V_n\in\bS^0_\xi(Z)$ for $n\in\omega$ such that $A=\bigcup_{n\in\omega}A_n\cap V_n$.
\end{itemize}
Furthermore, if $Z$ is zero-dimensional then the result holds for $\xi=1$ as well.
\end{lemma}
\begin{proof}
Use \cite[Theorem 22.16]{kechris}.
\end{proof}

For an introduction to the topic of games, we refer the reader to \cite[Section 20]{kechris}. Here, we only want to give the precise definition of determinacy. Given a set $A$, a \emph{play} of the game $\Ga(A,X)$ is described by the diagram
\begin{center}
\begin{tabular}{c|ccccl}
I & $a_0$ & & $a_2$ & & $\cdots$\\\hline
II & & $a_1$ & & $a_3$ & $\cdots$,
\end{tabular}
\end{center}
in which $a_n\in A$ for every $n\in\omega$ and $X\subseteq A^\omega$ is called the \emph{payoff set}. We will say that Player I \emph{wins} this play of the game $\Ga(A,X)$ if $(a_0,a_1,\ldots)\in X$. Player II \emph{wins} if Player I does not win.

A \emph{strategy} for a player is a function $\sigma:A^{<\omega}\longrightarrow A$. We will say that $\sigma$ is a \emph{winning strategy} for Player I if setting $a_{2n}=\sigma(a_1,a_3,\ldots,a_{2n-1})$ for each $n$ makes Player I win for every $(a_1,a_3,\ldots)\in A^\omega$. A winning strategy for Player II is defined similarly. We will say that the game $\Ga(A,X)$ (or simply the set $X$) is \emph{determined} if (exactly) one of the players has a winning strategy. In this article, we will exclusively deal with the case $A=\omega$. Given $\bS\subseteq\PP(\omega^\omega)$, we will write $\Det(\bS)$ to mean that every element of $\bS$ is determined. The assumption $\Det(\PP(\omega^\omega))$ is known as the \emph{Axiom of Determinacy} (briefly, $\AD$).\footnote{\,Quite amusingly, Van Wesep referred to $\AD$ as a ``frankly heretical postulate'' (see \cite[page 64]{vanwesept}), while Steel deemed it ``probably false'' (see \cite[page 63]{steel}).} The assumption $\Det(\bigcup_{1\leq n<\omega}\bS^1_n(\omega^\omega))$ is known as the axiom of \emph{Projective Determinacy}.

It is well-known that $\AD$ is incompatible with the Axiom of Choice (see \cite[Lemma 33.1]{jech}). This is the reason why, throughout this article, we will be working in $\ZF+\DC$. Recall that the principle of \emph{Dependent Choices} (briefly, $\DC$) states that if $R$ is a binary relation on a non-empty set $A$ such that for every $a\in A$ there exists $b\in A$ such that $(b,a)\in R$, then there exists a sequence $(a_0, a_1,\ldots)\in A^\omega$ such that $(a_{n+1},a_n)\in R$ for every $n\in\omega$. This principle is what is needed to carry out recursive constructions of length $\omega$. Another consequence (in fact, an equivalent formulation) of $\DC$ is that a relation $R$ on a set $A$ is well-founded iff there exists no sequence $(a_0, a_1,\ldots)\in A^\omega$ such that $(a_{n+1},a_n)\in R$ for every $n\in\omega$ (see \cite[Lemma 5.5.ii]{jech}). Furthermore, $\DC$ implies the Countable Axiom of Choice (see \cite[Exercise 5.7]{jech}). To the reader who is unsettled by the lack of the full Axiom of Choice, we recommend \cite{howardrubin}.

It is a theorem of Martin that $\Det(\Borel(\omega^\omega))$ holds in $\ZF+\DC$ (this was originally proved in \cite{martina}, but see also \cite[Remark (2) on page 307]{martinb}). On the other hand, Harrington showed that $\Det(\bS^1_1(\omega^\omega))$ has large cardinal strength (see \cite{harrington}). For the consistency of $\ZF+\DC+\AD$, see \cite{neeman} and \cite[Proposition 11.13]{kanamori}.

We conclude this section with more notation and well-known definitions, for the sake of clarity. A \emph{partition} of a set $Z$ is a collection $\VV\subseteq\PP(Z)$ consisting of pairwise disjoint non-empty sets such that $\bigcup\VV=Z$. We will denote by $\id_Z:Z\longrightarrow Z$ the identity function on a set $Z$. Given a set $A$, we will denote by $A^{<\omega}$ the collection of all functions $s:n\longrightarrow A$, where $n\in\omega$. Given $s\in A^{<\omega}$, we will use the notation $\Ne_s=\{z\in A^\omega:s\subseteq z\}$.\footnote{\,In all our applications, we will have $A=2$ or $A=\omega$.} Given a set $Z$ and $\bS\subseteq\PP(Z)$, we will denote by $\bool\bS$ the smallest subset of $\PP(Z)$ that contains $\bS$ and is closed under complements and finite intersections.

A subset of a space is \emph{clopen} if it is closed and open. A \emph{base} for a space $Z$ is a collection $\UU\subseteq\bS^0_1(Z)$ consisting of non-empty sets such that for every $x\in Z$ and every $U\in\bS^0_1(Z)$ containing $x$ there exists $V\in\UU$ such that $x\in V\subseteq U$. A space is \emph{zero-dimensional} if it is non-empty\footnote{\,The empty space has dimension $-1$ (see \cite[Section 7.1]{engelking}).} and it has a base consisting of clopen sets. A space $Z$ is a \emph{Borel space} if there exists a Polish space $W$ and an embedding $j:Z\longrightarrow W$ such that $j[Z]\in\Borel(W)$. By proceeding as in the proof of \cite[Proposition 4.2]{medinizdomskyy}, it is easy to show that a space $Z$ is Borel iff $j[Z]\in\Borel(W)$ for every Polish space $W$ and every embedding $j:Z\longrightarrow W$. For example, by \cite[Theorem 3.11]{kechris}, every Polish space is a Borel space.

Given $1\leq\xi<\omega_1$ and spaces $Z$ and $W$, a function $f:Z\longrightarrow W$ is \emph{$\bS^0_\xi$-measurable} if $f^{-1}[U]\in\bS^0_\xi(Z)$ for every $U\in\bS^0_1(W)$. A function $f:Z\longrightarrow W$ is \emph{Borel} if $f^{-1}[U]\in\Borel(Z)$ for every $U\in\bS^0_1(W)$. Using the existence of a countable base, it is easy to see that a function is Borel iff it is $\bS^0_{1+\xi}$-measurable for some $\xi<\omega_1$.

\section{Nice topological pointclasses}\label{sectionpointclasses}

In this section we will consider the natural concept of a topological pointclass, and then define a strengthening of it that will be convenient for technical reasons (without resulting in any loss of generality for our intended applications). It is in terms of these classes that our determinacy assumptions will be stated. In fact, the typical result in this article will begin by assuming $\Det(\bS(\omega^\omega))$, where $\bS$ is a suitable topological pointclass.

Notice that the term ``function'' in the following definition is an abuse of terminology, as each topological pointclass is a proper class. Therefore, every theorem in this paper that mentions these pointclasses is strictly speaking an infinite scheme (one theorem for each suitable topological pointclass). In fact, as we will make clear in the remainder of this section, topological pointclasses are simply a convenient expositional tool that will allow us to simultaneously state the Borel, Projective, and full-Determinacy versions of our results.

\begin{definition}
We will say that a function $\bS$ is a \emph{topological pointclass} if it satisfies the following requirements:
\begin{itemize}
\item The domain of $\bS$ is the class of all spaces,\footnote{\,Recall from Section \ref{sectionintroduction} that we are only considering separable metrizable spaces.}
\item $\bS(Z)\subseteq\PP(Z)$ for every space $Z$,
\item If $f:Z\longrightarrow W$ is a continuous function and $B\in\bS(W)$ then $f^{-1}[B]\in\bS(Z)$.
\end{itemize}
Furthermore, we will say that a topological pointclass $\bS$ is \emph{nice} if it satisfies the following additional properties:
\begin{enumerate}
\item\label{niceboolean} $\bool\bS(Z)=\bS(Z)$ for every space $Z$,
\item\label{niceborel} $\Borel(Z)\subseteq\bS(Z)$ for every space $Z$,
\item\label{nicepreimages} If $f:Z\longrightarrow W$ is a Borel function and $B\in\bS(W)$ then $f^{-1}[B]\in\bS(Z)$,
\item\label{niceabsoluteness} For every space $Z$, if $j[Z]\in\bS(W)$ for some Borel space $W$ and embedding $j:Z\longrightarrow W$, then $j[Z]\in\bS(W)$ for every Borel space $W$ and embedding $j:Z\longrightarrow W$.
\end{enumerate}
\end{definition}

Condition $(\ref{niceboolean})$ is mostly due to the complexity of the payoff set in the proof of Lemma \ref{strongwadgelemma}, but it also ensures other useful closure properties, especially in conjunction with condition $(\ref{niceborel})$. Condition $(\ref{nicepreimages})$ ensures that $\bS$ is suitably closed under expansions, in the terminology of Definition \ref{definitionexpansion}. Furthermore, as in the proof of the implication $(\ref{uncountablelevelpreimages})\rightarrow (\ref{uncountablelevelomega_1})$ of Corollary \ref{uncountablelevel}, it is easy to see that condition $(\ref{nicepreimages})$ implies the following:
\begin{itemize}
\item[($3'$)] For every space $Z$ and $A\in\bS(Z)$, if $V_n\in\bD^0_2(Z)$ and $A_n\leq A$ for $n\in\omega$, and the $V_n$ are pairwise disjoint, then $\bigcup_{n\in\omega}(A_n\cap V_n)\in\bS(Z)$.
\end{itemize}
Condition $(3')$ will be tacitly used in Section \ref{sectionsdmain} (see Claims 5 and 10 in the proof of Theorem \ref{sdmain}, and Lemma \ref{l0impliessomewheresmaller}), as it ensures that $\bS$ is suitably closed under $\PU_1$, in the notation of Definition \ref{pudefinition}. 

Condition $(\ref{niceabsoluteness})$ encapsulates the appropriate degree of ``topological absoluteness'' for spaces of complexity $\bS$, and it will be used exclusively in the proof of Lemma \ref{relativizationprelim}. We remark that our focus on Borel spaces is due to the fact that we will need a certain portion of the machinery of relativization to work for these spaces (see Section \ref{sectionrelativizationbasic}, in particular Footnote 12).

For the purposes of this paper, the following are the intended examples of nice topological pointclasses (this can be verified using \cite[Exercise 37.3]{kechris} and the methods of \cite[Section 4]{medinizdomskyy}):
\begin{itemize}
 \item[(A)] $\bS(Z)=\Borel(Z)$ for every space $Z$,
 \item[(B)] $\bS(Z)=\bool\bS^1_n(Z)$ for every space $Z$, where $1\leq n<\omega$,
 \item[(C)] $\bS(Z)=\bigcup_{1\leq n<\omega}\bS^1_n(Z)$ for every space $Z$,
 \item[(D)] $\bS(Z)=\PP(Z)$ for every space $Z$.
\end{itemize}
Regarding example (B), we remark that $\Det(\bool\bS^1_n(\omega^\omega))$ is equivalent to $\Det(\bS^1_n(\omega^\omega))$ whenever $1\leq n<\omega$ (this easily follows from \cite[Corollary 4.1]{mullerschindlerwoodin}).

We conclude with two well-known results, which clarify the relationship between determinacy assumptions and the Baire property in Polish spaces. Given $\bS\subseteq\PP(\omega^\omega)$, we will write $\BP(\bS)$ to mean that every element of $\bS$ has the Baire property in $\omega^\omega$.

\begin{theorem}\label{dettobp}
Let $\bS$ be a nice topological pointclass. If $\Det(\bS(\omega^\omega))$ holds then $\BP(\bS(\omega^\omega))$ holds.
\end{theorem}
\begin{proof}
Use the methods of \cite[Section 8.H]{kechris}.
\end{proof}

\begin{proposition}\label{bpbairetoall}
Let $\bS$ be a topological pointclass, and assume that $\BP(\bS(\omega^\omega))$ holds. Let $Z$ be a Polish space, and let $A\in\bS(Z)$. Then $A$ has the Baire property in $Z$.
\end{proposition}
\begin{proof}
Use the fact that, if $Z$ is non-empty, then there exists an open continuous surjection $f:\omega^\omega\longrightarrow Z$ (see \cite[Exercise 7.14]{kechris}).
\end{proof}

\section{The basics of Wadge theory}\label{sectionwadgebasic}

We begin by introducing a special notation for the collection of all non-selfdual Wadge classes in a given space. Throughout the paper, starting from the discussion at the end of this section and culminating with Theorem \ref{main}, it will become increasingly clear that these are the most important Wadge classes.

\begin{definition}
Given a space $Z$, define 
$$
\NSD(Z)=\{\bG:\bG\textrm{ is a non-selfdual Wadge class in }Z\}.
$$
Also set $\NSDS(Z)=\{\bG\in\NSD(Z):\bG\subseteq\bS(Z)\}$ whenever $\bS$ is a topological pointclass.
\end{definition}

The following simple lemma will allow us to generalize many Wadge-theoretic results from $\omega^\omega$ to an arbitrary zero-dimensional Polish space. This approach has already appeared in \cite[Section 5]{andretta}, where it is credited to Marcone. Recall that, given a space $Z$ and $W\subseteq Z$, a \emph{retraction} is a continuous function $\rho:Z\longrightarrow W$ such that $\rho\re W=\id_W$. By \cite[Theorem 7.8]{kechris}, every zero-dimensional Polish space is homeomorphic to a closed subspace $Z$ of $\omega^\omega$, and by \cite[Proposition 2.8]{kechris} there exists a retraction $\rho:\omega^\omega\longrightarrow Z$.

\begin{lemma}\label{bairetoall}
Let $Z\subseteq\omega^\omega$, and let $\rho:\omega^\omega\longrightarrow Z$ be a retraction. Fix $A,B\subseteq Z$. Then $A\leq B$ in $Z$ iff $\rho^{-1}[A]\leq\rho^{-1}[B]$ in $\omega^\omega$. 
\end{lemma}
\begin{proof}
If $f:Z\longrightarrow Z$ witnesses that $A\leq B$ in $Z$, then $f\circ\rho:\omega^\omega\longrightarrow\omega^\omega$ will witness that $\rho^{-1}[A]\leq\rho^{-1}[B]$ in $\omega^\omega$. On the other hand, if $f:\omega^\omega\longrightarrow\omega^\omega$ witnesses that $\rho^{-1}[A]\leq\rho^{-1}[B]$ in $\omega^\omega$, then $\rho\circ(f\re Z):Z\longrightarrow Z$ will witness that $A\leq B$ in $Z$.
\end{proof}

The most fundamental result of Wadge theory is Lemma \ref{wadgelemma} (commonly known as ``Wadge's Lemma''). Among other things, it shows that antichains with respect to $\leq$ have size at most $2$. However, instead of proving it directly, we will deduce it from the following lemma, which is essentially due to Louveau and Saint-Raymond (see \cite[Theorem 4.1.b]{louveausaintraymondc}).
Lemma \ref{strongwadgelemma} will also be a crucial tool in Section \ref{sectionrelativizationbasic}.

The following ``Extended Wadge game'' was also introduced by Louveau and Saint-Raymond (see \cite[Section 3]{louveausaintraymondc}), and it will be used in the proof of Lemma \ref{strongwadgelemma}. Given $D,A_0,A_1\subseteq\omega^\omega$, consider the game $\EW(D,A_0,A_1)$ described by the following diagram
\begin{center}
\begin{tabular}{c|ccccl}
I & $x_0$ & & $x_1$ & & $\cdots$\\\hline
II & & $y_0$ & & $y_1$ & $\cdots$,
\end{tabular}
\end{center}
where $x=(x_0,x_1,\ldots)\in\omega^\omega$, $y=(y_0,y_1,\ldots)\in\omega^\omega$, and Player II wins if one of the following conditions is verified:
\begin{itemize}
\item $x\in D$ and $y\in A_0$,
\item $x\notin D$ and $y\in A_1$.
\end{itemize}

\begin{lemma}\label{strongwadgelemma}
Let $\bS$ be a topological pointclass, and assume that $\Det(\bool\bS(\omega^\omega))$ holds. Let $\bG\subseteq\bool\bS(\omega^\omega)$ be continuously closed, and let  $A_0,A_1\in\bool\bS(\omega^\omega)$ be such that $A_0\cap A_1=\varnothing$. Then, one of the following conditions holds:
\begin{enumerate}
\item\label{strongwadgeexists} There exists $C\in\bG$ such that $A_0\subseteq C$ and $C\cap A_1=\varnothing$,
\item\label{strongwadgeforall} For all $D\in\bGc$ there exists a continuous $f:\omega^\omega\longrightarrow A_0\cup A_1$ such that $f^{-1}[A_0]=D$.
\end{enumerate}
\end{lemma}
\begin{proof}
Assume that condition $(\ref{strongwadgeforall})$ fails. We will show that condition $(\ref{strongwadgeexists})$ holds. Fix $D\in\bGc$ such that $f^{-1}[A_0]\neq D$ for every continuous $f:\omega^\omega\longrightarrow A_0\cup A_1$. First we claim that Player II does not have a winning strategy in the game $\EW(D,A_0,A_1)$. Assume, in order to get a contradiction, that there exists a winning strategy $\sigma$ for Player II. Given $x\in\omega^\omega$, view $x$ as describing the moves of Player I, then define $f(x)=y$, where $y$ is the response of Player II to $x$ according to the strategy $\sigma$. It is easy to realize that $f$ contradicts the assumption at the beginning of this proof.

Since the payoff set of the game $\EW(D,A_0,A_1)$ belongs to $\bool\bS(\omega^\omega)$, the assumption of $\Det(\bool\bS(\omega^\omega))$ guarantees the existence of a winning strategy $\tau$ for Player I. Given $y\in\omega^\omega$, view $y$ as describing the moves of Player II, then define $g(y)=x$, where $x$ is the response of Player I to $y$ according to the strategy $\tau$.

Set $C=g^{-1}[\omega^\omega\setminus D]$, and observe that $C\in\bG$ because $\bG$ is continuously closed. Notice that, since $\tau$ is a winning strategy for Player I, for every $y\in A_0\cup A_1$ neither of the following conditions holds:
\begin{itemize}
\item $g(y)\in D$ and $y\in A_0$,
\item $g(y)\notin D$ and $y\in A_1$.
\end{itemize}
Using this observation, one sees that $A_0\subseteq C$ and $C\cap A_1=\varnothing$.
\end{proof}

\begin{lemma}[Wadge]\label{wadgelemma}
Let $\bS$ be a topological pointclass, and assume that $\Det(\bool\bS(\omega^\omega))$ holds. Let $Z$ be a zero-dimensional Polish space, and let $A,B\in\bool\bS(Z)$. Then either $A\leq B$ or $Z\setminus B\leq A$.
\end{lemma}
\begin{proof}
For the case $Z=\omega^\omega$, apply Lemma \ref{strongwadgelemma} with $A_0=A$, $A_1=\omega^\omega\setminus A$, and $\bG=B\wc$. To obtain the full result from this particular case, use Lemma \ref{bairetoall} and the remarks preceding it.
\end{proof}

The following two results are simple applications of Wadge's Lemma, whose proofs are left to the reader.

\begin{lemma}\label{nonselfdualcontinuouslyclosed}
Let $\bS$ be a topological pointclass, and assume that $\Det(\bool\bS(\omega^\omega))$ holds. Let $Z$ be a zero-dimensional Polish space, and let $\bG\subseteq\bool\bS(Z)$. If $\bG$ is continuously closed and non-selfdual then $\bG$ is a Wadge class.
\end{lemma}

\begin{lemma}\label{selfdualwadgelemma}
Let $\bS$ be a topological pointclass, and assume that $\Det(\bool\bS(\omega^\omega))$ holds. Let $Z$ be a zero-dimensional Polish space, let $\bG\subseteq\bool\bS(Z)$ be a non-selfdual Wadge class, and let $\bD\subseteq\bool\bS(Z)$ be continuously closed and selfdual. If $\bG\nsubseteq\bD$ then $\bD\subsetneq\bG$.
\end{lemma}

The following is the second most fundamental theorem of Wadge theory after Wadge's Lemma. In fact, it is at the core of many proofs of important Wadge-theoretic results.

\begin{theorem}[Martin, Monk]\label{wellfounded}
Let $\bS$ be a nice topological pointclass, and assume that $\Det(\bS(\omega^\omega))$ holds. Let $Z$ be a zero-dimensional Polish space. Then the relation $\leq$ on $\bS(Z)$ is well-founded.
\end{theorem}
\begin{proof}
For the case $Z=\omega^\omega$, proceed as in \cite[proof of Theorem 21.15]{kechris}, using Theorem \ref{dettobp} and Proposition \ref{bpbairetoall}. To obtain the full result from this particular case, use Lemma \ref{bairetoall} and the remarks preceding it.
\end{proof}

Next, we state an elementary result, which shows that clopen sets are ``neutral sets'' for Wadge-reduction. By this we mean that, apart from trivial exceptions, intersections or unions with these sets do not change the Wadge class. The straightforward proof is left to the reader. For more sophisticated closure properties, see \cite[Section 12]{carroymedinimuller}.

\begin{lemma}\label{closureclopen}
Let $Z$ be a space, let $\bG$ be a Wadge class in $Z$, and let $A\in\bG$.
\begin{itemize}
\item Assume that $\bG\neq\{Z\}$. Then $A\cap V\in\bG$ for every $V\in\mathbf{\Delta}^0_1(Z)$.
\item Assume that $\bG\neq\{\varnothing\}$. Then $A\cup V\in\bG$ for every $V\in\mathbf{\Delta}^0_1(Z)$.
\end{itemize}
\end{lemma}

We conclude this section with some basic facts that will not be needed in the rest of the paper, but hopefully will help the reader in understanding how Wadge classes behave. In order to simplify the discussion, we will assume that $\AD$ holds until the end of this section. Given a zero-dimensional Polish space $Z$, define
$$
\Wa(Z)=\{\{\bG,\bGc\}:\bG\text{ is a Wadge class in }Z\}.
$$
Given $p,q\in\Wa(Z)$, define $p\prec q$ if $\bG\subseteq\bL$ for every $\bG\in p$ and $\bL\in q$. Using Lemma \ref{wadgelemma} and Theorem \ref{wellfounded}, one sees that the ordering $\prec$ on $\Wa(Z)$ is a well-order. Therefore, there exists an order-isomorphism $\phi:\Wa(Z)\longrightarrow\Theta$ for some ordinal $\Theta$.\footnote{\,For a characterization of $\Theta$, see \cite[Definition 0.1 and Lemma 0.2]{solovay}.} The reason for the ``1+'' in the definition below is simply a matter of technical convenience (see \cite[page 45]{andrettahjorthneeman}).
\begin{definition}
Let $Z$ be a zero-dimensional Polish space, and let $\bG$ be a Wadge class in $Z$. Define
$$
||\bG||=1+\phi(\{\bG,\bGc\}).
$$
We will say that $||\bG||$ is the \emph{Wadge-rank} of $\bG$.
\end{definition}

It is easy to check that $\{\{\varnothing\},\{Z\}\}$ is the minimal element of $\Wa(Z)$. Furthermore, elements of the form $\{\bG,\bGc\}$ for $\bG\in\NSD(Z)$ are always followed by $\{\bD\}$ for some selfdual Wadge class $\bD$ in $Z$, while elements of the form $\{\bD\}$ for some selfdual Wadge class $\bD$ in $Z$ are always followed by $\{\bG,\bGc\}$ for some $\bG\in\NSD(Z)$. This was proved by Van Wesep for $Z=\omega^\omega$ (see \cite[Corollary to Theorem 2.1]{vanwesept}), and it can be generalized to arbitrary uncountable zero-dimensional Polish spaces using Corollary \ref{selfdualcorollary} and the machinery of relativization that we will develop in Sections \ref{sectionrelativizationbasic} and \ref{sectionrelativizationuncountable}. Since these facts will not be needed in the remainder of the paper, we omit their proofs.

In fact, as Theorem \ref{orderisomorphism} will show, the ordering of the non-selfdual classes is independent of the space $Z$ (as long as it is uncountable, zero-dimensional, and Borel). However, the situation is more delicate for selfdual classes. For example, it follows easily from Corollary \ref{selfdualcorollary} that if $\bG$ is a Wadge class in $2^\omega$ such that $||\bG||$ is a limit ordinal of countable cofinality, then $\bG$ is non-selfdual. On the other hand, if $\bG$ is a Wadge class in $\omega^\omega$ such that $||\bG||$ is a limit ordinal of countable cofinality, then $\bG$ is selfdual (see \cite[Corollary to Theorem 2.1]{vanwesept} again).

\section{The analysis of selfdual sets}\label{sectionselfdual}

The aim of this section is to show that a set is selfdual iff it can be constructed in a certain way using sets of lower complexity. The easy implication is given by Proposition \ref{locallylowerimpliesselfdual}, while the hard implication can be obtained by applying Corollary \ref{selfdualcorollary} with $U=Z$. These are well-known results (see for example \cite[Lemmas 7.3.1.iv and 7.3.4]{louveaub}). Our approach is essentially the same as the one used in the proof of \cite[Theorem 16]{andrettamartin} or in \cite[Theorem 5.3]{mottoros}. However, since we would like our paper to be self-contained, and the proof becomes slightly simpler in our context, we give all the details below.

\begin{proposition}\label{locallylowerimpliesselfdual}
Let $\bS$ be a nice topological pointclass, and assume that $\Det(\bS(\omega^\omega))$ holds. Let $Z$ be a zero-dimensional Polish space, and let $A\in\bS(Z)$. Assume that $\UU$ is an open cover of $Z$ such that $A\cap U<A$ in $Z$ for every $U\in\UU$. Then $A$ is selfdual.
\end{proposition}
\begin{proof}
First notice that $\UU\neq\varnothing$ because $Z\neq\varnothing$. It follows that $A\neq\varnothing$ and $A\neq Z$. In particular, $A\cap U\neq Z$ for every $U\in\UU$. So, by Lemma \ref{closureclopen}, $A\cap V\leq A\cap U<A$ whenever $V\in\bD^0_1(Z)$ and $V\subseteq U\in\UU$. Therefore, since $Z$ is zero-dimensional, we can assume without loss of generality that $\UU$ is a disjoint clopen cover of $Z$.

We claim that $U\setminus A\leq A$ for every $U\in\UU$. Pick $U\in\UU$. If $A\cap U=\varnothing$ then $U\setminus A=U\in\bD^0_1(Z)$, hence the claim holds because $A\neq\varnothing$ and $A\neq Z$. On the other hand, if $A\cap U\neq\varnothing$ then
$$
U\setminus A=(Z\setminus (A\cap U))\cap U\leq Z\setminus (A\cap U)\leq A,
$$
where the first reduction holds by Lemma \ref{closureclopen} and the second reduction follows from $A\nleq A\cap U$ using Lemma \ref{wadgelemma}. In conclusion, we can fix $f_U:Z\longrightarrow Z$ witnessing that $U\setminus A\leq A$ in $Z$ for every $U\in\UU$. It is clear that $f=\bigcup\{f_U\re U:U\in\UU\}$ will witness that $Z\setminus A\leq A$.
\end{proof}

Given a space $Z$ and $A\subseteq Z$, define
\begin{multline}
\II(A)=\{V\in\mathbf{\Delta}^0_1(Z):\text{ there exists a partition }\UU\subseteq\mathbf{\Delta}^0_1(V)\text{ of }V\\\nonumber
\text{ such that }U\cap A< A\text{ in }Z\text{ for every }U\in\UU\}.
\end{multline}
Notice that $\II(A)$ is \emph{$\sigma$-additive}, in the sense that if $V_n\in\II(A)$ for $n\in\omega$ and $V=\bigcup_{n\in\omega}V_n\in\mathbf{\Delta}^0_1(Z)$, then $V\in\II(A)$.

We begin with two simple preliminary results. Recall that $F\subseteq 2^\omega$ is a \emph{flip-set} if whenever $z,w\in 2^\omega$ are such that $|\{n\in\omega:z(n)\neq w(n)\}|=1$ then $z\in F$ iff $w\notin F$.

\begin{lemma}\label{flipsetlemma}
Let $F\subseteq 2^\omega$ be a flip-set. Then $F$ does not have the Baire property.
\end{lemma}
\begin{proof}
Assume, in order to get a contradiction, that $F$ has the Baire property. Since $2^\omega\setminus F$ is also a flip-set, we can assume without loss of generality that $F$ is non-meager in $2^\omega$. By \cite[Proposition 8.26]{kechris}, we can fix $n\in\omega$ and $s\in 2^n$ such that $F\cap\Ne_s$ is comeager in $\Ne_s$. Fix $k\in\omega\setminus n$ and let $h:\Ne_s\longrightarrow \Ne_s$ be the homeomorphism defined by
$$
h(x)(i)=\left\{
\begin{array}{ll} x(i) & \textrm{if }i\neq k,\\
1-x(i) & \textrm{if }i=k
\end{array}
\right.
$$
for $x\in\Ne_s$ and $i\in\omega$. Observe that $(\Ne_s\cap F)\cap h[\Ne_s\cap F]$ is comeager in $\Ne_s$, hence it is non-empty. It is easy to realize that this contradicts the definition of flip-set.
\end{proof}

\begin{lemma}\label{selfduallemma}
Let $Z$ be a space, and let $A\subseteq Z$ be a selfdual set such that $A\notin\mathbf{\Delta}^0_1(Z)$. Assume that $V\in\mathbf{\Delta}^0_1(Z)$ and $V\notin\II(A)$. Then $V\cap A\leq V\setminus A$ in $V$.
\end{lemma}
\begin{proof}
Using Lemma \ref{closureclopen}, one sees that $V\cap A\leq A$ and $V\setminus A\leq Z\setminus A$, where both reductions are in $Z$. On the other hand, since $V\cap A<A$ would contradict the assumption that $V\notin\II(A)$, we see that $V\cap A\equiv A$. It follows that $V\setminus A\leq Z\setminus A\equiv A\equiv V\cap A$. Let $f:Z\longrightarrow Z$ be a function witnessing that $V\setminus A\leq V\cap A$. Notice that $V\setminus A\neq\varnothing$, otherwise we would have $V=V\cap A\equiv A$, contradicting the assumption that $A\notin\mathbf{\Delta}^0_1(Z)$. So we can fix $z\in V\setminus A$, and define $g:Z\longrightarrow V$ by setting
$$
g(x)=\left\{
\begin{array}{ll}x & \textrm{if }x\in V,\\
z & \textrm{if }x\in Z\setminus V.
\end{array}
\right.
$$
Since $V\in\mathbf{\Delta}^0_1(Z)$, the function $g$ is continuous. Finally, it is straightforward to verify that $g\circ(f\re V):V\longrightarrow V$ witnesses that $V\cap A\leq V\setminus A$ in $V$.
\end{proof}

\begin{theorem}\label{selfdualtheorem}
Let $\bS$ be a topological pointclass, and assume that $\BP(\bS(\omega^\omega))$ holds. Let $Z$ be a zero-dimensional Polish space, and let $A\in\bS(Z)$ be selfdual. Assume that $A\notin\mathbf{\Delta}^0_1(Z)$. Then $\mathbf{\Delta}^0_1(Z)=\II(A)$.
\end{theorem}
\begin{proof}
Assume, in order to get a contradiction, that $V\in\mathbf{\Delta}^0_1(Z)\setminus\II(A)$. Fix a complete metric on $Z$ that induces the given Polish topology. We will recursively construct sets $V_n$ and functions $f_n:V_n\longrightarrow V_n$ for $n\in\omega$. Before specifying which properties we require from them, we introduce some more notation. Given a set $X$ and a function $f:X\longrightarrow X$, set $f^0=\id_X$ and $f^1=f$. Furthermore, given $m,n\in\omega$ such that $m\leq n$ and $z\in 2^\omega$ (or just $z\in 2^{[m,n]}$), define
$$
f_{[m,n]}^z=f_m^{z(m)}\circ\cdots\circ f_n^{z(n)}.
$$
We will make sure that the following conditions are satisfied for every $n\in\omega$, where $\diam(X)$ denotes the diameter of $X\subseteq Z$:
\begin{enumerate}
\item $V_n\in\mathbf{\Delta}^0_1(Z)$,
\item\label{ideal} $V_n\notin\II(A)$,
\item $V_m\supseteq V_n$ whenever $m\leq n$,
\item\label{reduction} $f_n:V_n\longrightarrow V_n$ witnesses that $V_n\cap A\leq V_n\setminus A$ in $V_n$,
\item\label{smalldiam} $\diam(f_{[m,n]}^s[V_{n+1}])\leq 2^{-n}$ whenever $m\leq n$ and $s\in 2^{[m,n]}$.
\end{enumerate}

Start by setting $V_0=V$ and let $f_0:V_0\longrightarrow V_0$ be given by Lemma \ref{selfduallemma}. Now fix $n\in\omega$, and assume that $V_m$ and $f_m$ have already been constructed for every $m\leq n$. Fix a partition $\UU$ of $Z$ consisting of clopen sets of diameter at most $2^{-n}$. Given $m\leq n$ and $s\in 2^{[m,n]}$, define
$$
\VV^s_m=\{(f_{[m,n]}^s)^{-1}[U\cap V_m]:U\in\UU\}.
$$
Observe that each $\VV^s_m\subseteq\mathbf{\Delta}^0_1(V_n)$ because each $f_{[m,n]}^s$ is continuous. Furthermore, it is clear that each $\VV^s_m$ consists of pairwise disjoint sets, and that $\bigcup\VV^s_m=V_n$. Since there are only finitely many $m\leq n$ and $s\in 2^{[m,n]}$, it is possible to obtain a partition $\VV\subseteq\mathbf{\Delta}^0_1(V_n)$ of $V_n$ that simultaneously refines each $\VV^s_m$. This clearly implies that any choice of $V_{n+1}\in\VV$ will satisfy condition $(\ref{smalldiam})$. On the other hand since $\II(A)$ is $\sigma$-additive and $V_n\notin\II(A)$, it is possible to choose $V_{n+1}\in\VV$ such that $V_{n+1}\notin\II(A)$, thus ensuring that condition $(\ref{ideal})$ is satisfied as well. To obtain $f_{n+1}:V_{n+1}\longrightarrow V_{n+1}$ that satisfies condition $(\ref{reduction})$, simply apply Lemma \ref{selfduallemma}. This concludes the construction.

Fix an arbitrary $y_{n+1}\in V_{n+1}$ for $n\in\omega$. Given $m\in\omega$ and $z\in 2^\omega$, observe that the sequence $(f_{[m,n]}^z(y_{n+1}):m\leq n)$ is Cauchy by condition $(\ref{smalldiam})$, hence it makes sense to define
$$
x_m^z=\underset{n\to\infty}{\limi}f_{[m,n]}^z(y_{n+1}).
$$
To conclude the proof, we will show that $F=\{z\in 2^\omega:x_0^z\in A\}$ is a flip-set. Since the function $g:2^\omega\longrightarrow Z$ defined by setting $g(z)=x_0^z$ is continuous and $A\in\bS(Z)$, Proposition \ref{bpbairetoall} and Lemma \ref{flipsetlemma} will easily yield a contradiction.

Define $A^0=A$ and $A^1=Z\setminus A$. Given any $m\in\omega$ and $\varepsilon\in 2$, it is clear from the definition of $f_m^\varepsilon$ and condition $(\ref{reduction})$ that
$$
x\in A\text{ iff }f_m^\varepsilon(x)\in A^\varepsilon
$$
for every $x\in V_m$. Furthermore, using the continuity of $f_m^\varepsilon$ and the definition of $x_m^z$, it is easy to see that
$$
f_m^{z(m)}(x_{m+1}^z)=x_m^z
$$
for every $z\in 2^\omega$ and $m\in\omega$.

Fix $z\in 2^\omega$ and notice that, by the observations in the previous paragraph,
$$
x_0^z\in A\text{ iff } x_1^z\in A^{z(0)}\text{ iff }\cdots\text{ iff }x_{m+1}^z\in (\cdots(A^{z(0)})^{z(1)}\cdots)^{z(m)}
$$
for every $m\in\omega$. Now fix $w\in 2^\omega$ and $m\in\omega$ such that $z\re\omega\setminus\{m\}=w\re\omega\setminus\{m\}$ and $z(m)\neq w(m)$. We need to show that $x_0^z\in A$ iff $x_0^w\notin A$. For exactly the same reason as above, we have
$$
x_0^w\in A\text{ iff } x_1^w\in A^{w(0)}\text{ iff }\cdots\text{ iff }x_{m+1}^w\in (\cdots(A^{w(0)})^{w(1)}\cdots)^{w(m)}.
$$
Since $z\re m=w\re m$ and $z(m)\neq w(m)$, in order to finish the proof, it will be enough to show that $x_{m+1}^z=x_{m+1}^w$. To see this, observe that
$$
x_{m+1}^z=\underset{n\to\infty}{\limi}f_{[m+1,n]}^z(y_{n+1})=\underset{n\to\infty}{\limi}f_{[m+1,n]}^w(y_{n+1})=x_{m+1}^w,
$$
where the middle equality uses the assumption $z\re\omega\setminus (m+1)=w\re\omega\setminus (m+1)$.
\end{proof}

\begin{corollary}\label{selfdualcorollary}
Let $\bS$ be a nice topological pointclass, and assume that $\Det(\bS(\omega^\omega))$ holds. Let $Z$ be a zero-dimensional Polish space, let $A\in\bS(Z)$ be selfdual, and let $U\in\mathbf{\Delta}^0_1(Z)$. Then there exist pairwise disjoint $V_n\in\mathbf{\Delta}^0_1(U)$ and non-selfdual $A_n<A$ in $Z$ for $n\in\omega$ such that $\bigcup_{n\in\omega}V_n=U$ and $\bigcup_{n\in\omega}(A_n\cap V_n)=A\cap U$.
\end{corollary}
\begin{proof}
As one can easily check, it will be enough to show that there exists a partition $\VV\subseteq\mathbf{\Delta}^0_1(U)$ of $U$ such that for every $V\in\VV$ either $A\cap V\in\mathbf{\Delta}^0_1(Z)$ or $A\cap V$ is non-selfdual in $Z$. If this were not the case, then, using Theorem \ref{selfdualtheorem}, one could recursively construct a strictly $\leq$-decreasing sequence of subsets of $Z$, which would contradict Theorem \ref{wellfounded}.
\end{proof}

\section{Relativization: basic facts}\label{sectionrelativizationbasic}

When one tries to give a systematic exposition of Wadge theory, it soon becomes apparent that it would be very useful to be able to say when $A$ and $B$ belong to ``the same'' Wadge class $\bG$, even when $A\subseteq Z$ and $B\subseteq W$ for distinct ambient spaces $Z$ and $W$. It is clear how to do that in certain particular cases, for example when $\bG=\bP^0_2$ or $\bG=\Diff_2(\bS^0_1)$, because elements of those classes are obtained by performing set-theoretic operations\footnote{\,The key fact here is that these are Hausdorff operations (see Section \ref{sectionhausdorffbasic}). In fact, in \cite{carroymedinimuller}, we used Hausdorff operations (together with Theorem \ref{vanwesepmain}) to give an alternative treatment of relativization.} to the open sets. However, it is not a priori clear how to deal with this issue in the case of arbitrary, possibly rather exotic Wadge classes.

We will solve the problem by using Wadge classes in $\omega^\omega$ to parametrize Wadge classes in arbitrary spaces. Roughly, using this approach, two Wadge classes $\bL$ in $Z$ and $\bL'$ in $W$ will be ``the same'' if there exists a Wadge class $\bG$ in $\omega^\omega$ such that $\bG(Z)=\bL$ and $\bG(W)=\bL'$. We will refer to this process as relativization. This is essentially due to Louveau and Saint-Raymond (see \cite[Theorem 4.2]{louveausaintraymondc}), but here we tried to give a more systematic exposition. Furthermore, as we mentioned in Section \ref{sectionintroduction}, this topic does not appear at all in \cite{louveaub}.

The reason why we used the word ``roughly'' is that, in order for relativization to work, the Wadge classes in question have to be non-selfdual (see the discussion at the end of Section \ref{sectionwadgebasic}). Furthermore, the ambient spaces $Z$ and $W$ are generally assumed to be be zero-dimensional and Polish, even though for some results the assumption ``Polish'' can be relaxed to ``Borel'',\footnote{\,This form of relativization will be needed in the proof of Theorem \ref{everyclasshasalevelprelim}.} or even dropped altogether.

Lemma \ref{relativization}, whose straightforward proof is left to the reader, gives several ``reassuring'' and extremely useful facts about relativization. Lemma \ref{relativizationprelim} gives equivalent definitions of $\bG(Z)$. An important application of these appears in the proof of Lemma \ref{relativizationsubspace}, which shows that relativization is well-behaved with respect to subspaces. As another application, observe that if $\bS$ is a nice topological pointclass and $\Det(\bS(\omega^\omega))$ holds, then $\bG(Z)\subseteq\bS(Z)$ whenever $\bG\in\NSDS(\omega^\omega)$ and $Z$ is a zero-dimensional Borel space. Finally, Lemma \ref{relativizationexistsunique} shows that every non-selfdual Wadge class can be obtained through relativization, and in a unique way.

\begin{definition}[Louveau, Saint-Raymond]\label{relativizationdefinition}
Given a space $Z$ and $\bG\subseteq\PP(\omega^\omega)$, define
$$
\bG(Z)=\{A\subseteq Z:g^{-1}[A]\in\bG\text{ for every continuous }g:\omega^\omega\longrightarrow Z\}.
$$	
\end{definition}

\begin{lemma}\label{relativization}
Let $\bG\subseteq\PP(\omega^\omega)$, and let $Z$ and $W$ be spaces.
\begin{enumerate}
\item\label{relativizationpreimage} If $f:Z\longrightarrow W$ is continuous and $B\in\bG(W)$ then $f^{-1}[B]\in\bG(Z)$.
\item\label{relativizationhomeo} If $h:Z\longrightarrow W$ is a homeomorphism then $A\in\bG(Z)$ iff $h[A]\in\bG(W)$.
\item\label{relativizationcheck} $\widecheck{\bG(Z)}=\bGc(Z)$.
\item\label{relativizationbaire} If $\bG$ is continuously closed then $\bG(\omega^\omega)=\bG$.
\end{enumerate}
\end{lemma}

\begin{lemma}[Louveau, Saint-Raymond]\label{relativizationprelim}
Let $\bS$ be a nice topological pointclass, and assume that $\Det(\bS(\omega^\omega))$ holds. Let $Z$ be a zero-dimensional Borel space, let $\bG\in\NSDS(\omega^\omega)$, and let $A\in\bS(Z)$. Then, the following conditions are equivalent:
\begin{enumerate}
\item\label{forallcont} $A\in\bG(Z)$,
\item\label{forallemb} For every embedding $j:Z\longrightarrow\omega^\omega$ there exists $B\in\bG$ such that $A=j^{-1}[B]$,
\item\label{existsemb} There exists an embedding $j:Z\longrightarrow\omega^\omega$ and $B\in\bG$ such that $A=j^{-1}[B]$,
\item\label{existscont} There exists a continuous $f:Z\longrightarrow\omega^\omega$ and $B\in\bG$ such that $A=f^{-1}[B]$.
\end{enumerate}
\end{lemma}
\begin{proof}
In order to prove that $(\ref{forallcont})\rightarrow (\ref{forallemb})$, assume that condition $(\ref{forallcont})$ holds. Pick an embedding $j:Z\longrightarrow\omega^\omega$, then set $A_0=j[A]$ and $A_1=j[Z\setminus A]$. Using the fact that $\bS$ is a nice topological pointclass and that $Z$ is a Borel space, it is easy to see that $A_0,A_1\in\bS(\omega^\omega)$. Therefore, by the assumption $\Det(\bS(\omega^\omega))$, it is possible to apply Lemma \ref{strongwadgelemma}. Notice that it would be sufficient to show that there exists $B\in\bG$ such that $A_0\subseteq B$ and $B\cap A_1=\varnothing$, as $j^{-1}[B]=A$ would clearly follow. So assume, in order to get a contradiction, that no such $B\in\bG$ exists. Then, by Lemma \ref{strongwadgelemma}, for all $B\in\bGc$ there exists a continuous $f:\omega^\omega\longrightarrow\omega^\omega$ such that $f[\omega^\omega]\subseteq A_0\cup A_1$ and $f^{-1}[A_0]=B$. In particular, we can fix such a function $f$ when $B$ is such that $\bGc=B\wc$. Set $g=j^{-1}\circ f:\omega^\omega\longrightarrow Z$, and observe that $g$ is continuous because $j$ is an embedding. Then
$$
B=f^{-1}[A_0]=f^{-1}[j[A]]=g^{-1}[A]\in\bG
$$
by condition $(\ref{forallcont})$, contradicting the fact that $\bG$ is non-selfdual.

The implication $(\ref{forallemb})\rightarrow (\ref{existsemb})$ holds because $Z$ is zero-dimensional. The implication $(\ref{existsemb})\rightarrow (\ref{existscont})$ is trivial. In order to prove that $(\ref{existscont})\rightarrow (\ref{forallcont})$, assume that $f:Z\longrightarrow\omega^\omega$ and $B\in\bG$ are such that $A=f^{-1}[B]$. Pick a continuous $g:\omega^\omega\longrightarrow Z$. Since $f\circ g:\omega^\omega\longrightarrow\omega^\omega$ is continuous and $\bG$ is continuously closed, one sees that
\[
g^{-1}[A]=g^{-1}[f^{-1}[B]]=(f\circ g)^{-1}[B]\in\bG.\qedhere
\]
\end{proof}

\begin{lemma}\label{relativizationsubspace}
Let $\bS$ be a nice topological pointclass, and assume that $\Det(\bS(\omega^\omega))$ holds. Let $Z$ and $W$ be zero-dimensional Borel spaces such that $W\subseteq Z$, and let $\bG\in\NSDS(\omega^\omega)$. Then $B\in\bG(W)$ iff $A\cap W=B$ for some $A\in\bG(Z)$.
\end{lemma}
\begin{proof}
In order to prove the left-to-right implication, pick $B\in\bG(W)$. Since $Z$ is zero-dimensional, we can fix an embedding $j:Z\longrightarrow\omega^\omega$. Notice that $i=j\re W:W\longrightarrow\omega^\omega$ is also an embedding, hence by condition $(\ref{forallemb})$ of Lemma \ref{relativizationprelim} there exists $C\in\bG$ such that $i^{-1}[C]=B$. Let $A=j^{-1}[C]$, and observe that $A\cap W=B$. The fact that $A\in\bG(Z)$ follows from condition $(\ref{existsemb})$ of Lemma \ref{relativizationprelim}. In order to prove the right-to-left implication, pick $A\in\bG(Z)$. Let $i:W\longrightarrow Z$ be the inclusion. It follows from Lemma \ref{relativization}.\ref{relativizationpreimage} that $A\cap W=i^{-1}[A]\in\bG(W)$.
\end{proof}

\begin{lemma}\label{relativizationexistsunique}
Let $\bS$ be a nice topological pointclass, and assume that $\Det(\bS(\omega^\omega))$ holds. Let $Z$ be a zero-dimensional Polish space, and let $\bL\in\NSDS(Z)$. Then there exists a unique $\bG\in\NSDS(\omega^\omega)$ such that $\bG(Z)=\bL$.
\end{lemma}
\begin{proof}
First we will prove the existence of $\bG$. Pick $A\subseteq Z$ such that $\bL=A\wc$. By \cite[Theorem 7.8]{kechris} and Lemma \ref{relativization}.\ref{relativizationhomeo} we can assume without loss of generality that $Z$ is a closed subspace of $\omega^\omega$. Therefore, by \cite[Proposition 2.8]{kechris}, we can fix a retraction $\rho:\omega^\omega\longrightarrow Z$. Set $\bG=\rho^{-1}[A]\wc$ in $\omega^\omega$, and observe that $\bG$ is non-selfdual by Lemma \ref{bairetoall}. We claim that $\bL=\bG(Z)$. Since $A=\rho^{-1}[A]\cap Z\in\bG(Z)$ by Lemma \ref{relativizationsubspace} and $\bG(Z)$ is continuously closed, one sees that $\bL=A\wc\subseteq\bG(Z)$. To see that the other inclusion holds, pick $B\in\bG(Z)$. Then $\rho^{-1}[B]\in\bG$, hence $\rho^{-1}[B]\leq\rho^{-1}[A]$. It follows from Lemma \ref{bairetoall} that $B\leq A$.

Now assume, in order to get a contradiction, that $\bG,\bG'\in\NSDS(\omega^\omega)$ are such that $\bG\neq\bG'$ and $\bG(Z)=\bL=\bG'(Z)$. Notice that $\bG'=\bGc$ is impossible, as an application of Lemma \ref{relativization}.\ref{relativizationcheck} would contradict the fact that $\bL$ is non-selfdual. Therefore, we can assume without loss of generality that $\bG\subseteq\bG'$, hence $\bG'\nsubseteq\bG$. By Lemma \ref{wadgelemma}, it follows that $\bGc\subseteq\bG'$. Therefore $\bGc(Z)\subseteq\bG'(Z)=\bG(Z)$, which contradicts the fact that $\bG(Z)=\bL$ is non-selfdual.
\end{proof}

\section{Relativization: uncountable spaces}\label{sectionrelativizationuncountable}

Notice that, in the previous section, we never assumed the uncountability of the ambient spaces. As the following two results show, the situation gets particularly pleasant when this assumption is satisfied.
In particular, Theorem \ref{orderisomorphism} shows that the ordering of non-selfdual Wadge classes becomes independent of the ambient space. Observe that the uncountability assumption cannot be dropped in either result, as $\bG(Z)=\PP(Z)$ whenever $Z$ is a countable space and $\bG\subseteq\PP(\omega^\omega)$ is such that $\bD^0_2(\omega^\omega)\subseteq\bG$.

\begin{theorem}\label{orderisomorphism}
Let $\bS$ be a nice topological pointclass, and assume that $\Det(\bS(\omega^\omega))$ holds. Let $Z$ and $W$ be uncountable zero-dimensional Borel spaces, and let $\bG,\bL\in\NSDS(\omega^\omega)$. Then
$$
\bG(Z)\subseteq\bL(Z)\text{ iff }\bG(W)\subseteq\bL(W).
$$
\end{theorem}
\begin{proof}
It will be enough to prove the left-to-right implication, as the other implication is perfectly analogous. So assume that $\bG(Z)\subseteq\bL(Z)$, and let $B\in\bG(W)$. Since $Z$ is an uncountable Borel space and $W$ is zero-dimensional, there exists an embedding of $W$ into $Z$. Hence, using Lemma \ref{relativization}.\ref{relativizationhomeo}, we can assume without loss of generality that $W\subseteq Z$. By Lemma \ref{relativizationsubspace}, there exists $A\in\bG(Z)$ such that $A\cap W=B$. Since $A\in\bL(Z)$ by our assumption, a further application of Lemma \ref{relativizationsubspace} shows that $B\in\bL(W)$.	
\end{proof}

\begin{theorem}\label{vanwesepsurrogate}
Let $\bS$ be a nice topological pointclass, and assume that $\Det(\bS(\omega^\omega))$ holds. Let $Z$ be an uncountable zero-dimensional Polish space. Then
$$
\NSDS(Z)=\{\bG(Z):\bG\in\NSDS(\omega^\omega)\}.
$$
\end{theorem}
\begin{proof}
The inclusion $\subseteq$ holds by Lemma \ref{relativizationexistsunique}. In order to prove that the inclusion $\supseteq$ holds, pick $\bG\in\NSDS(\omega^\omega)$. By Lemma \ref{nonselfdualcontinuouslyclosed}, it will be enough to show that $\bG(Z)$ is continuously closed and non-selfdual. The fact that $\bG(Z)$ is continuously closed follows from Lemma \ref{relativization}.\ref{relativizationpreimage}. Now assume, in order to get a contradiction, that $\bG(Z)$ is selfdual. Then $\bGc(Z)=\bG(Z)$ by Lemma \ref{relativization}.\ref{relativizationcheck}, which implies $\bGc(\omega^\omega)=\bG(\omega^\omega)$ by Theorem \ref{orderisomorphism}. It follows from Lemma \ref{relativization}.\ref{relativizationbaire} that $\bGc=\bG$, which is a contradiction.
\end{proof}

\section{Hausdorff operations: basic facts}\label{sectionhausdorffbasic}

For a history of the following important notion, see \cite[page 583]{hausdorff}. For a modern survey, we recommend \cite{zafrany}. Most of the proofs in this section are straightforward, hence we leave them to the reader.

\begin{definition}
Given a set $Z$ and $D\subseteq\PP(\omega)$, define
$$
\HH_D(A_0,A_1,\ldots)=\{x\in Z:\{n\in\omega:x\in A_n\}\in D\}
$$
whenever $A_0,A_1,\ldots\subseteq Z$. Functions of this form are called \emph{Hausdorff operations} (or \emph{$\omega$-ary Boolean operations}).
\end{definition}

Of course, the function $\HH_D$ depends on the set $Z$, but what $Z$ is will usually be clear from the context. In case there might be uncertainty about the ambient space, we will use the notation $\HH_D^Z$. Notice that, once $D$ is specified, the corresponding Hausdorff operation simultaneously defines functions $\PP(Z)^\omega\longrightarrow\PP(Z)$ for every $Z$.

The following proposition lists the most basic properties of Hausdorff operations. Given $n\in\omega$, set $S_n=\{A\subseteq\omega:n\in A\}$.

\begin{proposition}\label{hausdorffsettheoretic}
Let $I$ be a set, and let $D_i\subseteq\PP(\omega)$ for every $i\in I$. Fix an ambient set $Z$ and $A_0,A_1,\ldots\subseteq Z$.
\begin{itemize}
\item $\HH_{S_n}(A_0,A_1,\ldots)=A_n$ for all $n\in\omega$.
\item $\bigcap_{i\in I}\HH_{D_i}(A_0,A_1,\ldots)=\HH_D(A_0,A_1,\ldots)$, where $D=\bigcap_{i\in I}D_i$.
\item $\bigcup_{i\in I}\HH_{D_i}(A_0,A_1,\ldots)=\HH_D(A_0,A_1,\ldots)$, where $D=\bigcup_{i\in I}D_i$.
\item $Z\setminus\HH_D(A_0,A_1,\ldots)=\HH_{\PP(\omega)\setminus D}(A_0,A_1,\ldots)$ for all $D\subseteq\PP(\omega)$.
\end{itemize}
\end{proposition}
The point of the above proposition is that any operation obtained by combining unions, intersections and complements can be expressed as a Hausdorff operation. For example, if $D=\bigcup_{n\in\omega}(S_{2n+1}\setminus S_{2n})$, then $\HH_D(A_0,A_1,\ldots)=\bigcup_{n\in\omega}(A_{2n+1}\setminus A_{2n})$.

The following proposition shows that the composition of Hausdorff operations is again a Hausdorff operation. We will assume that a bijection $\langle\cdot,\cdot\rangle:\omega\times\omega\longrightarrow\omega$ has been fixed.
\begin{proposition}\label{hausdorffcomposition}
Let $Z$ be a set, let $D\subseteq\PP(\omega)$ and $E_m\subseteq\PP(\omega)$ for $m\in\omega$. Then there exists $F\subseteq\PP(\omega)$ such that
$$
\HH_D(B_0,B_1,\ldots)=\HH_F(A_0,A_1,\ldots)
$$
for all $A_0,A_1,\ldots\subseteq Z$, where $B_m=\HH_{E_m}(A_{\langle m,0\rangle},A_{\langle m,1\rangle},\ldots)$ for $m\in\omega$.
\end{proposition}
\begin{proof}
Define $z\in F$ if $\{m\in\omega:\{n\in\omega:\langle m,n\rangle\in z\}\in E_m\}\in D$. The rest of the proof is a straightforward verification.
\end{proof}

Finally, we state a result that will be needed in the next section (see the proof of Lemma \ref{hausdorffrelativization}).
\begin{lemma}\label{relativizationsettheoretic}
Let $Z$ and $W$ be sets, let $D\subseteq\PP(\omega)$, let $A_0,A_1,\ldots\subseteq Z$, and let $B_0,B_1,\ldots\subseteq W$.
\begin{enumerate}
\item\label{preimagesettheoretic} $f^{-1}[\HH_D(B_0,B_1,\ldots)]=\HH_D(f^{-1}[B_0],f^{-1}[B_1],\ldots)$ for all $f:Z\longrightarrow W$.
\item\label{homeomorphismsettheoretic} $f[\HH_D(A_0,A_1,\ldots)]=\HH_D(f[A_0],f[A_1],\ldots)$ for all bijections $f:Z\longrightarrow W$.
\item\label{subspacesettheoretic} $W\cap\HH_D^Z(A_0,A_1,\ldots)=\HH_D^W(A_0\cap W,A_1\cap W,\ldots)$ whenever $W\subseteq Z$.
\end{enumerate}
\end{lemma}

\section{Hausdorff operations: the associated classes}\label{sectionhausdorffclasses}

In the context of this article, the most important fact regarding Hausdorff operations is that they all give rise in a natural way to a non-selfdual Wadge class. More precisely, as in the following definition, this class is obtained by applying the given Hausdorff operation to all open subsets of a given space. The claim that these are all non-selfdual Wadge classes will be proved in the next section (see Theorem \ref{addisontheorem}), by employing the classical notion of a universal set.

\begin{definition}
Given a space $Z$ and $D\subseteq\PP(\omega)$, define
$$
\bG_D(Z)=\{\HH_D(A_0,A_1,\ldots):A_n\in\mathbf{\Sigma}^0_1(Z)\text{ for every }n\in\omega\}.
$$
\end{definition}

\begin{definition}
Given a space $Z$, define
$$
\Ha(Z)=\{\bG_D(Z):D\subseteq\PP(\omega)\}.
$$
Also set $\HaS(Z)=\{\bG\in\Ha(Z):\bG\subseteq\bS(Z)\}$ whenever $\bS$ is a topological pointclass.
\end{definition}

As examples (that will be useful later), consider the following two simple propositions.

\begin{proposition}\label{hausdorffdifferences}
Let $1\leq\eta<\omega_1$. Then there exists $D\subseteq\PP(\omega)$ such that $\bG_D(Z)=\mathsf{D}_\eta(\mathbf{\Sigma}^0_1(Z))$ for every space $Z$.
\end{proposition}
\begin{proof}
This follows from Propositions \ref{hausdorffsettheoretic} and \ref{hausdorffcomposition} (in case $\eta>\omega$, use a bijection $\pi:\eta\longrightarrow\omega$).
\end{proof}

\begin{proposition}\label{hausdorffborel}
Let $1\leq\xi<\omega_1$. Then there exists $D\subseteq\PP(\omega)$ such that $\bG_D(Z)=\mathbf{\Sigma}^0_\xi(Z)$ for every space $Z$.
\end{proposition}
\begin{proof}
This can be proved by induction on $\xi$, using Propositions \ref{hausdorffsettheoretic} and \ref{hausdorffcomposition}.
\end{proof}

Finally, we state a useful lemma, which shows that this notion behaves well with respect to subspaces and continuous functions. It extends (and is inspired by) \cite[Lemma 2.3]{vanengelena}.

\newpage
\begin{lemma}\label{hausdorffrelativization}
Let $Z$ and $W$ be spaces, and let $D\subseteq\PP(\omega)$.
\begin{enumerate}
\item\label{hausdorffrelativizationpreimage} If $f:Z\longrightarrow W$ is continuous and $B\in\bG_D(W)$ then $f^{-1}[B]\in\bG_D(Z)$.
\item\label{hausdorffrelativizationhomeomorphism} If $h:Z\longrightarrow W$ is a homeomorphism then $A\in\bG_D(Z)$ iff $h[A]\in\bG_D(W)$.
\item\label{hausdorffrelativizationsubspace} Assume that $W\subseteq Z$. Then $B\in\bG_D(W)$ iff there exists $A\in\bG_D(Z)$ such that $B=A\cap W$.
\end{enumerate}
\end{lemma}
\begin{proof}
This is a straightforward consequence of Lemma \ref{relativizationsettheoretic}.
\end{proof}

\section{Hausdorff operations: universal sets}\label{sectionhausdorffuniversal}

The aim of this section is to show that $\Ha(Z)\subseteq\NSD(Z)$ whenever $Z$ is an uncountable zero-dimensional Polish space (see Theorem \ref{addisontheorem} for a more precise statement). Notice that the uncountability requirement cannot be dropped, as $\bS^0_2(Z)=\PP(Z)$ is selfdual whenever $Z$ is countable. The ideas presented here are well-known, but since we could not find a satisfactory reference, we will give all the details. Our approach is inspired by \cite[Section 22.A]{kechris}. 

\begin{definition}
Let $Z$ and $W$ be spaces, and let $D\subseteq\PP(\omega)$. Given $U\subseteq W\times Z$ and $x\in W$, let $U_x=\{y\in Z:(x,y)\in U\}$ denote the vertical section of $U$ above $x$. We will say that $U\subseteq W\times Z$ is a \emph{$W$-universal set} for $\bG_D(Z)$ if the following two conditions hold:
\begin{itemize}
\item $U\in\bG_D(W\times Z)$,
\item $\{U_x:x\in W\}=\bG_D(Z)$.
\end{itemize}	
\end{definition}

Notice that, by Proposition \ref{hausdorffborel}, the above definition applies to the class $\mathbf{\Sigma}^0_\xi(Z)$ whenever $1\leq\xi <\omega_1$. Furthermore, for these classes, this definition agrees with \cite[Definition 22.2]{kechris}.

\begin{proposition}\label{existsCuniversal}
Let $Z$ be a space, and let $D\subseteq\PP(\omega)$. Then there exists a $2^\omega$-universal set for $\bG_D(Z)$.
\end{proposition}
\begin{proof}
By \cite[Theorem 22.3]{kechris}, we can fix a $2^\omega$-universal set $U$ for $\mathbf{\Sigma}^0_1(Z)$. Let $h:2^\omega\longrightarrow (2^\omega)^\omega$ be a homeomorphism, and let $\pi_n:(2^\omega)^\omega\longrightarrow 2^\omega$ be the projection on the $n$-th coordinate for $n\in\omega$. Notice that, given any $n\in\omega$, the function $f_n:2^\omega\times Z\longrightarrow 2^\omega\times Z$ defined by $f_n(x,y)=(\pi_n(h(x)),y)$ is continuous. Let $V_n=f_n^{-1}[U]$ for each $n$, and observe that each $V_n\in\mathbf{\Sigma}^0_1(2^\omega\times Z)$. Set $V=\HH_D(V_0,V_1,\ldots)$.

We claim that $V$ is a $2^\omega$-universal set for $\bG_D(Z)$. It is clear that $V\in\bG_D(2^\omega\times Z)$. Furthermore, using Lemma \ref{hausdorffrelativization}, one can easily check that $V_x\in\bG_D(Z)$ for every $x\in 2^\omega$. To complete the proof, fix $A\in\bG_D(Z)$. Let $A_0,A_1,\ldots\in\mathbf{\Sigma}^0_1(Z)$ be such that $A=\HH_D(A_0,A_1,\ldots)$. Since $U$ is $2^\omega$-universal, we can fix $z_n\in 2^\omega$ such that $U_{z_n}=A_n$ for every $n\in\omega$. Set $z=h^{-1}(z_0,z_1,\ldots)$. It is straightforward to verify that $V_z=A$.
\end{proof}

\begin{corollary}\label{existsZuniversal}
Let $Z$ be a space in which $2^\omega$ embeds, and let $D\subseteq\PP(\omega)$. Then there exists a $Z$-universal set for $\bG_D(Z)$.
\end{corollary}
\begin{proof}
By Proposition \ref{existsCuniversal}, we can fix a $2^\omega$-universal set $U$ for $\bG_D(Z)$. Fix an embedding $j:2^\omega\longrightarrow Z$ and set $W=j[2^\omega]$. Notice that $(j\times\id_Z)[U]\in\bG_D(W\times Z)$ by Lemma \ref{hausdorffrelativization}.\ref{hausdorffrelativizationhomeomorphism}. Therefore, by Lemma \ref{hausdorffrelativization}.\ref{hausdorffrelativizationsubspace}, there exists $V\in\bG_D(Z\times Z)$ such that $V\cap (W\times Z)=(j\times\id_Z)[U]$. Using Lemma \ref{hausdorffrelativization} again, one can easily check that $V$ is a $Z$-universal set for $\bG_D(Z)$.
\end{proof}

\begin{lemma}\label{Zuniversalnonselfdual}
Let $Z$ be a space, and let $D\subseteq\PP(\omega)$. Assume that there exists a $Z$-universal set for $\bG_D(Z)$. Then $\bG_D(Z)$ is non-selfdual.
\end{lemma}
\begin{proof}
Fix a $Z$-universal set $U\subseteq Z\times Z$ for $\bG_D(Z)$. Assume, in order to get a contradiction, that $\bG_D(Z)$ is selfdual. Let $f:Z\longrightarrow Z\times Z$ be the function defined by $f(x)=(x,x)$, and observe that $f$ is continuous. Since $f^{-1}[U]\in\bG_D(Z)=\bGc_D(Z)$, we see that $Z\setminus f^{-1}[U]\in\bG_D(Z)$. Therefore, since $U$ is $Z$-universal, we can fix $z\in Z$ such that $U_z=Z\setminus f^{-1}[U]$. If $z\in U_z$ then $f(z)=(z,z)\in U$ by the definition of $U_z$, contradicting the fact that $U_z=Z\setminus f^{-1}[U]$. On the other hand, if $z\notin U_z$ then $f(z)=(z,z)\notin U$ by the definition of $U_z$, contradicting the fact that $Z\setminus U_z=f^{-1}[U]$.
\end{proof}

The case $Z=\omega^\omega$ of the following result is \cite[Proposition 5.0.3]{vanwesept}, and it is credited to Addison by Van Wesep.
\begin{theorem}\label{addisontheorem}
Let $Z$ be a zero-dimensional space in which $2^\omega$ embeds. Then $\Ha(Z)\subseteq\NSD(Z)$.
\end{theorem}
\begin{proof}
Pick $D\subseteq\PP(\omega)$. The fact that $\bG_D(Z)$ is non-selfdual follows from Corollary \ref{existsZuniversal} and Lemma \ref{Zuniversalnonselfdual}. Therefore, it will be enough to show that $\bG_D(Z)$ is a Wadge class. By Proposition \ref{existsCuniversal}, we can fix a $2^\omega$-universal set $U\subseteq 2^\omega\times Z$ for $\bG_D(Z)$. Fix an embedding $j:2^\omega\times Z\longrightarrow Z$ and set $W=j[2^\omega\times Z]$. By Lemma \ref{hausdorffrelativization}, we can fix $A\in\bG_D(Z)$ such that $A\cap W=j[U]$. We claim that $\bG_D(Z)=A\wc$. The inclusion $\supseteq$ follows from Lemma \ref{hausdorffrelativization}.\ref{hausdorffrelativizationpreimage}. In order to prove the other inclusion, pick $B\in\bG_D(Z)$. Since $U$ is $2^\omega$-universal, we can fix $z\in 2^\omega$ such that $B=U_z$. Consider the function $f:Z\longrightarrow 2^\omega\times Z$ defined by $f(x)=(z,x)$, and observe that $f$ is continuous. It is straightforward to check that $j\circ f:Z\longrightarrow Z$ witnesses that $B\leq A$ in $Z$.
\end{proof}

\section{The complete analysis of $\bD^0_2$}\label{sectioncompleteanalysis}

Let $Z$ be an uncountable zero-dimensional Polish space. Observe that $\Diff_\eta(\bS^0_1(Z))\in\NSD(Z)$ whenever $1\leq\eta<\omega_1$ by Proposition \ref{hausdorffdifferences} and Theorem \ref{addisontheorem}. In this section we will show that these are the only non-trivial elements of $\NSD(Z)$ contained in $\bD^0_2(Z)$. We will need this fact in the proof of Theorem \ref{sdmain}. The case $Z=\omega^\omega$ of this result is already mentioned in \cite[pages 84-85]{vanwesept}, but we are not aware of a satisfactory reference for it. We begin by stating a classical result (see \cite[Theorem 22.27]{kechris} for a proof).

\begin{theorem}[Hausdorff, Kuratowski]\label{hausdorffkuratowskitheorem}
Let $Z$ be a Polish space, and let $1\leq\xi<\omega_1$. Then
$$
\bD^0_{\xi+1}(Z)=\bigcup_{1\leq\eta<\omega_1}\Diff_\eta(\bS^0_\xi(Z)).
$$
\end{theorem}

\begin{theorem}\label{completeanalysisdifferences}
Let $Z$ be a zero-dimensional Polish space, and let $\bG\in\NSD(Z)$ be such that $\bG\subseteq\bD^0_2(Z)$. Assume that $\bG\neq\{\varnothing\}$ and $\bG\neq\{Z\}$. Then there exists $1\leq\eta<\omega_1$ such that $\bG=\Diff_\eta(\bS^0_1(Z))$ or $\bG=\widecheck{\Diff}_\eta(\bS^0_1(Z))$.
\end{theorem}
\begin{proof}
Pick $A\subseteq Z$ such that $\bG=A\wc$. By Theorem \ref{hausdorffkuratowskitheorem}, we can fix the minimal$\eta$ such that $1\leq\eta<\omega_1$ and $A\in\Diff_\eta(\bS^0_1(Z))\cup\widecheck{\Diff}_\eta(\bS^0_1(Z))$. We will only give the proof in the case $A\in\Diff_\eta(\bS^0_1(Z))$, as the other case is perfectly analogous. More specifically, assume that $A=\Diff_\eta(A_\xi:\xi<\eta)$, where each $A_\xi\in\bS^0_1(Z)$ and $(A_\xi:\xi<\eta)$ is a $\subseteq$-increasing sequence. We claim that $\bG=\Diff_\eta(\bS^0_1(Z))$.~By~Lemma \ref{wadgelemma} and the fact that $\Diff_\eta(\bS^0_1(Z))\in\NSD(Z)$, it will be enough to show that $A\notin\widecheck{\Diff}_\eta(\bS^0_1(Z))$.

Assume, in order to get a contradiction, that $A\in\widecheck{\Diff}_\eta(\bS^0_1(Z))$. More specifically, assume that $Z\setminus A=\Diff_\eta(B_\xi:\xi<\eta)$, where each $B_\xi\in\bS^0_1(Z)$ and $(B_\xi:\xi<\eta)$ is a $\subseteq$-increasing sequence. First assume that $\eta$ is a limit ordinal. Define $\UU=\{A_\xi:\xi<\eta\}\cup\{B_\xi:\xi<\eta\}$. It is clear that $\UU$ is an open cover of $Z$. Furthermore, it is easy to realize that for every $U\in\UU$ there exists $\xi<\eta$ such that either $A\cap U\in\Diff_\xi(\bS^0_1(Z))$ or $(Z\setminus A)\cap U\in\Diff_\xi(\bS^0_1(Z))$. Using Lemma \ref{closureclopen} and the fact that $Z$ is zero-dimensional, we can assume without loss of generality that $\UU$ consists of clopen subsets of $Z$.

We claim that $A\cap U<A$ for every $U\in\UU$. This will conclude the proof by Proposition \ref{locallylowerimpliesselfdual}, as it will contradict the fact that $A$ is non-selfdual. Pick $U\in\UU$. If there exists $\xi<\eta$ such that $A\cap U\in\Diff_\xi(\bS^0_1(Z))$, then the claim holds by the minimality of $\eta$. If $(Z\setminus A)\cap U\in\Diff_\xi(\bS^0_1(Z))$, then
$$
A\cap U=(Z\setminus ((Z\setminus A)\cap U))\cap U\in\widecheck{\Diff}_\xi(\bS^0_1(Z))
$$
by Lemma \ref{closureclopen}. Hence the claim follows from the minimality of $\eta$ again.

Finally, assume that $\eta=\zeta+1$ is a successor ordinal. Notice that $\zeta\neq 0$, otherwise $A$ would be a clopen subset of $Z$ such that $\varnothing\subsetneq A\subsetneq Z$, contradicting the fact that $A$ is non-selfdual. So it makes sense to set $C=\Diff_\zeta(A_\xi:\xi<\zeta)$ and $D=\Diff_\zeta(B_\xi:\xi<\zeta)$. Observe that $A=A_\zeta\setminus C$ and $Z\setminus A=B_\zeta\setminus D$. Since $A_\zeta\cup B_\zeta=Z$, using the fact that $Z$ is zero-dimensional it is possible to find $U,V\in\bD^0_1(Z)$ such that $U\subseteq A_\zeta$, $V\subseteq B_\zeta$, $U\cup V=Z$ and $U\cap V=\varnothing$. Notice that $A\cap U=U\setminus C=(Z\setminus C)\cap U\in\widecheck{\Diff}_\zeta(\bS^0_1(Z))$ and $A\cap V=D\cap V\in\Diff_\zeta(\bS^0_1(Z))$ by Lemma \ref{closureclopen}. As above, this contradicts the fact that $A$ is non-selfdual.
\end{proof}

\section{Kuratowski's transfer theorem}\label{sectionkuratowski}

In this section, we will prove many forms of a classical result, known as ``Kuratowski's transfer theorem''. The most powerful form of this result (as it gives the sharpest bounds in terms of complexity) is Theorem \ref{kuratowskitheorem} (which is taken from \cite[Theorem 7.1.6]{louveaub}). This strong version of the result will only be needed in Section \ref{sectionlevelhard}. The weaker versions will be used to successfully employ the notion of expansion. We also point out that Corollary \ref{sigmaintoopen} can be easily obtained from \cite[Theorem 22.18]{kechris}, and viceversa.

Given $f:Z\longrightarrow W$, we will denote by $f^\ast:Z\longrightarrow Z\times W$ the function defined by setting $f^\ast(x)=(x,f(x))$ for every $x\in Z$. Given a set $I$, a function $f:Z\longrightarrow\prod_{k\in I}W_k$ and $k\in I$, we will denote the $k$-th coordinate of $f$ by $f_k:Z\longrightarrow W_k$. More precisely, set $f_k(x)=f(x)(k)$ for every $x\in Z$. In almost all of our applications, $I=\omega$ and $W_k=\omega$ for each $k$, so that $f:Z\longrightarrow\omega^\omega$ and $f_k(x)=n_k$ for $x\in Z$, where $f(x)=(n_0,n_1,\ldots)$.

The following is the crucial concept in Theorem \ref{kuratowskitheorem}, and it will feature prominently in Section \ref{sectionlevelhard} as well. Its name comes from the fact that it yields a finer topology (see Corollary \ref{sigmaintoopen}). 

\begin{definition}
Let $Z$ be a space, let $2\leq\xi<\omega_1$, and let $\Aa\subseteq\bS^0_\xi(Z)$. We will say that $f:Z\longrightarrow\omega^\omega$ is a \emph{$\xi$-refining function} for $\Aa$ if it satisfies the following conditions, where $F=f^\ast[Z]$ denotes the graph of $f$:
\begin{enumerate}
 \item\label{kuratowskigraphclosed} $F$ is closed in $Z\times\omega^\omega$,
 \item\label{kuratowskiinverseimage} $f^\ast[A]\in\bS^0_1(F)$ for every $A\in\Aa$,
 \item\label{kuratowskistrictlysmaller} For all $k\in\omega$ there exists $\xi_k$ such that $1\leq\xi_k<\xi$ and $f_k^{-1}(j)\in\bP^0_{\xi_k}(Z)$ for every $j\in\omega$.
\end{enumerate}
\end{definition}

\begin{theorem}[Louveau]\label{kuratowskitheorem}
Let $Z$ be a zero-dimensional space, let $2\leq\xi<\omega_1$, and let $\Aa\subseteq\bS^0_\xi(Z)$ be countable. Then there exists a $\xi$-refining function $f:Z\longrightarrow\omega^\omega$ for $\Aa$.
\end{theorem}
\begin{proof}
The result is trivial if $\Aa=\varnothing$, so assume that $\Aa\neq\varnothing$. Let $\Aa=\{A_n:n\in\omega\}$ be an enumeration. We will proceed by induction on $\xi$. First assume that $\xi=2$. Pick $A_{(n,i)}\in\bP^0_1(Z)$ for $n,i\in\omega$ such that $A_n=\bigcup_{i\in\omega}A_{(n,i)}$. Since $Z$ is zero-dimensional, it is possible to pick $A_{(n,i,j)}\in\bD^0_1(Z)$ for $n,i,j\in\omega$ such that $Z\setminus A_{(n,i)}=\bigcup_{j\in\omega}A_{(n,i,j)}$ and $A_{(n,i,j)}\cap A_{(n,i,j')}=\varnothing$ whenever $j\neq j'$. Define $f_{(n,i)}:Z\longrightarrow\omega$ for $n,i\in\omega$ by setting
$$
f_{(n,i)}(x)=\left\{
\begin{array}{ll} 0 & \textrm{if }x\in A_{(n,i)},\\
j+1 & \textrm{if }x\in A_{(n,i,j)}.
\end{array}
\right.
$$
Fix a bijection $\langle\cdot,\cdot\rangle:\omega\times\omega\longrightarrow\omega$ and define $f:Z\longrightarrow\omega^\omega$ by setting $f(x)(\langle n,i\rangle)=f_{(n,i)}(x)$ for $n,i\in\omega$. It is clear that condition $(\ref{kuratowskistrictlysmaller})$ holds.

To show that condition $(\ref{kuratowskigraphclosed})$ holds, pick $x_m\in Z$ for $m\in\omega$ such that $(x_m,f(x_m))\to(x,y)\in Z\times\omega^\omega$. We need to show that $y=f(x)$. So fix $k=\langle n,i\rangle\in\omega$. By the definition of $f$, in order to see that $y(k)=f(x)(k)$, we need to show that $y(k)=f_{(n,i)}(x)$. But $f_{(n,i)}(x_m)=f(x_m)(k)$ is eventually constant (with value $y(k)$), hence all but finitely many $x_m$ belong to $A_{(n,i)}$, or there exists $j\in\omega$ such that all but finitely many $x_m$ belong to $A_{(n,i,j)}$. Since $x_m\to x$ and these sets are all closed, it follows that $f_{(n,i)}(x)=y(k)$.

To show that condition $(\ref{kuratowskiinverseimage})$ holds for a given $n\in\omega$, pick $x\in A_n$. We need to find $V\in\bS^0_1(F)$ such that $(x,f(x))\in V\subseteq f^\ast[A_n]$. Pick $i\in\omega$ such that $x\in A_{(n,i)}$, then set $U=\{y\in\omega^\omega:y(\langle n,i\rangle)=0\}$, and observe that $U\in\bD^0_1(\omega^\omega)$. Using the definition of $f$, it is easy to realize that $V=(U\times Z)\cap F$ is as required.

Now assume that $\xi\geq 3$ and that the theorem holds for all $\xi'$ such that $2\leq\xi'<\xi$. Pick $A_{(n,i)}\in\bP^0_{\xi(n,i)}(Z)$ for $n,i\in\omega$ such that $A_n=\bigcup_{i\in\omega}A_{(n,i)}$, where $1\leq\xi(n,i)<\xi$. Then  pick $A_{(n,i,j)}\in\bD^0_{\xi(n,i)}$ for $n,i,j\in\omega$ such that $Z\setminus A_{(n,i)}=\bigcup_{j\in\omega}A_{(n,i,j)}$ and $A_{(n,i,j)}\cap A_{(n,i,j')}=\varnothing$ whenever $j\neq j'$.

Set $\Aa_{(n,i)}=\{A_{(n,i,j)}:j\in\omega\}\cup\{Z\setminus A_{(n,i,j)}:j\in\omega\}$ for $(n,i)\in\omega\times\omega$. By the inductive hypothesis, for each $(n,i)$ we can fix a $\xi(n,i)$-refining function $g_{(n,i)}:Z\longrightarrow\omega^\omega$ for $\Aa_{(n,i)}$. This means that the following conditions will hold, where $G_{(n,i)}=g_{(n,i)}^\ast[Z]$ denotes the graph of $g_{(n,i)}$:
\begin{enumerate}
 \item[(4)] $G_{(n,i)}$ is closed in $Z\times\omega^\omega$,
 \item[(5)] $g_{(n,i)}^\ast[A_{(n,i,j)}]\in\bD^0_1(G_{(n,i)})$ for every $j\in\omega$,
 \item[(6)] For all $k\in\omega$ there exists $\xi_k$ such that $1\leq\xi_k<\xi(n,i)$ and $(g_{(n,i)})_k^{-1}(j)\in\bP^0_{\xi_k}(Z)$ for every $j\in\omega$.
\end{enumerate}

Fix a bijection $\langle\cdot,\cdot,\cdot\rangle:\omega\times\omega\times\omega\longrightarrow\omega$ and define $g:Z\longrightarrow\omega^\omega$ by setting $g(x)(\langle n,i,j\rangle)=g_{(n,i)}(x)(j)$ for every $x\in Z$ and $n,i,j\in\omega$. Denote by $G=g^\ast[Z]$ the graph of $g$. Using condition $(4)$ and arguments as in the case $\xi=2$, one can show that $G$ is closed in $Z\times\omega^\omega$. Using condition $(5)$, it is easy to see that each $g^\ast[A_{(n,i,j)}]\in\bD^0_1(G)$. Since $G\setminus g^\ast[A_{(n,i)}]=\bigcup_{j\in\omega}g^\ast[A_{(n,i,j)}]$, by proceeding as in the proof of the case $\xi=2$ it is possible to obtain $h:G\longrightarrow\omega^\omega$ that satisfies the following conditions, where $H=h^\ast[G]$ denotes the graph of $h$:
\begin{enumerate}
 \item[(7)] $H$ is closed in $G\times\omega^\omega$,
 \item[(8)] $h^\ast[g^\ast[A_n]]\in\bS^0_1(H)$ for every $n\in\omega$,
 \item[(9)] $h_{\langle n,i\rangle}^{-1}(0)=g^\ast[A_{(n,i)}]$ and $h_{\langle n,i\rangle}^{-1}(j+1)=g^\ast[A_{(n,i,j)}]$ for every $n,i,j\in\omega$.
\end{enumerate} 

Finally, define $f:Z\longrightarrow\omega^{\omega+\omega}$ by setting 
$f_k(x)=g_k(x)$ and $f_{\omega+k}(x)=h_k(x,g(x))$ for every $x\in Z$ and $k\in\omega$. Also set $F=f^\ast[Z]$. Using conditions $(6)$ and $(9)$, it is easy to check that condition $(\ref{kuratowskistrictlysmaller})$ will hold. By identifying $\omega^{\omega+\omega}$ with $\omega^\omega\times\omega^\omega$ in the obvious way, $F$ can be identified with $H$. Therefore, condition $(\ref{kuratowskiinverseimage})$ holds by condition $(8)$. Furthermore, since condition $(7)$ holds and $G\times\omega^\omega$ is closed in $Z\times\omega^\omega\times\omega^\omega$, it follows that condition $(\ref{kuratowskigraphclosed})$ holds.
\end{proof}

\begin{corollary}\label{sigmaintoopenbijection}
Let $Z$ be a zero-dimensional Polish space, let $1\leq\xi<\omega_1$, and let $\Aa\subseteq\mathbf{\Sigma}^0_\xi(Z)$ be countable. Then there exists a zero-dimensional Polish space $W$ and a $\mathbf{\Sigma}^0_\xi$-measurable bijection $f:Z\longrightarrow W$ such that $f[A]\in\mathbf{\Sigma}^0_1(W)$ for every $A\in\Aa$.
\end{corollary}
\begin{proof}
The case $\xi=1$ is trivial, so assume that $\xi\geq 2$. Then the desired result follows from Theorem \ref{kuratowskitheorem}, by setting $W=F$ and $f=f^\ast$.
\end{proof}

\begin{corollary}[Kuratowski]\label{sigmaintoopen}
Let $(Z,\tau)$ be a zero-dimensional Polish space, let $1\leq\xi<\omega_1$, and let $\BB\subseteq\mathbf{\Sigma}^0_\xi(Z,\tau)$ be countable. Then there exists a zero-dimensional Polish topology $\sigma$ on the set $Z$ such that $\tau\subseteq\sigma\subseteq\mathbf{\Sigma}^0_\xi(Z,\tau)$ and $\BB\subseteq\sigma$.
\end{corollary}
\begin{proof}
Let $\UU$ be a countable base for $(Z,\tau)$. Let $f$ and $W$ be given by applying Corollary \ref{sigmaintoopenbijection} with $\Aa=\BB\cup\UU$, then define
$$
\sigma=\{f^{-1}[U]:U\in\bS^0_1(W)\}.
$$
It is straightforward to verify that $\sigma$ satisfies all of the desired properties.
\end{proof}

We conclude this section with the ``two-variable'' versions of Corollaries \ref{sigmaintoopen} and \ref{sigmaintoopenbijection} respectively. Corollary \ref{sigmaintoopenbijectioneta} will be needed in Section \ref{sectionexpansioncomposition}.

\begin{theorem}\label{sigmaintoopeneta}
Let $(Z,\tau)$ be a zero-dimensional Polish space, let $\xi,\eta<\omega_1$, and let $\Aa\subseteq\mathbf{\Sigma}^0_{1+\eta+\xi}(Z,\tau)$ be countable. Then there exists a zero-dimensional Polish topology $\sigma$ on the set $Z$ such that $\tau\subseteq\sigma\subseteq\mathbf{\Sigma}^0_{1+\eta}(Z,\tau)$ and $\Aa\subseteq\mathbf{\Sigma}^0_{1+\xi}(Z,\sigma)$.
\end{theorem}
\begin{proof}
By \cite[Lemma 13.3]{kechris}, it will be enough to consider the case $\Aa=\{A\}$. We will proceed by induction on $\xi$. The case $\xi=0$ is Corollary \ref{sigmaintoopen}. Now assume that $\xi>0$ and the result holds for every $\xi'<\xi$. Write $A=\bigcup_{n\in\omega}(Z\setminus A_n)$, where each $A_n\in\mathbf{\Sigma}^0_{1+\eta+\xi_n}(Z,\tau)$ for suitable $\xi_n<\xi$. By the inductive assumption, for each $n$, we can fix a zero-dimensional Polish topology $\sigma_n$ on the set $Z$ such that $\tau\subseteq\sigma_n\subseteq\mathbf{\Sigma}^0_{1+\eta}(Z,\tau)$ and $A_n\in\mathbf{\Sigma}^0_{1+\xi}(Z,\sigma_n)$. Using \cite[Lemma 13.3]{kechris} again, it is easy to check that the topology $\sigma$ on $Z$ generated by $\bigcup_{n\in\omega}\sigma_n$ is as desired.
\end{proof}

\begin{corollary}\label{sigmaintoopenbijectioneta}
Let $Z$ be a zero-dimensional Polish space, let $\xi,\eta<\omega_1$, and let $\Aa\subseteq\mathbf{\Sigma}^0_{1+\eta+\xi}(Z)$ be countable. Then there exists a zero-dimensional Polish space $W$ and a $\mathbf{\Sigma}^0_{1+\eta}$-measurable bijection $f:Z\longrightarrow W$ such that $f[A]\in\mathbf{\Sigma}^0_{1+\xi}(W)$ for every $A\in\Aa$.
\end{corollary}
\begin{proof}
The space $W$ is simply the set $Z$ with the finer topology given by Theorem \ref{sigmaintoopeneta}, while $f=\id_Z$.
\end{proof}

\section{Expansions: basic facts}\label{sectionexpansionbasic}

The following notion is essentially due to Wadge (see \cite[Chapter IV]{wadget}), and it is inspired by work of Kuratowski. It is one of the fundamental concepts needed to state our main result (see Definition \ref{louveauhierarchy}). Proposition \ref{expansionbasic}, whose straightforward proof is left to the reader, lists some of its most basic properties.

\begin{definition}\label{definitionexpansion}
Let $Z$ be a space, and let $\xi<\omega_1$. Given $\bG\subseteq\PP(Z)$, define
$$
\bG^{(\xi)}=\{f^{-1}[A]:A\in\bG\text{ and }f:Z\longrightarrow Z\text{ is $\mathbf{\Sigma}^0_{1+\xi}$-measurable}\}.
$$
We will refer to $\bG^{(\xi)}$ as an \emph{expansion} of $\bG$.
\end{definition}

\begin{proposition}\label{expansionbasic}
Let $Z$ be a space, let $\bG\subseteq\PP(Z)$, and let $\xi<\omega_1$.
\begin{itemize}
\item $\bG^{(\xi)}$ is continuously closed.
\item $\bG\subseteq\bG^{(\eta)}\subseteq\bG^{(\xi)}$ whenever $\eta\leq\xi$.
\item $\bG^{(0)}=\bG$ whenever $\bG$ is continuously closed.
\item $\widecheck{\bG^{(\xi)}}=\bGc^{(\xi)}$.
\end{itemize}
\end{proposition}

The following is the corresponding definition in the context of Hausdorff operations. Lemma \ref{movexihausdorff} below shows that this is in fact the ``right'' definition.
\begin{definition}\label{hausdorffexpansion}
Let $Z$ be a space, let $D\subseteq\PP(\omega)$, and let $\xi<\omega_1$. Define
$$
\bG_D^{(\xi)}(Z)=\{\HH_D(A_0,A_1,\ldots):A_n\in\mathbf{\Sigma}^0_{1+\xi}(Z)\text{ for every }n\in\omega\}.
$$
\end{definition}

As an example (that will be useful later), consider the following simple observation.

\begin{proposition}\label{expansiondifferences}
Let $1\leq\eta<\omega_1$. Then there exists $D\subseteq\PP(\omega)$ such that $\bG_D^{(\xi)}(Z)=\mathsf{D}_\eta(\mathbf{\Sigma}^0_{1+\xi}(Z))$ for every space $Z$ and every $\xi<\omega_1$.
\end{proposition}
\begin{proof}
This is proved like Proposition \ref{hausdorffdifferences} (in fact, the same $D$ will work).
\end{proof}

The following proposition shows that Definition \ref{hausdorffexpansion} actually fits in the context provided by Section \ref{sectionhausdorffclasses}.

\begin{proposition}\label{movexidown}
Let $D\subseteq\PP(\omega)$, and let $\xi<\omega_1$. Then there exists $E\subseteq\PP(\omega)$ such that $\bG_D^{(\xi)}(Z)=\bG_E(Z)$ for every space $Z$.
\end{proposition}
\begin{proof}
This is proved by combining Propositions \ref{hausdorffborel} and \ref{hausdorffcomposition}.
\end{proof}
\begin{corollary}\label{expansionhausdorffnonselfdual}
Let $Z$ be a zero-dimensional space in which $2^\omega$ embeds, let $D\subseteq\PP(\omega)$, and let $\xi<\omega_1$. Then $\bG_D^{(\xi)}(Z)\in\NSD(Z)$.
\end{corollary}
\begin{proof}
This is proved by combining Proposition \ref{movexidown} and Theorem \ref{addisontheorem}.
\end{proof}

The following useful result is the analogue of Lemma \ref{relativization} in the present context.

\begin{lemma}\label{expansionhausdorffrelativization}
Let $Z$ and $W$ be spaces, let $D\subseteq\PP(\omega)$, and let $\xi<\omega_1$.
\begin{enumerate}
\item\label{expansionpreimage} If $f:Z\longrightarrow W$ is continuous and $B\in\bG_D^{(\xi)}(W)$ then $f^{-1}[B]\in\bG_D^{(\xi)}(Z)$.
\item\label{expansionmeasurablepreimage} If $f:Z\longrightarrow W$ is $\mathbf{\Sigma}^0_{1+\xi}$-measurable and $B\in\bG_D(W)$ then $f^{-1}[B]\in\bG_D^{(\xi)}(Z)$.
\item\label{expansionhomeomorphism} If $h:Z\longrightarrow W$ is a homeomorphism then $A\in\bG_D^{(\xi)}(Z)$ iff $h[A]\in\bG_D^{(\xi)}(W)$.
\item\label{expansionsubspace} Assume that $W\subseteq Z$. Then $B\in\bG_D^{(\xi)}(W)$ iff there exists $A\in\bG_D^{(\xi)}(Z)$ such that $B=A\cap W$.
\end{enumerate}
\end{lemma}
\begin{proof}
This is a straightforward consequence of Proposition \ref{relativizationsettheoretic}.
\end{proof}

Finally, we show that $\Ha(Z)$ is closed under expansions (see Proposition \ref{expansionsarehausdorff}). We will need the following result, which is another variation on the theme of Kuratowski's transfer theorem. Notice however that, at this point, we do not know that $\bG^{(\xi)}$ is a non-selfdual Wadge class whenever $\bG$ is. That this is true will follow from Theorem \ref{main}.

\begin{lemma}\label{expansionbijectionhausdorff}
Let $Z$ be a zero-dimensional Polish space, let $D\subseteq\PP(\omega)$, let $\xi<\omega_1$, and let $\Aa\subseteq\bG_D^{(\xi)}(Z)$ be countable. Then there exists a zero-dimensional Polish space $W$ and a $\mathbf{\Sigma}^0_{1+\xi}$-measurable bijection $f:Z\longrightarrow W$ such that $f[A]\in\bG_D(W)$ for every $A\in\Aa$.
\end{lemma}
\begin{proof}
If $\Aa=\varnothing$ the desired result is trivial, so assume that $\Aa\neq\varnothing$. Let $\Aa=\{A_m:m\in\omega\}$ be an enumeration. Given $m\in\omega$, fix $B_{(m,n)}\in\mathbf{\Sigma}^0_{1+\xi}(Z)$ for $n\in\omega$ such that $A_m=\HH_D(B_{(m,0)},B_{(m,1)},\ldots)$. Define $\BB=\{B_{(m,n)}:m,n\in\omega\}$. By Corollary \ref{sigmaintoopenbijection}, we can fix a zero-dimensional Polish space $W$ and a $\mathbf{\Sigma}^0_{1+\xi}$-measurable bijection $f:Z\longrightarrow W$ such that $f[B]\in\mathbf{\Sigma}^0_1(W)$ for every $B\in\BB$. It remains to observe that
$$
f[A_m]=f[\HH_D(B_{(m,0)},B_{(m,1)},\ldots)]=\HH_D(f[B_{(m,0)}],f[B_{(m,1)}],\ldots)\in\bG_D(W)
$$
for every $m\in\omega$, where the second equality follows from Proposition \ref{relativizationsettheoretic}.\ref{homeomorphismsettheoretic}.
\end{proof}

\begin{lemma}\label{movexihausdorff}
Let $Z$ be an uncountable zero-dimensional Polish space, let $D\subseteq\PP(\omega)$, and let $\xi<\omega_1$. Then $\bG_D(Z)^{(\xi)}=\bG_D^{(\xi)}(Z)$.
\end{lemma}
\begin{proof}
The inclusion $\bG_D(Z)^{(\xi)}\subseteq\bG_D^{(\xi)}(Z)$ follows from Lemma \ref{expansionhausdorffrelativization}.\ref{expansionmeasurablepreimage}. In order to prove the other inclusion, pick $A\in\bG_D^{(\xi)}(Z)$. By Lemma \ref{expansionbijectionhausdorff}, we can fix a zero-dimensional Polish space $W$ and a $\mathbf{\Sigma}^0_{1+\xi}$-measurable bijection $f:Z\longrightarrow W$ such that $f[A]\in\bG_D(W)$. Since $2^\omega$ embeds in $Z$ and $W$ is zero-dimensional, using Lemma \ref{hausdorffrelativization}.\ref{hausdorffrelativizationhomeomorphism} we can assume without loss of generality that $W$ is a subspace of $Z$, so that $f:Z\longrightarrow Z$. By Lemma \ref{hausdorffrelativization}.\ref{hausdorffrelativizationsubspace}, we can fix $B\in\bG_D(Z)$ such that $B\cap W=f[A]$. It is easy to check that $A=f^{-1}[B]$, which concludes the proof.
\end{proof}

\begin{proposition}\label{expansionsarehausdorff}
Let $Z$ be an uncountable zero-dimensional Polish space, let $\xi<\omega_1$, and let $\bG\in\Ha(Z)$. Then $\bG^{(\xi)}\in\Ha(Z)$.
\end{proposition}
\begin{proof}
This follows from Lemma \ref{movexihausdorff} and Proposition \ref{movexidown}.
\end{proof}

\section{Expansions: relativization}\label{sectionexpansionrelativization}

In this section we collect some useful results, showing that expansions interact in the expected way with the machinery of relativization. Lemma \ref{expansionbijectionrelativization} is yet another variation on the theme of Kuratowski's transfer theorem. These facts will be needed in Section \ref{sectionexpansionmain}.

\begin{lemma}\label{relativizationexpansion}
Let $\bS$ be a nice topological pointclass, and assume that $\Det(\bS(\omega^\omega))$ holds. Let $Z$ and $W$ be zero-dimensional Polish spaces, let $\bG\in\NSDS(\omega^\omega)$, let $\xi<\omega_1$, and let $f:Z\longrightarrow W$ be $\mathbf{\Sigma}^0_{1+\xi}$-measurable. Then $f^{-1}[A]\in\bG^{(\xi)}(Z)$ for every $A\in\bG(W)$.
\end{lemma}
\begin{proof}
Pick $A\in\bG(W)$. To see that $f^{-1}[A]\in\bG^{(\xi)}(Z)$, we have to show that $g^{-1}[f^{-1}[A]]\in\bG^{(\xi)}$ for every continuous $g:\omega^\omega\longrightarrow Z$. So pick such a $g$. Since $A\in\bG(W)$, by condition $(\ref{existsemb})$ of Lemma \ref{relativizationprelim}, we can fix an embedding $j:W\longrightarrow\omega^\omega$ and $B\in\bG$ such that $A=j^{-1}[B]$. The proof is concluded by observing that
$$
g^{-1}[f^{-1}[A]]=g^{-1}[f^{-1}[j^{-1}[B]]]=(j\circ f\circ g)^{-1}[B]\in\bG^{(\xi)}
$$
by the definition of expansion, since $j\circ f\circ g$ is $\mathbf{\Sigma}^0_{1+\xi}$-measurable by Lemma \ref{composefunctions}.
\end{proof}

\begin{lemma}\label{movexirelativization}
Let $\bS$ be a nice topological pointclass, and assume that $\Det(\bS(\omega^\omega))$ holds. Let $Z$ be an uncountable zero-dimensional Polish space, let $\bG\in\NSDS(\omega^\omega)$, and let $\xi<\omega_1$. Then $\bG(Z)^{(\xi)}=\bG^{(\xi)}(Z)$.
\end{lemma}
\begin{proof}
In order to prove the inclusion $\subseteq$, pick $A\in\bG(Z)^{(\xi)}$. We have to show that $g^{-1}[A]\in\bG^{(\xi)}$ for every continuous function $g:\omega^\omega\longrightarrow Z$. So pick such a $g$. By the definition of expansion, there exists a $\bS^0_{1+\xi}$-measurable function $f:Z\longrightarrow Z$ and $B\in\bG(Z)$ such that $f^{-1}[B]=A$. By condition $(\ref{existsemb})$ of Lemma \ref{relativizationprelim}, there exists an embedding $j:Z\longrightarrow\omega^\omega$ and $C\in\bG$ such that $B=j^{-1}[C]$. Using Lemma \ref{composefunctions}, one sees that $j\circ f\circ g:\omega^\omega\longrightarrow\omega^\omega$ is a $\bS^0_{1+\xi}$-measurable function. Therefore $g^{-1}[A]=(j\circ f\circ g)^{-1}[C]\in\bG^{(\xi)}$ by the definition of expansion.

In order to prove the inclusion $\supseteq$, pick $A\in\bG^{(\xi)}(Z)$. By \cite[Theorem 7.8]{kechris}, we can fix an embedding $i:Z\longrightarrow\omega^\omega$ such that $i[Z]$ is closed in $\omega^\omega$. Therefore, by \cite[Proposition 2.8]{kechris}, there exists a retraction $\rho:\omega^\omega\longrightarrow i[Z]$. Observe that $i[A]\in\bG^{(\xi)}(i[Z])$ by Lemma \ref{relativization}.\ref{relativizationhomeo}. Set $A'=\rho^{-1}[i[A]]$, and observe that $A'\in\bG^{(\xi)}$. Therefore, there exist a $\bS^0_{1+\xi}$-measurable function $f:\omega^\omega\longrightarrow\omega^\omega$ and $B'\in\bG$ such that $f^{-1}[B']=A'$. Since $Z$ is uncountable, we can fix an embedding $j:\omega^\omega\longrightarrow Z$. By Lemmas \ref{relativization}.\ref{relativizationhomeo}, \ref{relativization}.\ref{relativizationbaire} and \ref{relativizationsubspace}, there exists $B\in\bG(Z)$ such that $B\cap j[\omega^\omega]=j[B']$. The proof is concluded by observing that $A=(j\circ f\circ i)^{-1}[B]$, and that $j\circ f\circ i$ is $\bS^0_{1+\xi}$-measurable by Lemma \ref{composefunctions}.
\end{proof}

\begin{lemma}\label{expansionbijectionrelativization}
Let $Z$ be a zero-dimensional Polish space, let $\bG\subseteq\PP(\omega^\omega)$, and let $\xi<\omega_1$. Assume that $\Aa\subseteq\bG(Z)^{(\xi)}$ and $\BB\subseteq\bS^0_{1+\xi}(Z)$ are countable. Then there exists a zero-dimensional Polish space $W$ and a $\mathbf{\Sigma}^0_{1+\xi}$-measurable bijection $f:Z\longrightarrow W$ such that $f[A]\in\bG(W)$ for every $A\in\Aa$ and $f[B]\in\bS^0_1(W)$ for every $B\in\BB$.
\end{lemma}
\begin{proof}
If $\Aa=\varnothing$ then the desired result is Corollary \ref{sigmaintoopenbijection}, so assume that $\Aa\neq\varnothing$. Let $\Aa=\{A_n:n\in\omega\}$ be an enumeration. By the definition of expansion, we can fix $\mathbf{\Sigma}^0_{1+\xi}$-measurable functions $g_n:Z\longrightarrow Z$ and $B_n\in\bG(Z)$ for $n\in\omega$ such that $A_n=g_n^{-1}[B_n]$. Fix a countable base $\UU$ for $Z$, and set
$$
\CC=\{g_n^{-1}[U]:n\in\omega,U\in\UU\}\cup\BB.
$$
By Corollary \ref{sigmaintoopenbijection}, there exists a zero-dimensional Polish space $W$ and $\mathbf{\Sigma}^0_{1+\xi}$-measurable bijection $f:Z\longrightarrow W$ such that $f[C]\in\mathbf{\Sigma}^0_1(W)$ for every $C\in\CC$. We claim that $f[A_n]\in\bG(W)$ for every $n\in\omega$. So pick $n\in\omega$. Since $f[A_n]=f[g_n^{-1}[B_n]]=(g_n\circ f^{-1})^{-1}[B_n]$, by Lemma \ref{relativization}.\ref{relativizationpreimage}, it will be enough to show that $g_n\circ f^{-1}$ is continuous. This follows from the fact that $(g_n\circ f^{-1})^{-1}[U]=f[g_n^{-1}[U]]\in\mathbf{\Sigma}^0_1(W)$ for every $U\in\UU$.
\end{proof}

\section{Level: basic facts}\label{sectionlevelbasic}

In this section we will introduce the notion of level, which is one of the fundamental concepts involved in our main result (see Definition \ref{louveauhierarchy}). We will need the following preliminary definition. Both notions are taken from \cite{louveausaintraymondf}, which was however limited to the Borel context (see also \cite[Section 7.3.4]{louveaub}).\footnote{\,In \cite{louveausaintraymondf}, the notation $\mathbf{\Delta}_{1+\xi}^0\text{-}\PU$ is used instead of $\PU_\xi$, and $\lambda_\mathsf{C}$ is used instead of $\ell$.}

\begin{definition}[Louveau, Saint-Raymond]\label{pudefinition}
Let $Z$ be a space, let $\bG\subseteq\PP(Z)$, and let $\xi<\omega_1$. Define $\PU_\xi(\bG)$ to be the collection of all sets of the form
$$
\bigcup_{n\in\omega}(A_n\cap V_n),
$$
where each $A_n\in\bG$, each $V_n\in\mathbf{\Delta}_{1+\xi}^0(Z)$, the $V_n$ are pairwise disjoint, and $\bigcup_{n\in\omega}V_n=Z$. A set in this form is called a \emph{partitioned union} of sets in $\bG$.
\end{definition}

Notice that the sets $V_n$ in the above definition are not required to be non-empty. The following proposition, whose straightforward proof is left to the reader, collects the most basic facts about partitioned unions.

\begin{proposition}\label{pubasic}
Let $Z$ be a space, let $\bG\subseteq\PP(Z)$, and let $\xi<\omega_1$.
\begin{enumerate}
\item\label{pubasiccc} If $\bG$ is continuously closed then $\PU_\xi(\bG)$ is continuously closed.
\item\label{pubasicincreasing} $\bG\subseteq\PU_\eta(\bG)\subseteq\PU_\xi(\bG)$ whenever $\eta\leq\xi$.
\item\label{pubasicwadge} $\PU_0(\bG)=\bG$ whenever $\bG$ is a Wadge class in $Z$.
\item\label{pubasiccheck} $\widecheck{\PU_\xi(\bG)}=\PU_\xi(\bGc)$.
\item\label{pubasicidempotent} $\PU_\xi(\PU_\xi(\bG))=\PU_\xi(\bG)$.
\end{enumerate}
\end{proposition}

\begin{definition}[Louveau, Saint-Raymond]
Let $Z$ be a space, let $\bG\subseteq\PP(Z)$, and let $\xi<\omega_1$. Define
\begin{itemize}
\item $\ell(\bG)\geq\xi$ if $\PU_\xi(\bG)=\bG$,
\item $\ell(\bG)=\xi$ if $\ell(\bG)\geq\xi$ and $\ell(\bG)\not\geq\xi+1$,
\item $\ell(\bG)=\omega_1$ if $\ell(\bG)\geq\eta$ for every $\eta<\omega_1$.
\end{itemize}
We refer to $\ell(\bG)$ as the \emph{level} of $\bG$.
\end{definition}

Notice that, by Proposition \ref{pubasic}.\ref{pubasicwadge}, $\ell(\bG)\geq 0$ for every Wadge class $\bG$. Using \cite[Theorem 22.4 and Exercise 37.3]{kechris}, one sees that the following hold for every uncountable Polish space $Z$:
\begin{itemize}
\item $\ell(\{\varnothing\})=\ell(\{Z\})=\omega_1$,
\item $\ell(\bS^0_{1+\xi}(Z))=\ell(\bP^0_{1+\xi}(Z))=\xi$ whenever $\xi<\omega_1$,
\item $\ell(\bS^1_n(Z))=\ell(\bP^1_n(Z))=\omega_1$ whenever $1\leq n<\omega$.
\end{itemize}
In fact, the classes of uncountable level can be characterized as those closed under Borel preimages (see Corollary \ref{uncountablelevel} for a more precise statement).

We remark that it is not clear at this point whether for every non-selfdual Wadge class $\bG$ there exists $\xi\leq\omega_1$ such that $\ell(\bG)=\xi$.\footnote{\,It is conceivable that $\PU_\xi(\bG)=\bG$ for all $\xi<\eta$, where $\eta$ is a limit ordinal, while $\PU_\eta(\bG)\neq\bG$.} In Section \ref{sectionlevelhard}, we will show that this is in fact the case (see Corollary \ref{everyclasshasalevel}).

The following simple proposition shows that the notion of level becomes rather trivial when the ambient space is countable.

\begin{proposition}\label{levelcountable}
Let $Z$ be a countable space, and let $\{\varnothing,Z\}\subseteq\bG\subseteq\PP(Z)$. Assume that $\ell(\bG)\geq 1$. Then $\bG=\PP(Z)$.	
\end{proposition}
\begin{proof}
Use the fact that $\{\{x\}:x\in Z\}$ is a countable partition of $Z$ consisting of $\bD^0_2$ sets.	
\end{proof}

We conclude this section with another simple result, which shows that classes of high level are guaranteed to have certain closure properties. Its straightforward proof is left to the reader.

\begin{lemma}\label{levelintersection}
Let $Z$ be a space, let $\bG\subseteq\PP(Z)$ be such that $\varnothing\in\bG$, and let $\xi<\omega_1$. Assume that $\ell(\bG)\geq\xi$. Then $A\cap V\in\bG$ whenever $A\in\bG$ and $V\in\bD^0_{1+\xi}(Z)$.
\end{lemma}

\section{Expansions: the main theorem}\label{sectionexpansionmain}

The main result of this section is Theorem \ref{expansiontheorem}, which clarifies the crucial connection between level and expansion. This result can be traced back to \cite[Th\'eor\`{e}me 8]{louveausaintraymondf}, but the proof given here is essentially the same as \cite[proof of Theorem 7.3.9.ii]{louveaub}. Both of these are however limited to the Borel context.

\begin{theorem}\label{expansiontheorem}
Let $\bS$ be a nice topological pointclass, and assume that $\Det(\bS(\omega^\omega))$ holds. Let $Z$ be an uncountable zero-dimensional Polish space, and let $\xi<\omega_1$. Then, for every $\bG\in\NSDS(Z)$, the following conditions are equivalent:
\begin{enumerate}
\item\label{levelgeqxi} $\ell(\bG)\geq\xi$,
\item\label{expansionofsomething} $\bG=\bL^{(\xi)}$ for some $\bL\in\NSDS(Z)$.
\end{enumerate}
\end{theorem}
\begin{proof}
First we will prove that the implication $(\ref{levelgeqxi})\rightarrow (\ref{expansionofsomething})$ holds. Pick $\bG(Z)\in\NSDS(Z)$, where $\bG\in\NSDS(\omega^\omega)$. Assume that $\ell(\bG(Z))\geq\xi$. Let $\bL(Z)\in\NSDS(Z)$, where $\bL\in\NSDS(\omega^\omega)$, be $\subseteq$-minimal with the property that $\bG(Z)\subseteq\bL(Z)^{(\xi)}$. We claim that $\bG(Z)=\bL(Z)^{(\xi)}$. Assume, in order to get a contradiction, that $\bL(Z)^{(\xi)}\nsubseteq\bG(Z)$. It follows from Lemma \ref{wadgelemma} that $\bG(Z)\subseteq\bLc(Z)^{(\xi)}$. Fix $A\subseteq Z$ such that $\bG(Z)=A\wc$, and observe that $\{A,Z\setminus A\}\subseteq\bL(Z)^{(\xi)}$. Then, by Corollary \ref{expansionbijectionrelativization}, we can fix a zero-dimensional Polish space $W$ and a $\mathbf{\Sigma}^0_{1+\xi}$-measurable bijection $f:Z\longrightarrow W$ such that $\{f[A],f[Z\setminus A]\}\subseteq\bL(W)$.

Next, we will show that $f[A]$ is selfdual in $W$. Assume, in order to get a contradiction, that this is not the case. Then we can fix $\bP\in\NSDS(\omega^\omega)$ such that $f[A]\wc=\bP(W)$. Notice that $\bP(W)\subseteq\bL(W)$. Furthermore $W\setminus f[A]=f[Z\setminus A]\in\bL(W)$, hence $\bPc(W)\subseteq\bL(W)$. Since $\bP(W)$ is non-selfdual, it follows that $\bP(W)\subsetneq\bL(W)$. Therefore, $\bP(Z)\subsetneq\bL(Z)$ by Theorem \ref{orderisomorphism}. On the other hand, Lemmas \ref{relativizationexpansion} and \ref{movexirelativization} show that $A=f^{-1}[f[A]]\in\bP^{(\xi)}(Z)=\bP(Z)^{(\xi)}$. Hence $\bG(Z)\subseteq\bP(Z)^{(\xi)}$, which contradicts the minimality of $\bL(Z)$.

Since $f[A]$ is selfdual in $W$, by Corollary \ref{selfdualcorollary}, we can fix $A_n\subseteq W$, pairwise disjoint $V_n\in\mathbf{\Delta}^0_1(W)$, and $\bG_n\in\NSDS(\omega^\omega)$ for $n\in\omega$ such that $\bigcup_{n\in\omega}V_n=W$,
$$
f[A]=\bigcup_{n\in\omega}(A_n\cap V_n),
$$
and $A_n\in\bG_n(W)\subsetneq\bL(W)$ for each $n$. Notice that $\bG_n(Z)\subsetneq\bL(Z)$ for each $n$ by Theorem \ref{orderisomorphism}, hence $\bG(Z)\nsubseteq\bG_n(Z)^{(\xi)}$ for each $n$ by the minimality of $\bL(Z)$. It follows from Lemma \ref{wadgelemma} that $\bGc_n(Z)^{(\xi)}\subseteq\bG(Z)$ for each $n$.

Set $B_n=W\setminus A_n\in\bGc_n(W)$ for $n\in\omega$. Observe that $f^{-1}[B_n]\in\bGc_n^{(\xi)}(Z)=\bGc_n(Z)^{(\xi)}\subseteq\bG(Z)$ for each $n$ by Lemmas \ref{relativizationexpansion} and \ref{movexirelativization}. Furthermore, it is clear that $f^{-1}[V_n]\in\mathbf{\Delta}^0_{1+\xi}(Z)$ for each $n$ and $\bigcup_{n\in\omega}f^{-1}[V_n]=Z$. In conclusion, since $W\setminus f[A]=\bigcup_{n\in\omega}(B_n\cap V_n)$, we see that
$$
Z\setminus A=\bigcup_{n\in\omega}(f^{-1}[B_n]\cap f^{-1}[V_n])\in\PU_\xi(\bG(Z))=\bG(Z),
$$
where the last equality uses the assumption that $\ell(\bG(Z))\geq\xi$. This contradicts the fact that $\bG(Z)$ is non-selfdual.

In order to show that $(\ref{expansionofsomething})\rightarrow(\ref{levelgeqxi})$, assume that $\bG,\bL\in\NSDS(\omega^\omega)$ are such that $\bL(Z)^{(\xi)}=\bG(Z)$. Pick $A_n\in\bG(Z)$ and pairwise disjoint $V_n\in\mathbf{\Delta}^0_{1+\xi}(Z)$ for $n\in\omega$ such that $\bigcup_{n\in\omega}V_n=Z$. We need to show that $\bigcup_{n\in\omega}(A_n\cap V_n)\in\bG(Z)$. By Lemma \ref{expansionbijectionrelativization}, we can fix a zero-dimensional Polish space $W$ and a $\mathbf{\Sigma}^0_{1+\xi}$-measurable bijection $f:Z\longrightarrow W$ such that each $f[A_n]\in\bL(W)$ and each $f[V_n]\in\mathbf{\Delta}^0_1(W)$. Set $B=\bigcup_{n\in\omega}(f[A_n]\cap f[V_n])$, and observe that $B\in\PU_0(\bL(W))=\bL(W)$. It follows from Lemmas \ref{relativizationexpansion} and \ref{movexirelativization} that 
\[
\bigcup_{n\in\omega}(A_n\cap V_n)=f^{-1}[B]\in\bL^{(\xi)}(Z)=\bL(Z)^{(\xi)}=\bG(Z).
\qedhere
\]
\end{proof}

\begin{corollary}\label{levelinvariance}
Let $\bS$ be a nice topological pointclass, and assume that $\Det(\bS(\omega^\omega))$ holds. Let $Z$ and $W$ be uncountable zero-dimensional Polish spaces, let $\xi<\omega_1$, and let $\bG\in\NSDS(\omega^\omega)$. Then $\ell(\bG(Z))\geq\xi$ iff $\ell(\bG(W))\geq\xi$.
\end{corollary}
\begin{proof}
We will only prove the left-to-right implication, as the other one can be proved similarly. Assume that $\ell(\bG(Z))\geq\xi$. Then, by Theorem \ref{expansiontheorem}, there exists $\bL\in\NSDS(\omega^\omega)$ such that $\bL(Z)^{(\xi)}=\bG(Z)$. Therefore $\bL^{(\xi)}(Z)=\bG(Z)$ by Lemma \ref{movexirelativization}. Notice that $\bL^{(\xi)}$ is non-selfdual, otherwise $\bG(Z)$ would be selfdual. Furthemore, $\bL^{(\xi)}$ is continuously closed by Proposition \ref{expansionbasic}. So $\bL^{(\xi)}\in\NSDS(\omega^\omega)$ by Lemma \ref{nonselfdualcontinuouslyclosed}. Hence it is possible to apply Theorem \ref{orderisomorphism}, which yields $\bL^{(\xi)}(W)=\bG(W)$. By applying Corollary \ref{movexirelativization} again, we see that $\bL(W)^{(\xi)}=\bG(W)$, which implies $\ell(\bG(W))\geq\xi$ by Theorem \ref{expansiontheorem}.
\end{proof}

\section{Level: every non-selfdual Wadge class has one}\label{sectionlevelhard}

The main result of this section states that every non-selfdual Wadge classes has an exact level (see Corollary \ref{everyclasshasalevel}). This fact will be needed in the proof of our main result (Theorem \ref{main}), and its proof requires the sharp analysis given by Theorem \ref{kuratowskitheorem}, as well as the machinery of relativization. The Borel version of Corollary \ref{everyclasshasalevel} appears as \cite[Proposition 7.3.7]{louveaub}, however we believe that the proof given there is not correct. The proof given here is inspired by the proof of \cite[Theorem 7.1.9]{louveaub}.

We will need the basic theory of trees. For a comprehensive treatment, we refer to \cite[Section 2]{kechris}. However, we will remind the reader of the necessary notions as follows. Given a set $A$, a \emph{tree} on $A$ is a subset $T$ of $A^{<\omega}$ such that $s\re n\in T$ for all $s\in T$ and $n\leq m$, where $m$ is the domain of $s$. An \emph{infinite branch} of $T$ is a function $f:\omega\longrightarrow A$ such that $f\re n\in T$ for every $n\in\omega$. A \emph{terminal node} of $T$ is an element $s\in T$ such that $s\hat{\,\,\,}a\notin T$ for every $a\in A$. A tree is \emph{well-founded} if it has no infinite branches. If $A$ is countable and $T$ is well-founded then there exists a unique \emph{rank function} $\rho_T:T\longrightarrow\omega_1$ such that
$$
\rho_T(s)=\supi\{\rho_T(t)+1:t\in T\text{ and }s\subsetneq t\}.
$$
for every $s\in T$ and $\rho_T(s)=0$ for every terminal node $s\in T$. The \emph{rank} of a well-founded tree $T$ is defined as follows 
$$
\rho(T)=\left\{
\begin{array}{ll} 0 & \textrm{if }T=\varnothing,\\
\rho_T(\varnothing)+1 & \textrm{if }T\neq\varnothing.
\end{array}
\right.
$$

Given a tree $T$ on a set $A$ and $s\in A^{<\omega}$, define
$$
T/s=\{t\in A^{<\omega}:s\hat{\,\,\,} t\in T\}.
$$
Notice that $T/s=\varnothing$ whenever $s\notin T$. For our purposes, the fundamental property of $T/s$ is that if $T$ is well-founded then $T/s$ is well-founded and $\rho(T/s)<\rho(T)$ whenever $s\in T$ and $s\neq\varnothing$.

\begin{theorem}\label{everyclasshasalevelprelim}
Let $\bS$ be a nice topological pointclass, and assume that $\Det(\bS(\omega^\omega))$ holds. Let $Z$ be a zero-dimensional Polish space, let $\bG\in\NSDS(\omega^\omega)$, and let $\eta<\omega_1$ be a limit ordinal. Assume that $\ell(\bG(Z))\geq\xi$ for every $\xi<\eta$. Then $\ell(\bG(Z))\geq\eta$.
\end{theorem}

\begin{proof}
By \cite[Theorem 7.8]{kechris} and Lemma \ref{relativization}.\ref{relativizationhomeo}, we can assume without loss of generality that $Z$ is a closed subspace of $\omega^\omega$. By Proposition \ref{levelcountable}, we can also assume that $Z$ is uncountable. Given a subspace $W$ of $\omega^\omega$, we will use the notation $W_s=W\cap\Ne_s$ for $s\in\omega^{<\omega}$. Given a subspace $W$ of $\omega^\omega$, a function $f:W\longrightarrow\omega^\omega$ and $\VV\subseteq\PP(W)$, define 
$$
T_{(f,\VV)}=\{(s,t)\in\omega^{<\omega}\times\omega^{<\omega}:(W_s\times\Ne_t)\cap F\nsubseteq f^\ast[V]\text{ for all }V\in\VV\},
$$
where $F=f^\ast[W]$ denotes the graph of $f$. It is clear that $T_{(f,\VV)}$ is a subtree of $\omega^{<\omega}\times\omega^{<\omega}$, where we identify $\omega^{<\omega}\times\omega^{<\omega}$ with $(\omega\times\omega)^{<\omega}$ in the natural way. Furthermore, it is a simple exercise to check that if $W$ is closed in $\omega^\omega$, $\VV$ is a cover of $W$, and $f^\ast[V]\in\bS^0_1(F)$ for every $V\in\VV$, then $T_{(f,\VV)}$ is well-founded. In particular, this will be the case when $W=Z$, $\VV\subseteq\bS^0_\eta(Z)$ is a countable partition of $Z$ and $f$ is an $\eta$-refining function for $\VV$. Since the existence of such a function is guaranteed by Theorem \ref{kuratowskitheorem}, in order to conclude the proof, it will be sufficient to show that the following condition holds for every $\xi<\omega_1$.
\begin{itemize}
\item[$\circledast(\xi)$] Let $W$ be a non-empty Borel subspace of $\omega^\omega$, let $\VV\subseteq\bS^0_\eta(W)$ be a countable partition of $W$, and let $f:W\longrightarrow\omega^\omega$ be an $\eta$-refining function for $\VV$ such that $T_{(f,\VV)}$ is well-founded and has rank at most $\xi$. Then $\bigcup_{V\in\VV}(\psi(V)\cap V)\in\bG(W)$ for every $\psi:\VV\longrightarrow\bG(W)$.
\end{itemize}

First we will show that $\circledast(0)$ holds. In this case $T_{(f,\VV)}=\varnothing$, hence there exists $V\in\VV$ such that $F=(W\times\omega^\omega)\cap F\subseteq f^\ast[V]$. Therefore $\VV=\{W\}$. It is clear that the desired conclusion holds in this case.

Now assume that $0<\xi<\omega_1$ and that $\circledast(\xi')$ holds whenever $\xi'<\xi$. Fix $W$, $\VV$ and $f$ as in the statement of $\circledast(\xi)$. Pick $\psi:\VV\longrightarrow\bG(W)$. Set 
$$
I=\{(m,n)\in\omega\times\omega:m=x(0)\text{ and }n=f_0(x)\text{ for some }x\in W\}.
$$
Given $(m,n)\in I$, make the following definitions:
\begin{itemize}
\item $W_{(m,n)}=W_{(m)}\cap f_0^{-1}(n)$,
\item $\VV_{(m,n)}=\{V\cap W_{(m,n)}:V\in\VV\}\setminus\{\varnothing\}$,
\item $f_{(m,n)}:W_{(m,n)}\longrightarrow\omega^\omega$ is the function obtained by setting $f_{(m,n)}(x)(k)=f(x)(k+1)$ for every $k\in\omega$,
\item $\psi_{(m,n)}:\VV_{(m,n)}\longrightarrow\bG(W_{(m,n)})$ is the unique function such that $\psi_{(m,n)}(V\cap W_{(m,n)})=\psi(V)\cap W_{(m,n)}$ for every $V\in\VV$ such that $V\cap W_{(m,n)}\neq\varnothing$.
\end{itemize}

It is straightforward to check that 
$$
T_{(f,\VV)}/((m),(n))=T_{(f_{(m,n)},\VV_{(m,n)})},
$$
hence the right-hand side has rank strictly smaller than $\xi$. It follows from the inductive hypothesis that $\bigcup_{V\in\VV_{(m,n)}}(\psi_{(m,n)}(V)\cap V)\in\bG(W_{(m,n)})$, so by Lemma \ref{relativizationsubspace} we can fix $A_{(m,n)}\in\bG(W)$ such that this union is equal to $A_{(m,n)}\cap W_{(m,n)}$. In conclusion,
$$
\bigcup_{V\in\VV}(\psi(V)\cap V)=\bigcup_{(m,n)\in I}\bigcup_{V\in\VV_{(m,n)}}(\psi_{(m,n)}(V)\cap V)=\bigcup_{(m,n)\in I}(A_{(m,n)}\cap W_{(m,n)}).
$$

Let $1\leq\eta_0<\eta$ be such that $f_0^{-1}(n)\in\bP^0_{\eta_0}(W)$ for every $n\in\omega$, as in the definition of $\eta$-refining function. Since each $W_{(m,n)}\in\bD^0_{\eta_0+1}(W)$, in order to show that the right-hand side of the above equation belongs to $\bG(W)$ it will be enough to show that $\ell(\bG(W))\geq\xi$ for every $\xi<\eta$. This can be easily achieved by viewing $W$ as a subspace of $Z$ (which can be done since $Z$ is uncountable), and using the corresponding assumption on $Z$ in conjunction with Lemma \ref{relativizationsubspace}.
\end{proof}

\begin{corollary}\label{everyclasshasalevel}
Let $\bS$ be a nice topological pointclass, and assume that $\Det(\bS(\omega^\omega))$ holds. Let $Z$ be a zero-dimensional Polish space, and let $\bG\in\NSDS(Z)$. Then there exists $\xi\leq\omega_1$ such that $\ell(\bG)=\xi$.	
\end{corollary}

\section{Expansions: composition}\label{sectionexpansioncomposition}

In this section we will show that the composition of two expansions can be obtained as a single expansion (see Theorem \ref{composeexpansions}). While this fact is of independent interest, our reason for proving it is Corollary \ref{expansionbigger}, which will be needed in the proof of Theorem \ref{main}.

\begin{lemma}\label{composefunctions}
Let $Z$, $W$ and $T$ be spaces, and let $\xi,\eta<\omega_1$. Assume that $f:Z\longrightarrow W$ is $\bS^0_{1+\xi}$-measurable and $g:W\longrightarrow T$ is $\bS^0_{1+\eta}$-measurable. Then $g\circ f$ is $\bS^0_{1+\xi+\eta}$-measurable.
\end{lemma}
\begin{proof}
It will be enough to prove that $f^{-1}[A]\in\bS^0_{1+\xi+\eta}(Z)$ for every $A\in\bS^0_{1+\eta}(W)$. We will proceed by induction on $\eta$. The case $\eta=0$ is trivial. Now assume that the claim holds for every $\eta'<\eta$. Pick $A\in\bS^0_{1+\eta}(W)$, and let $A_n\in\bS^0_{1+\eta_n}(W)$ for $n\in\omega$ be such that $A=\bigcup_{n\in\omega}(W\setminus A_n)$, where each $\eta_n<\eta$. Then
$$
f^{-1}[A]=\bigcup_{n\in\omega}(Z\setminus f^{-1}[A_n])\in\bS^0_{1+\xi+\eta}(Z),
$$
because each $f^{-1}[A_n]\in\bS^0_{1+\xi+\eta_n}(Z)$ by the inductive assumption.
\end{proof}

\begin{theorem}\label{composeexpansions}
Let $\bS$ be a nice topological pointclass, and assume that $\Det(\bS(\omega^\omega))$ holds. Let $Z$ be an uncountable zero-dimensional Polish space, and let $\xi,\eta<\omega_1$. Then $(\bG^{(\eta)})^{(\xi)}=\bG^{(\xi+\eta)}$ whenever $\bG,\bG^{(\eta)}\in\NSDS(Z)$.
\end{theorem}
\begin{proof}
Fix $\bG\in\NSDS(\omega^\omega)$ such that $\bG(Z),\bG(Z)^{(\eta)}\in\NSDS(Z)$. We will show that
$$
(\bG(Z)^{(\eta)})^{(\xi)}=\bG(Z)^{(\xi+\eta)}.
$$
The inclusion $\subseteq$ follows from Lemma \ref{composefunctions}. In order to prove the inclusion $\supseteq$, pick $A\in\bG(Z)^{(\xi+\eta)}$. Fix a $\bS^0_{1+\xi+\eta}$-measurable function $g:Z\longrightarrow Z$ and $B\in\bG(Z)$ such that $g^{-1}[B]=A$. Fix a countable base $\UU$ for $Z$. By Corollary \ref{sigmaintoopenbijectioneta}, there exists a Polish space $W$ and a $\mathbf{\Sigma}^0_{1+\xi}$-measurable bijection $f:Z\longrightarrow W$ such that $f[g^{-1}[U]]\in\mathbf{\Sigma}^0_{1+\eta}(W)$ for every $U\in\UU$. Observe that this ensures that $g\circ f^{-1}$ is $\mathbf{\Sigma}^0_{1+\eta}$-measurable. Set $C=f[A]$, and observe that $C=(g\circ f^{-1})^{-1}[B]\in\bG^{(\eta)}(W)$ by Lemma \ref{relativizationexpansion}. A further application of Lemma \ref{relativizationexpansion} shows that 
$$
A=f^{-1}[C]\in\big(\bG^{(\eta)}\big)^{(\xi)}(Z)=\big(\bG^{(\eta)}(Z)\big)^{(\xi)}=\big(\bG(Z)^{(\eta)}\big)^{(\xi)},
$$
where the last two equalities hold by Lemma \ref{movexirelativization}.

Notice that, in order to apply Lemma \ref{movexirelativization} to obtain the middle equality above, we need to know that $\bG^{(\eta)}\in\NSDS(\omega^\omega)$. We conclude the proof by showing that this is the case. Fix $\bL\in\NSDS(\omega^\omega)$ such that $\bL(Z)=\bG(Z)^{(\eta)}$. First, we claim that $\bL\subseteq\bG^{(\eta)}$. In order to prove the claim, pick $A\in\bL$. Fix an embedding $j:\omega^\omega\longrightarrow W$. By Lemma \ref{relativizationsubspace}, there exists $A'\in\bL(Z)$ such that $A'\cap W=j[A]$. So we can fix a $\bS^0_{1+\eta}$-measurable $f:Z\longrightarrow Z$ and $B\in\bG(Z)$ such that $f^{-1}[B]=A'$. Now let $i:Z\longrightarrow\omega^\omega$ be an embedding. By Lemma \ref{relativizationsubspace}, we can pick $B'\in\bG(\omega^\omega)=\bG$ such that $B'\cap i[Z]=i[B]$. Notice that $i\circ f\circ j:\omega^\omega\longrightarrow\omega^\omega$ is $\bS^0_{1+\eta}$-measurable by Lemma \ref{composefunctions}. It follows from the definition of expansion that $A=(i\circ f\circ j)^{-1}[B']\in\bG^{(\eta)}$, which proves the claim.

Now assume, in order to get a contradiction, that $\bL\subsetneq\bG^{(\eta)}$. Pick $A\in\bG^{(\eta)}\setminus\bL$ and $B\subseteq\omega^\omega$ such that $B\wc=\bL$. Lemma \ref{wadgelemma} shows that $\bLc\subseteq\bG^{(\eta)}$, hence
$$
\widecheck{\bG(Z)^{(\eta)}}=\widecheck{\bL(Z)}=\bLc(Z)\subseteq\bG^{(\eta)}(Z)=\bG(Z)^{(\eta)},
$$
where the last equality holds by Lemma \ref{movexirelativization}. This contradicts the fact that $\bG(Z)^{(\eta)}$ is non-selfdual.
\end{proof}

\begin{corollary}\label{expansionbigger}
Let $\bS$ be a nice topological pointclass, and assume that $\Det(\bS(\omega^\omega))$ holds. Let $Z$ be an uncountable zero-dimensional Polish space, let $\bG\in\NSDS(Z)$, and let $\xi=\ell(\bG)$. Assume that $0<\xi<\omega_1$. Then $\bG\subsetneq\bG^{(\xi)}$.
\end{corollary}
\begin{proof}
Assume, in order to get a contradiction, that $\bG=\bG^{(\xi)}$. Then
$$
\bG=\bG^{(\xi)}=\big(\bG^{(\xi)}\big)^{(\xi)}=\bG^{(\xi+\xi)},
$$
where the last equality holds by Theorem \ref{composeexpansions}. It follows from Theorem \ref{expansiontheorem} that $\xi=\ell(\bG)\geq\xi+\xi$, which contradicts the assumption that $\xi>0$.
\end{proof}

We conclude this section with a result that will not be needed in the remainder of the article, but helps to clarify the notion of level. Given a space $Z$ and $\bG\subseteq\PP(Z)$, we will say that $\bG$ is \emph{closed under Borel preimages} if $f^{-1}[B]\in\bG$ whenever $f:Z\longrightarrow Z$ is a Borel function and $B\in\bG$.

\begin{corollary}\label{uncountablelevel}
Let $\bS$ be a nice topological pointclass, and assume that $\Det(\bS(\omega^\omega))$ holds. Let $Z$ be an uncountable zero-dimensional Polish space, and let $\bG\in\NSDS(Z)$. Then the following are equivalent:
\begin{enumerate}
\item\label{uncountablelevelpreimages} $\bG$ is closed under Borel preimages,
\item\label{uncountablelevelomega_1} $\ell(\bG)=\omega_1$.
\end{enumerate}
\end{corollary}
\begin{proof}
In order to prove that $(\ref{uncountablelevelpreimages})\rightarrow (\ref{uncountablelevelomega_1})$, assume that condition $(\ref{uncountablelevelpreimages})$ holds. Pick $\xi<\omega_1$. We will show that $\ell(\bG)\geq\xi$. Fix $A$ such that $\bG=A\wc$. Pick $A_n\in\bG$ and pairwise disjoint $V_n\in\bD^0_{1+\xi}(Z)$ for $n\in\omega$ such that $\bigcup_{n\in\omega}V_n=Z$. Let $f_n:Z\longrightarrow Z$ for $n\in\omega$ witness that $A_n\leq A$. Define
$$
f=\bigcup\{f_n\re V_n:n\in\omega\},
$$
and observe that $f:Z\longrightarrow Z$ is a Borel function. Since $\bG$ is closed under Borel preimages, it follows that $\bigcup_{n\in\omega}(A_n\cap V_n)=f^{-1}[A]\in\bG$. 

In order to prove that $(\ref{uncountablelevelomega_1})\rightarrow (\ref{uncountablelevelpreimages})$, assume that condition $(\ref{uncountablelevelomega_1})$ holds. Pick a Borel $f:Z\longrightarrow Z$, and let $\xi<\omega_1$ be such that $f$ is $\bS^0_{1+\xi}$-measurable. By Theorem \ref{expansiontheorem}, there exists $\bL\in\NSDS(Z)$ such that $\bL^{(\xi\cdot\omega)}=\bG$. Then
$$
\bG^{(\xi)}=\bL^{(\xi+\xi\cdot\omega)}=\bL^{(\xi\cdot\omega)}=\bG,
$$
where the first equality holds by Theorem \ref{composeexpansions}. It follows from the definition of expansion that $f^{-1}[A]\in\bG$ for every $A\in\bG$.
\end{proof}

\section{Separated differences: basic facts}\label{sectionsdbasic}

The following notion was essentially introduced in \cite{louveaua}, but we will follow the simplified approach from \cite{louveaub}. It is the last fundamental concept needed to state our main result (see Definition \ref{louveauhierarchy}).

\begin{definition}[Louveau]
Let $Z$ be a space, let $1\leq\eta<\omega_1$, and let $V_{\xi,n}, A_{\xi,n}, A^\ast\subseteq Z$ for $\xi<\eta$ and $n\in\omega$. Define
\begin{multline}
\SD_\eta((V_{\xi,n}:\xi<\eta,n\in\omega),(A_{\xi,n}:\xi<\eta,n\in\omega),A^\ast)=\\\nonumber=\bigcup_{\substack{\xi<\eta\\n\in\omega}}(A_{\xi,n}\cap (V_{\xi,n}\setminus\bigcup_{\substack{\xi'<\xi\\m\in\omega}}V_{\xi',m} ))\cup (A^\ast\setminus\bigcup_{\substack{\xi<\eta\\n\in\omega}}V_{\xi,n}).
\end{multline}
Given $\bD,\bG^\ast\subseteq\PP(Z)$, define $\SD_\eta(\bD,\bG^\ast)$ as the collection of all sets of the form $\SD_\eta((V_{\xi,n}:\xi<\eta,n\in\omega),(A_{\xi,n}:\xi<\eta,n\in\omega),A^\ast)$, where each $V_{\xi,n}\in\bS^0_1(Z)$ and $V_{\xi,m}\cap V_{\xi,n}=\varnothing$ whenever $m\neq n$, each $A_{\xi,n}\in\bD$, and $A^\ast\in\bG^\ast$.
\end{definition}

We begin with two preliminary results. In particular, Lemma \ref{sddifferences} gives the first ``concrete'' examples of Wadge classes that can be obtained using separated differences.

\begin{lemma}\label{sdcheck}
Let $Z$ be a space, let $1\leq\eta<\omega_1$, and let $\bD,\bG\subseteq\PP(Z)$. Then
$$
\widecheck{\SD_\eta(\bD,\bG)}=\SD_\eta(\widecheck{\bD},\widecheck{\bG}).
$$
\end{lemma}
\begin{proof}
It is not hard to realize that the equality
\begin{multline}
Z\setminus\SD_\eta((V_{\xi,n}:\xi<\eta,n\in\omega),(A_{\xi,n}:\xi<\eta,n\in\omega),A^\ast)=\\\nonumber\SD_\eta((V_{\xi,n}:\xi<\eta,n\in\omega),(Z\setminus A_{\xi,n}:\xi<\eta,n\in\omega),Z\setminus A^\ast)
\end{multline}
holds whenever $V_{\xi,n},A_{\xi,n},A^\ast\subseteq Z$ for each $\xi<\eta$ and $n\in\omega$. The desired result follows immediately.
\end{proof}

\begin{lemma}\label{sddifferences}
Let $Z$ be a zero-dimensional space, let $1\leq\eta<\omega_1$, and let $\bD=\{\varnothing\}\cup\{Z\}$. Then
$$
\SD_\eta(\bD,\{\varnothing\})=\Diff_\eta(\bS^0_1(Z))\text{ and }\SD_\eta(\bD,\{Z\})=\widecheck{\Diff}_\eta(\bS^0_1(Z)).
$$
\end{lemma}
\begin{proof}
We will only prove the first equality, as the second one follows from it by Lemma \ref{sdcheck}. We will proceed by induction on $\eta$. The case $\eta=1$ is trivial. For the successor case, assume that the desired result holds for a given $\eta$. We will show that it also holds for $\eta+1$. In order to prove the inclusion $\subseteq$, pick $A=\SD_{\eta+1}((V_{\xi,n}:\xi<\eta+1,n\in\omega),(A_{\xi,n}:\xi<\eta+1,n\in\omega),\varnothing)$ for suitable $V_{\xi,n}$ and $A_{\xi,n}$. It is easy to realize that $A=B\setminus C$, where 
$$
B=\bigcup\{V_{\xi,n}:\xi<\eta+1,n\in\omega\text{ and }A_{\xi,n}=Z\}
$$
and $C=\SD_\eta((V_{\xi,n}:\xi<\eta,n\in\omega),(Z\setminus A_{\xi,n}:\xi<\eta,n\in\omega),\varnothing)$. By Lemma \ref{diffsuccessor}, this shows that $A\in\Diff_{\eta+1}(\bS^0_1(Z))$.

In order to prove the inclusion $\supseteq$, pick $A\in\Diff_{\eta+1}(\bS^0_1(Z))$. By Lemma \ref{diffsuccessor}, it is possible to write $A=B\setminus C$, where $B\in\bS^0_1(Z)$, $C\in\Diff_\eta(\bS^0_1(Z))$, and $C\subseteq B$. Since $C\in\SD_\eta(\bD,\{\varnothing\})$ by the inductive hypothesis, it is possible to write $C=\SD_\eta((V_{\xi,n}:\xi<\eta,n\in\omega),(C_{\xi,n}:\xi<\eta,n\in\omega),\varnothing)$ for suitable $V_{\xi,n}$ and $C_{\xi,n}$. Let $V_{\eta,0}=B$ and $V_{\eta,n}=\varnothing$ whenever $1\leq n<\omega$. It is easy to realize that $A=\SD_{\eta+1}((V_{\xi,n}:\xi<\eta+1,n\in\omega),(A_{\xi,n}:\xi<\eta+1,n\in\omega),\varnothing)$, where $A_{\xi,n}=Z\setminus C_{\xi,n}$ if $\xi<\eta$, and $A_{\eta,n}=Z$.

Finally, assume that $\eta$ is a limit ordinal and that the desired result holds for all $\eta'<\eta$. Since $2\cdot\eta=\eta$, the inclusion $\subseteq$ follows easily from the definition of separated differences. In order to prove the inclusion $\supseteq$, pick $A\in\Diff_\eta(\bS^0_1(Z))$. Using Lemma \ref{difflimit}, it is possible to find pairwise disjoint $V_k\in\bS^0_1(Z)$ for $k\in\omega$ such that $A=\bigcup_{k\in\omega}(A\cap V_k)$ and for each $k$ there exists $\eta'<\eta$ such that $A\cap V_k\in\Diff_{\eta'}(\bS^0_1(Z))$. Therefore, by the inductive assumption, for every $k$ we can fix suitable $V^k_{\xi,n}$ and $A^k_{\xi,n}$ such that $A\cap V_k=\SD_\eta((V^k_{\xi,n}:\xi<\eta,n\in\omega),(A^k_{\xi,n}:\xi<\eta,n\in\omega),\varnothing)$. Without loss of generality, assume that each $V^k_{\xi,n}\subseteq V_k$. Under this assumption, one sees that $A=\SD_\eta((V^k_{\xi,n}:\xi<\eta,(n,k)\in\omega\times\omega),(A^k_{\xi,n}:\xi<\eta,(n,k)\in\omega\times\omega),\varnothing)$.
\end{proof}

Finally, we show that $\Ha(Z)$ is closed under separated differences (see Proposition \ref{sdarehausdorff}). Notice however that, at this point, we do not know that $\SD_\eta(\bD,\bG^\ast)$ is a non-selfdual Wadge class whenever each of the classes $\bL_n$ and $\bG^\ast$ described below are. That this is true will follow from Theorem \ref{main}.

\begin{proposition}\label{sdarehausdorff}
Let $Z$ be a zero-dimensional space in which $2^\omega$ embeds, let $1\leq\eta<\omega_1$, and let $\bG=\SD_\eta(\bD,\bG^\ast)$, where $\bD$ and $\bG^\ast$ satisfy the following conditions:
\begin{itemize}
\item $\bD=\bigcup_{n\in\omega}(\bL_n\cup\bLc_n)$, where each $\bL_n\in\Ha(Z)$ and each $\ell(\bL_n)\geq 1$,
\item $\bG^\ast\in\Ha(Z)$ and $\bG^\ast\subseteq\bD$.
\end{itemize}
Then $\bG\in\Ha(Z)$ and $\ell(\bG)=0$.
\end{proposition}

\begin{proof}
First we will show that $\bG\in\Ha(Z)$. If $\bD=\{\varnothing,Z\}$ then this follows from Lemma \ref{sddifferences} and Proposition \ref{hausdorffdifferences}. So assume that $\{\varnothing,Z\}\subsetneq\bD$, and notice that this implies that $\bS^0_1(Z)\cup\bP^0_1(Z)\subseteq\bD$. Fix $\pi:\omega\longrightarrow\omega$ such that for every $m\in\omega$ there exist infinitely many $n\in\omega$ such that $\pi(2n)=\pi(2n+1)=m$. Let $\bG'$ be the collection of all sets of the form
$$
\bigcup_{\substack{\xi<\eta\\n\in\omega}}\bigg(A_{\xi,n}\cap V_{\xi,n}\setminus\Big(\bigcup_{\substack{\xi'<\xi\\m\in\omega}}V_{\xi',m}\cup\bigcup_{\substack{m\in\omega\\m\neq n}}V_{\xi,m}\Big)\bigg)\cup\bigg(A^\ast\setminus\bigcup_{\substack{\xi<\eta\\n\in\omega}}V_{\xi,n}\bigg),
$$
where $A_{\xi,n}\in\bL_{\pi(n)}$ if $n$ is even, $A_{\xi,n}\in\bLc_{\pi(n)}$ if $n$ is odd, $A^\ast\in\bG^\ast$, and each $V_{\xi,n}\in\bS^0_1(Z)$. Since each $\bL_n\in\Ha(Z)$ and $\bG^\ast\in\Ha(Z)$, using Lemmas \ref{hausdorffcomposition} and \ref{hausdorffsettheoretic} it is easy to realize that $\bG'\in\Ha(Z)$. Therefore, to conclude this part of the proof, it will be enough to show that $\bG'=\bG$.

Notice that, in the case that $V_{\xi,m}\cap V_{\xi,n}=\varnothing$ whenever $m\neq n$, the term $\bigcup_{\substack{m\in\omega\\m\neq n}}V_{\xi,m}$ is redundant. This shows that $\bG'\supseteq\SD_\eta(\bD,\bG^\ast)$. To see that the other inclusion holds, pick $A\in\bG'$ as above. For every $\xi<\eta$, by \cite[Theorem 22.16]{kechris}, we can fix open sets $V'_{\xi,n}\subseteq V_{\xi,n}$ for $n\in\omega$ such that $\bigcup_{n\in\omega}V'_{\xi,n}=\bigcup_{n\in\omega}V_{\xi,n}$ and $V'_{\xi,m}\cap V'_{\xi,n}=\varnothing$ whenever $m\neq n$. Also set 
$$
A'_{\xi,n}=A_{\xi,n}\setminus\bigcup_{\substack{m\in\omega\\m\neq n}}V_{\xi,m}
$$
for $\xi<\eta$ and $n\in\omega$. We claim that each $A'_{\xi,n}\in\bD$. This will conclude the proof because, as is straightforward to check,
$$
A=\SD_\eta\big((V'_{\xi,n}:\xi<\eta,n\in\omega),(A'_{\xi,n}:\xi<\eta,n\in\omega),A^\ast\big).
$$
If $A_{\xi,n}\in\bL_n\cup\bLc_n$ for some $n\in\omega$ such that $\bL_n\neq\{\varnothing\}$ and $\bL_n\neq\{Z\}$, then the claim follows from Lemma \ref{levelintersection}. If $A_{\xi,n}=\varnothing$, the claim is trivial. Finally, if $A_{\xi,n}=Z$, the claim holds because $\bP^0_1(Z)\subseteq\bD$.

\newpage

It remains to show that $\ell(\bG)=0$. Observe that $\bG\in\NSD(Z)$ by the first part of this proof and Theorem \ref{addisontheorem}. It is easy to realize that $\bD\subseteq\bG\subseteq\PU_1(\bD)$, where the second inclusion uses the assumption $\bG^\ast\subseteq\bD$. Therefore, using Proposition \ref{pubasic}.\ref{pubasicidempotent}, one sees that
$$
\PU_1(\bD)\subseteq\PU_1(\bG)\subseteq\PU_1(\PU_1(\bD))=\PU_1(\bD),
$$
which implies that $\PU_1(\bD)=\PU_1(\bG)$. Since $\PU_1(\bD)$ is selfdual by Proposition \ref{pubasic}.\ref{pubasiccheck}, it follows that $\bG\neq\PU_1(\bG)$, hence $\ell(\bG)=0$.
\end{proof}

\section{Separated differences: the main theorem}\label{sectionsdmain}

The aim of this section is to show that every non-selfdual Wadge class $\bG$ of level zero can be obtained by applying the operation of separated differences to classes of lower complexity (see Theorem \ref{sdmain}). For technical reasons, we will need to assume that $\bL\in\Ha(Z)$ whenever $\bL\in\NSD(Z)$ is such that $\bL\subsetneq\bG$. Notice however that, once Theorem \ref{main} is proved, it will be possible to drop this assumption. To avoid cluttering the exposition, several preliminary lemmas are postponed until the end of the section.

\begin{theorem}\label{sdmain}
Let $\bS$ be a nice topological pointclass, and assume that $\Det(\bS(\omega^\omega))$ holds. Let $Z$ be an uncountable zero-dimensional Polish space, and let $\bG\in\NSDS(Z)$. Assume that $\ell(\bG)=0$ and $\bL\in\Ha(Z)$ whenever $\bL\in\NSD(Z)$ is such that $\bL\subsetneq\bG$. Then there exist $1\leq\eta<\omega_1$, $\bD$ and $\bG^\ast$ satisfying the following conditions:
\begin{itemize}
\item $\bD=\bigcup_{n\in\omega}(\bL_n\cup\bLc_n)$, where each $\bL_n\in\NSDS(Z)$ and each $\ell(\bL_n)\geq 1$,
\item $\bG^\ast\in\NSDS(Z)$ and $\bG^\ast\subseteq\bD$,
\item $\bG=\SD_\eta(\bD,\bG^\ast)$.
\end{itemize}
\end{theorem}
\begin{proof}
Fix $A$ such that $\bG=A\wc$, and a countable base $\UU\subseteq\bD^0_1(Z)$ for $Z$. Define
\begin{multline}
\Phi=\{\bL\in\NSDS(\omega^\omega):\ell(\bL)\geq 1\text{ and there exists }U\in\UU\\\nonumber\text{ such that }\bL(U)\text{ is non-selfdual and }\bL(U)=(A\cap U)\wc\},
\end{multline}
and observe that $\Phi$ is non-empty by Lemma \ref{phinonempty}. Furthermore, using the uniqueness part of Lemma \ref{relativizationexistsunique}, one sees that $\Phi$ is countable. Given a space $W$, define
$$
\bD[W]=\bigcup\{\bL(W)\cup\bLc(W):\bL\in\Phi\}.
$$
Also define $\bD=\bigcup\{\bL\cup\bLc:\bL\in\Phi\}$.\footnote{\,This notation is limited to this proof, as it is easy to realize that the desired $\bD$ (the one mentioned in the statement of the theorem) will in fact be $\bD[Z]$.} It is easy to check that the following facts hold whenever $W$ is a space and $\bL$ is a Wadge class in $\omega^\omega$:
\begin{itemize}
\item If $\bL\subseteq\bD$ then $\bL(W)\subseteq\bD[W]$,
\item If $\bD\subseteq\bL$ then $\bD[W]\subseteq\bL(W)$.
\end{itemize}

Given $W\subseteq Z$, define
$$
\partial(W)=W\setminus\bigcup\{U\in\UU:U\cap W\neq\varnothing\text{ and }A\cap U\cap W\in\bD[U\cap W]\}.
$$
Recursively define a subset $Z_\xi$ of $Z$ for every $\xi<\omega_1$ as follows:
\begin{itemize}
\item $Z_0=Z$,
\item $Z_\xi=\bigcap_{\xi'<\xi}Z_{\xi'}$ if $\xi$ is a limit ordinal,
\item $Z_{\xi+1}=\partial(Z_\xi)$.
\end{itemize}
Since the $Z_\xi$ form a decreasing sequence of closed sets, we can fix the minimal $\zeta<\omega_1$ such that $Z_\zeta=Z_{\zeta+1}$.

Define
$$
\UU_\xi=\{U\in\UU:U\cap Z_\xi\neq\varnothing\text{ and }A\cap U\cap Z_\xi\in\bD[U\cap Z_\xi]\}
$$
for every $\xi<\zeta$. Given $\xi<\zeta$, using the fact that $Z$ is zero-dimensional, it is possible to obtain $\{V_{\xi,n}:n\in\omega\}\subseteq\bD^0_1(Z)$ satisfying the following conditions:
\begin{itemize}
\item $V_{\xi,m}\cap V_{\xi,n}=\varnothing$ whenever $m\neq n$,
\item for all $n\in\omega$ there exists $U\in\UU_\xi$ such that $V_{\xi,n}\subseteq U$,
\item $\bigcup_{n\in\omega}V_{\xi,n}=\bigcup\UU_\xi$.
\end{itemize}
Observe that
$$
Z_\xi=Z\setminus\bigcup_{\substack{\xi'<\xi\\n\in\omega}}V_{\xi',n}
$$
for every $\xi<\omega_1$.

Notice that, by Lemma \ref{relativizationsubspace}, for every $\xi<\zeta$ and $U\in\UU_\xi$ we can fix $A_{\xi,U}\in\bD[Z]$ such that $A_{\xi,U}\cap U\cap Z_\xi=A\cap U\cap Z_\xi$. Given $\xi<\zeta$ and $n\in\omega$, choose $U\in\UU_\xi$ such that $V_{\xi,n}\subseteq U$ and define $A_{\xi,n}=A_{\xi,U}$. At this point, it is easy to check that
\begin{itemize}
 \item[$\circledast(\xi)$] $A=\SD_\xi((V_{\xi',n}:\xi'<\xi,n\in\omega),(A_{\xi',n}:\xi'<\xi,n\in\omega),A\cap Z_\xi)$
\end{itemize}
whenever $1\leq\xi<\zeta$.

The next part of the proof is aimed at showing that $\{\xi<\omega_1:Z_\xi\neq\varnothing\}$ has a maximal element. This will follow from Claims 1 and 5. Claim 2 is a technical result, which will be needed in the proofs of Claims 5 and 7. The significance of Claim 3 will be explained later. Claim 4 will be needed in the proofs of Claims 5 and 10, in order to show that certain separated differences are non-selfdual Wadge classes.

\noindent{\bf Claim 1:} $Z_\zeta=\varnothing$.

Assume, in order to get a contradiction, that $Z_\zeta\neq\varnothing$. Notice that $Z_\zeta$ cannot have isolated points, otherwise we would have $Z_{\zeta+1}\subsetneq Z_\zeta$. By applying Lemma \ref{phinonempty} to $Z_\zeta$ and its base $\{U\cap Z_\zeta:U\in\UU\}\setminus\{\varnothing\}$, we can fix $\bL\in\NSDS(\omega^\omega)$ and $U\in\UU$ such that $U\cap Z_\zeta\neq\varnothing$, $\ell(\bL)\geq 1$, $\bL(U\cap Z_\zeta)$ is non-selfdual, and $\bL(U\cap Z_\zeta)=(A\cap U\cap Z_{\zeta})\wc$. Observe that $U$ is uncountable because $Z_\zeta$ has no isolated points. First we will show that $\bD\subseteq\bL$. If this were not the case, then Lemma \ref{selfdualwadgelemma} would imply that $\bL\subseteq\bD$. Then, it would follow that $A\cap U\cap Z_{\zeta}\in\bL(U\cap Z_{\zeta})\subseteq\bD[U\cap Z_{\zeta}]$, which would contradict the assumption that $Z_\zeta=Z_{\zeta+1}$.

To finish the proof of the claim, we will show that $\bL(U)=(A\cap U)\wc$. Notice that this will imply $\bL\in\Phi$, hence $\bL\subseteq\bD$. Since we have already seen that $\bD\subseteq\bL$, this will contradict the fact that $\bL$ is non-selfdual.

Observe that
$$
\{U\cap Z_\zeta\}\cup\{V_{\xi,n}\cap U\cap (Z_\xi\setminus Z_{\xi+1}):\xi<\zeta\text{ and }n\in\omega\}
$$
is cover of $U$ consisting of pairwise disjoint $\bD^0_2$ sets. By Lemma \ref{relativizationsubspace}, we can fix $B\in\bL(U)$ such that $B\cap U\cap Z_\zeta=A\cap U\cap Z_\zeta$. It is easy to realize that
$$
A\cap U=(B\cap U\cap Z_\zeta)\cup\bigcup_{\substack{\xi<\zeta\\n\in\omega}}(A_{\xi,n}\cap V_{\xi,n}\cap U\cap (Z_\xi\setminus Z_{\xi+1})).
$$
Notice that each $A_{\xi,n}\cap U\in\bD[U]\subseteq\bL(U)$ by Lemma \ref{relativizationsubspace}. Since $\bL\neq\{\omega^\omega\}$ by the fact that $\bD\subseteq\bL$, it follows from Lemma \ref{closureclopen} that each $A_{\xi,n}\cap V_{\xi,n}\cap U\in\bL(U)$. In conclusion, since $\ell(\bL(U))\geq 1$ by Corollary \ref{levelinvariance}, one sees that $A\cap U\in\PU_1(\bL(U))=\bL(U)$. This allows us to apply Lemma \ref{relativizationsd2} with $Z=U$, $W=U\cap Z_\zeta$, and $g$ the natural embedding, which yields $\bL(U)=(A\cap U)\wc$. $\blacksquare$

\noindent{\bf Claim 2:} Let $1\leq\nu<\zeta$, and let $\bG'\in\NSDS(\omega^\omega)$. If $A\in\SD_\nu(\bD[Z],\bG'(Z))$ then $A\cap Z_\nu\in\bG'(Z_\nu)$.

Suppose $A=\SD_\nu((V'_{\xi,n}:\xi<\nu,n\in\omega),(A'_{\xi,n}:\xi<\nu,n\in\omega),A')$, where the $V'_{\xi,n}\in\bS^0_1(Z)$ are pairwise disjoint, each $A'_{\xi,n}\in\bD[Z]$, and $A'\in\bG'(Z)$. Define
$$
Z'_\xi=Z\setminus\bigcup_{\substack{\xi'<\xi\\n\in\omega}}V'_{\xi',n}
$$
for $\xi\leq\nu$. First we will prove, by induction on $\xi$, that $Z_\xi\subseteq Z'_\xi$ for every $\xi\leq\nu$. The case $\xi=0$ and the limit case are trivial. Now assume that $Z_\xi\subseteq Z'_\xi$ for a given $\xi<\nu$. In order to prove that $Z_{\xi+1}\subseteq Z'_{\xi+1}$, it will be enough to show that $A\cap U\cap Z_\xi\in\bD[U\cap Z_\xi]$ for every $U\in\UU$ such that $U\cap Z_\xi\neq\varnothing$ and $U\subseteq V'_{\xi,n}$ for some $n\in\omega$. This follows from Lemma \ref{relativizationsubspace} plus the fact that $A\cap U\cap Z_\xi=A'_{\xi,n}\cap U\cap Z_\xi$ for every such $U$, which can easily be deduced by inspecting the expression of $A$ as a separated difference, using the inductive assumption that $Z_\xi\subseteq Z'_\xi$.

At this point, using the fact that $Z_\nu\subseteq Z'_\nu$, it is easy to realize that $A\cap Z_\nu=A'\cap Z_\nu$. Since $A'\cap Z_\nu\in\bG'(Z_\nu)$ by Lemma \ref{relativizationsubspace}, this concludes the proof of the claim. $\blacksquare$

From this point on, it will be useful to assume that $\{\varnothing,\omega^\omega\}\subsetneq\bD$. The following claim, together with Theorem \ref{completeanalysisdifferences} and Lemma \ref{sddifferences}, shows that this does not result in any loss of generality. In the remainder of the proof, we will use two consequences of this assumption. The first one is that, by Lemma \ref{closureclopen}, we will have $A\cap U\in\bD[Z]$ whenever $A\in\bD[Z]$ and $U\in\bD^0_1(Z)$. The second one is given by Claim 8.

\noindent{\bf Claim 3:} Assume that $\bD=\{\varnothing,\omega^\omega\}$. Then $A\in\bD^0_2(Z)$.

Notice that $\{Z_\xi\setminus Z_{\xi+1}:\xi<\zeta\}$ is a partition of $Z$ by Claim 1. Therefore
$$
A=\bigcup_{\xi<\zeta}A\cap (Z_\xi\setminus Z_{\xi+1})\text{ and }Z\setminus A=\bigcup_{\xi<\zeta}(Z\setminus A)\cap (Z_\xi\setminus Z_{\xi+1}).
$$
By inspecting the definition of $\partial$, using the fact that $\bD=\{\varnothing,\omega^\omega\}$, it is easy to realize that both
$$
A\cap (Z_\xi\setminus Z_{\xi+1})\in\bS^0_1(Z_\xi)\text{ and }(Z\setminus A)\cap (Z_\xi\setminus Z_{\xi+1})\in\bS^0_1(Z_\xi)
$$
for every $\xi<\zeta$. Since $\bS^0_1(F)\subseteq\bS^0_2(Z)$ for every $F\in\bP^0_1(Z)$, it follows that both $A\in\bS^0_2(Z)$ and $Z\setminus A\in\bS^0_2(Z)$. Hence $A\in\bD^0_2(Z)$. $\blacksquare$

\noindent{\bf Claim 4:} $\bL(Z)\subsetneq\bG$ for every $\bL\in\Phi$. In particular, $\bL(Z)\in\Ha(Z)$ for every $\bL\in\Phi$.

Pick $\bL\in\Phi$, and let $U\in\UU$ be such that $\bL(U)$ is non-selfdual and $\bL(U)=(A\cap U)\wc$. It will be enough to show that $\bL(Z)\subseteq\bG$, as $\ell(\bG)=0$ by assumption, while $\ell(\bL(Z))\geq 1$ by Corollary \ref{levelinvariance}.

First assume that $U$ is countable. Observe that $\bS^0_2(\omega^\omega),\bP^0_2(\omega^\omega)\in\NSDS(\omega^\omega)$ by Proposition \ref{hausdorffborel} and Theorem \ref{addisontheorem}. Since $\bL(U)$ is non-selfdual, we must have $\bS^0_2(\omega^\omega)\nsubseteq\bL$ and $\bP^0_2(\omega^\omega)\nsubseteq\bL$. Therefore, $\bL\subseteq\bD^0_2(\omega^\omega)$ by Lemma \ref{wadgelemma}. Notice that it is not possible that $\bL=\Diff_\nu(\bS^0_1(\omega^\omega))$ or $\bL=\widecheck{\Diff}_\nu(\bS^0_1(\omega^\omega))$ for some $1\leq\nu<\omega_1$, because these classes have level $0$ by Lemma \ref{sddifferences} and Proposition \ref{sdarehausdorff}. It follows from Theorem \ref{completeanalysisdifferences} that $\bL=\{\varnothing\}$ or $\bL=\{\omega^\omega\}$, which concludes the proof of the claim in this case.

Now assume that $U$ is uncountable, and that $\bL\neq\{\omega^\omega\}$. By Lemma \ref{relativizationsubspace}, there exists $A'\in\bL(Z)$ such that $A'\cap U=A\cap U$. It follows from Lemma \ref{closureclopen} that $A\cap U\in\bL(Z)$. Therefore, an application of Lemma \ref{relativizationsd2} with $W=U$ and $g:U\longrightarrow Z$ the natural embedding yields $\bL(Z)=(A\cap U)\wc$. Since $A\cap U\in\bG$ by Lemma \ref{closureclopen}, this concludes the proof of the claim. $\blacksquare$

\noindent{\bf Claim 5:} Let $\nu<\omega_1$ be a limit ordinal such that $Z_\xi\neq\varnothing$ for every $\xi<\nu$. Then $Z_\nu\neq\varnothing$.

Observe that $\nu\leq\zeta$ by Claim 1. Assume, in order to get a contradiction, that $Z_\nu=\varnothing$. Given any $\xi<\nu$ and $W\subseteq\bigcup_{\substack{\xi'<\xi+1\\n\in\omega}}V_{\xi',n}$, using condition $\circledast(\xi+1)$, it is easy to realize that
$$
A\cap W=\mathsf{SD}_{\xi+1}((V_{\xi',n}:\xi'<\xi+1,n\in\omega),(A_{\xi',n}\cap W:\xi'<\xi+1,n\in\omega),\varnothing).
$$
Furthermore, if $W\in\bD^0_1(Z)$ then each $A_{\xi',n}\cap W\in\bD[Z]$ by Lemma \ref{closureclopen}, since we are assuming that $\{\varnothing,\omega^\omega\}\subsetneq\bD$. In particular, $A\cap V_{\xi,n}\in\SD_{\xi+1}(\bD[Z],\{\varnothing\})$ for every $\xi<\nu$ and $n\in\omega$. By Claim 4, Lemma \ref{sdarehausdorff}, and Theorem \ref{addisontheorem}, for every $\xi<\omega_1$ we can fix $B_\xi\subseteq Z$ such that $\SD_{\xi+1}(\bD[Z],\{\varnothing\})=B_\xi\wc$. Finally, we will show that $B_\xi<A$ in $Z$ whenever $\xi<\nu$. Since $\{V_{\xi,n}:\xi<\nu\text{ and }n\in\omega\}$ is a cover of $Z$ by the assumption that $Z_\nu=\varnothing$, an application of Proposition \ref{locallylowerimpliesselfdual} will contradict the fact that $A$ is non-selfdual, hence conclude the proof of the claim. 

Pick $\xi<\nu$. We need to show that $A\nleq B_\xi$ and $B_\xi\leq A$. First assume, in order to get a contradiction, that $A\leq B_\xi$. Then $A\cap Z_{\xi+1}=\varnothing$ by Claim 2. It follows from the definition of $\partial$ that $Z_{\xi+2}=\varnothing$, which contradicts our assumptions. Now assume, in order to get a contradiction, that $B_\xi\nleq A$. Then $A\leq Z\setminus B_\xi$ by Lemma \ref{wadgelemma}, hence $A\in\SD_{\xi+1}(\bD[Z],\{Z\})$ by Lemma \ref{sdcheck}. Therefore $A\cap Z_{\xi+1}=Z_{\xi+1}$ by Claim 2. Again, it follows from the definition of $\partial$ that $Z_{\xi+2}=\varnothing$, which contradicts our assumptions. $\blacksquare$

As we mentioned above, Claims 1 and 5 allow us to define
$$
\eta=\maxi\{\xi<\omega_1:Z_\xi\neq\varnothing\}.
$$
Observe that $1\leq\eta<\zeta$, where the first inequality is given by the following claim, while the second one is an obvious consequence of Claim 1.

\noindent{\bf Claim 6:} $\eta\geq 1$.

Assume, in order to get a contradiction, that $\eta=0$. This means that $\UU_0=\{U\in\UU:A\cap U\in\bD[U]\}$ is a cover of $Z_0=Z$. We will show that $A\cap U<A$ in $Z$ for every $U\in\UU_0$. This will conclude the proof by Proposition \ref{locallylowerimpliesselfdual}, as the fact that $A$ is non-selfdual will be contradicted.

The reduction $A\cap U\leq A$ follows from Lemma \ref{closureclopen}. Now assume, in order to get a contradiction, that $A\leq A\cap U$. By Lemma \ref{relativizationsubspace}, there exists $A'\in\bD[Z]$ such that $A'\cap U=A\cap U$. On the other hand $A'\cap U\in\bD[Z]$ by Lemma \ref{closureclopen}, because are assuming that $\{\varnothing,\omega^\omega\}\subsetneq\bD$. In conclusion, we see that $A\leq A\cap U=A'\cap U\in\bD[Z]$, which contradicts Claim 4. $\blacksquare$

The next claim will allow us to define $\bG^\ast$. Set $A^\ast=A\cap Z_\eta$.

\noindent{\bf Claim 7:} $A^\ast$ is non-selfdual in $Z_\eta$.

Assume, in order to get a contradiction, that $A^\ast$ is selfdual in $Z_\eta$. By Corollary \ref{selfdualcorollary}, we can fix pairwise disjoint $U_k\in\mathbf{\Delta}^0_1(Z_\eta)$ and non-selfdual $A_k<A^\ast$ in $Z_\eta$ for $k\in\omega$ such that $\bigcup_{k\in\omega}U_k=Z_\eta$ and $\bigcup_{k\in\omega}(A_k\cap U_k)=A^\ast$. By Lemma \ref{relativizationexistsunique}, we can fix $\bG_k\in\NSDS(\omega^\omega)$ for $k\in\omega$ such that $\bG_k(Z_\eta)=A_k\wc$. By \cite[Theorem 7.3]{kechris}, we can fix a retraction $\rho:Z\longrightarrow Z_\eta$. Define $W_k=\rho^{-1}[U_k]$ for each $k$. Next, we will show that $A\cap W_k<A$ in $Z$ for each $k$. Notice that, by Proposition \ref{locallylowerimpliesselfdual}, this will contradict the fact that $A$ is non-selfdual, hence conclude the proof of the claim. 

Pick $k\in\omega$. First assume that $\bG_k\neq\{\omega^\omega\}$. Then $A_k\cap U_k\in\bG_k(Z_\eta)$ by Lemma \ref{closureclopen}. Therefore, since $A^\ast\cap W_k=A^\ast\cap U_k=A_k\cap U_k$, by Lemma \ref{relativizationsubspace} there exists $A'\in\bG_k(Z)$ such that $A'\cap Z_\eta=A^\ast\cap W_k$. Using condition $\circledast(\eta)$, it is easy to realize that
$$
A\cap W_k=\SD_\eta((V_{\xi,n}:\xi<\eta,n\in\omega),(A_{\xi,n}\cap W_k:\xi<\eta,n\in\omega),A').
$$
Furthermore, using Lemma \ref{closureclopen} and the assumption that $\{\varnothing,\omega^\omega\}\subsetneq\bD$, one sees that each $A_{\xi,n}\cap W_k\in\bD[Z]$. In conclusion, we see that $A\cap W_k\in\SD_\eta(\bD[Z],\bG_k(Z))$. To see that the same holds in the case $\bG_k=\{\omega^\omega\}$, observe that
$$
A\cap W_k=\SD_\eta((V'_{\xi,n}:\xi<\eta,-1\leq n<\omega),(A_{\xi,n}:\xi<\eta,-1\leq n<\omega),Z),
$$
where $V'_{\xi,n}=V_{\xi,n}\cap W_k$ for every $n\in\omega$, $V'_{\xi,-1}=Z\setminus W_k$, and $A_{\xi,-1}=\varnothing$.

Using the fact that $A_k<A^\ast$ in $Z_\eta$ and $A_k$ is non-selfdual in $Z_\eta$, one sees that $A^\ast\notin\bG_k(Z_\eta)$ and $A^\ast\notin\bGc_k(Z_\eta)$. Therefore $A\notin\SD_\eta(\bD[Z],\bG_k(Z))$ and $A\notin\SD_\eta(\bD[Z],\bGc_k(Z))$ by Claim 2. In particular, $A\notin\bG_k(Z)$ and $A\notin\bGc_k(Z)$. Hence $\bG_k(Z)\subsetneq\bG$ by Lemma \ref{wadgelemma}, which implies $\bG_k(Z)\in\Ha(Z)$ by assumption. In conclusion, by Claim 4, Proposition \ref{sdarehausdorff}, and Theorem \ref{addisontheorem}, we can fix $B\subseteq Z$ such that $B\wc=\SD_\eta(\bD[Z],\bG_k(Z))$. It remains to show that $B<A$. This follows from the second sentence of this paragraph, using Lemmas \ref{wadgelemma} and \ref{sdcheck}. $\blacksquare$

By Claim 7, we can fix $\bG^\ast\in\NSDS(\omega^\omega)$ such that $\bG^\ast(Z_\eta)=A^\ast\wc$. The final part of the proof (namely, Claims 9 and 10) will show that $\bD[Z]$ and $\bG^\ast(Z)$ satisfy the desired requirements. Claim 8 will be used in the proof of Claim 9.

\noindent{\bf Claim 8:} $\bS^0_2(\omega^\omega)\cup\bP^0_2(\omega^\omega)\subseteq\bD$.

Since $\bD$ is selfdual, it will be enough to show that $\bS^0_2(\omega^\omega)\subseteq\bD$ or $\bP^0_2(\omega^\omega)\subseteq\bD$. Using the assumption $\{\varnothing,\omega^\omega\}\subsetneq\bD$, we can pick $\bL\in\Phi$ such that $\bL\neq\{\varnothing\}$ and $\bL\neq\{\omega^\omega\}$. Assume, in order to get a contradiction, that $\bS^0_2(\omega^\omega)\nsubseteq\bL$ and $\bP^0_2(\omega^\omega)\nsubseteq\bL$. Now proceed as in the proof of Claim 4. $\blacksquare$

\noindent{\bf Claim 9:} $\bG^\ast(Z)\subseteq\bD[Z]$.

First assume that $Z_\eta$ is countable. Observe that $\bG^\ast\subseteq\bP^0_2(\omega^\omega)$, otherwise we would have $\bS^0_2(\omega^\omega)\subseteq\bG^\ast$ by Lemma \ref{wadgelemma}, hence $\bG^\ast(Z_\eta)=\PP(Z_\eta)$, which would contradict the fact that $\bG^\ast(Z_\eta)$ is non-selfdual. It follows that $\bG^\ast(Z)\subseteq\bP^0_2(Z)\subseteq\bD[Z]$, where the last inclusion holds by Claim 8. Now assume that $Z_\eta$ is uncountable. Assume, in order to get a contradiction, that $\bG^\ast(Z)\nsubseteq\bD[Z]$. Then $\bD[Z]\subsetneq\bG^\ast(Z)$ by Lemma \ref{selfdualwadgelemma}. Since $Z_\eta$ is uncountable, it follows from Theorem \ref{orderisomorphism} that $\bD[Z_\eta]\subsetneq\bG^\ast(Z_\eta)$. Notice that
$$
\VV_\eta=\{U\cap Z_\eta:U\in\UU_\eta\}\subseteq\bD^0_1(Z_\eta)
$$
is a cover of $Z_\eta$ by the definition of $\eta$. To conclude the proof of the claim, it will be enough to show that $A^\ast\cap V< A^\ast$ in $Z_\eta$ for every $V\in\VV_\eta$, as this will contradict Claim 7 by Proposition \ref{locallylowerimpliesselfdual}. Pick $V\in\VV_\eta$. It is clear from the definition of $\UU_\eta$ that $A^\ast\cap V\in\bD[V\cap Z_\eta]$. Therefore, by Lemma \ref{relativizationsubspace}, there exists $B\in\bD[Z_\eta]$ such that $B\cap V=A^\ast\cap V$. By Lemma \ref{closureclopen}, using the assumption that $\{\varnothing,\omega^\omega\}\subsetneq\bD$, it follows that $A^\ast\cap V\in\bD[Z_\eta]$. Since $\bD[Z_\eta]\subsetneq\bG^\ast(Z_\eta)$, this finishes the proof of the claim. $\blacksquare$

\noindent{\bf Claim 10:} $\bG=\SD_\eta(\bD[Z],\bG^\ast(Z))$.

Notice that $\SD_\eta(\bD[Z],\bG^\ast(Z))\in\NSDS(Z)$ by Claims 4 and 9, Lemma \ref{sdarehausdorff}, and Theorem \ref{addisontheorem}. Furthermore, condition $\circledast(\eta)$ shows that $A\in\SD_\eta(\bD[Z],\bG^\ast(Z))$. This proves that the inclusion $\subseteq$ holds. By Lemma \ref{wadgelemma}, in order to prove that the inclusion $\supseteq$ holds, it will be enough to show that $Z\setminus A\notin\SD_\eta(\bD[Z],\bG^\ast(Z))$. Assume, in order to get a contradiction, that this is not the case. Then $A\in\SD_\eta(\bD[Z],\widecheck{\bG^\ast}(Z))$ by Lemma \ref{sdcheck}, hence $A^\ast=A\cap Z_\eta\in\widecheck{\bG^\ast}(Z_\eta)$ by Claim 2. This contradicts Claim 7. $\blacksquare$
\end{proof}

\begin{lemma}\label{relativizationsd1}
Let $\bS$ be a nice topological pointclass, and assume that $\Det(\bS(\omega^\omega))$ holds. Let $Z$ be a zero-dimensional Borel space, let $W$ be an uncountable zero-dimensional Polish space, and let $\bL\in\NSDS(\omega^\omega)$. Assume that $A\in\bL(Z)$ and $B\wc =\bL(W)$. Then there exists a continuous function $f:Z\longrightarrow W$ such that $A=f^{-1}[B]$.
\end{lemma}
\begin{proof}
By Lemma \ref{relativization}.\ref{relativizationhomeo}, we can assume without loss of generality that $Z$ and $W$ are subspaces of $\omega^\omega$. In fact, we will  also assume that $Z=\omega^\omega$, since the general case follows easily from Lemma \ref{relativizationsubspace} and this particular case.

Set $A_0=B$ and $A_1=W\setminus B$. Assume, in order to get a contradiction, that there exists $C\in\bLc$ such that $A_0\subseteq C$ and $C\cap A_1=\varnothing$. It follows from Lemma \ref{relativizationsubspace} that $B=A_0\in\bLc(W)$. Since $\bL(W)$ is non-selfdual by Theorem \ref{vanwesepsurrogate}, this is a contradiction. Therefore, an application of Lemma \ref{strongwadgelemma} with $\bG=\bLc$ and $D=A$ yields the desired function.
\end{proof}

\begin{lemma}\label{relativizationsd2}
Let $\bS$ be a nice topological pointclass, and assume that $\Det(\bS(\omega^\omega))$ holds. Let $Z$ be a zero-dimensional Borel space, let $W$ be an uncountable zero-dimensional Polish space, and let $\bL\in\NSDS(\omega^\omega)$. Assume that $A\in\bL(Z)$, $B\wc =\bL(W)$, and that there exists a continuous function $g:W\longrightarrow Z$ such that $B=g^{-1}[A]$. Then $A\wc =\bL(Z)$.
\end{lemma}
\begin{proof}
Since $A\in\bL(Z)$, we only need to show that $\bL(Z)\subseteq A\wc$. So pick $C\in\bL(Z)$. By Lemma \ref{relativizationsd1} there exists a continuous function $f:Z\longrightarrow W$ such that $C=f^{-1}[B]$. It is clear that $g\circ f$ witnesses that $C\leq A$. Hence $C\in A\wc$.
\end{proof}

\begin{lemma}\label{phinonempty}
Let $\bS$ be a nice topological pointclass, and assume that $\Det(\bS(\omega^\omega))$ holds. Let $Z$ be a zero-dimensional Polish space, let $A\in\bS(Z)$, and let $\UU\subseteq\bD^0_1(Z)$ be a base for $Z$. Then there exist $U\in\UU$ and $\bL\in\NSDS(\omega^\omega)$ such that $\ell(\bL)\geq 1$, $\bL(U)$ is non-selfdual, and $\bL(U)=(A\cap U)\wc$.
\end{lemma}
\begin{proof}
Observe that if $U\in\UU$ is such that $A\cap U=\varnothing$ or $U\subseteq A$, then setting $\bL=\{\varnothing\}$ or $\bL=\{\omega^\omega\}$ respectively will yield the desired result. Therefore, we can assume without loss of generality that $A\cap Z$ and $Z\setminus A$ are both dense in $Z$. Notice that, in particular, it follows that $Z$ has no isolated points.

Set $\Aa=\{A\cap U:U\in\UU\}$, and observe that $\Aa$ has a $\leq$-minimal element by Theorem \ref{wellfounded}. Let $B$ be such an element of $\Aa$, and fix $U\in\UU$ such that $B=A\cap U$. First we will show that $B$ is non-selfdual in $Z$. Assume, in order to get a contradiction, that $B$ is selfdual in $Z$. The assumption that $A$ and $Z\setminus A$ are both dense in $Z$ implies that $B\notin\bD^0_1(Z)$, hence it is possible to apply Theorem \ref{selfdualtheorem}. In particular, we can fix $V\in\bD^0_1(Z)$ such that $U\cap V\neq\varnothing$ and $B\cap V<B$. Notice that $B\cap V\neq Z$ because $B\neq Z$, hence by Lemma \ref{closureclopen} we can assume without loss of generality that $V\in\UU$ and $V\subseteq U$. This contradicts the minimality of $B$.

Since $B$ is non-selfdual in $Z$, we can fix $\bL\in\NSDS(\omega^\omega)$ such that $\bL(Z)=B\wc$. We claim that $\bL(U)=B\wc$. First notice that $B\in\bL(U)$ by Lemma \ref{relativizationsubspace}, and that $U$ is uncountable because $Z$ has no isolated points. Furthermore, using the fact that $U\nsubseteq A$, it is easy to construct a continuous function $g:Z\longrightarrow U$ such that $g^{-1}[B]=B$. Hence, the claim follows from Lemma \ref{relativizationsd2}.

Finally, we will show that $\ell(\bL)\geq 1$. By Corollary \ref{levelinvariance}, it will be enough to show that $\ell(\bL(U))\geq 1$. Assume, in order to get a contradiction, that $\ell(\bL(U))=0$. By Lemma \ref{l0impliessomewheresmaller}, there exists a non-empty $V\in\mathbf{\Delta}^0_1(U)$ such that $B\cap V<B$ in $U$, hence in $Z$. As above, this contradicts the minimality of $B$.
\end{proof}

\begin{lemma}\label{l0impliessomewheresmaller}
Let $\bS$ be a nice topological pointclass, and assume that $\Det(\bS(\omega^\omega))$ holds. Let $Z$ be a zero-dimensional Polish space, and let $\bG=B\wc\in\NSDS(Z)$ be such that $\ell(\bG)=0$. Then there exists a non-empty $V\in\mathbf{\Delta}^0_1(Z)$ such that $B\cap V<B$.	
\end{lemma}
\begin{proof}
Notice that $\mathsf{PU}_1(\bG)\nsubseteq\bG$ because $\ell(\bG)=0$. By Lemma \ref{wadgelemma} and the fact that $\mathsf{PU}_1(\bG)$ is continuously closed, it follows that $\bGc\subseteq\mathsf{PU}_1(\bG)$. Therefore, we can fix $B_n\in\bG$ and pairwise disjoint $V_n\in\mathbf{\Delta}^0_2(Z)$ for $n\in\omega$ such that $\bigcup_{n\in\omega}V_n=Z$ and
$$
\bigcup_{n\in\omega}(B_n\cap V_n)=Z\setminus B.
$$
Since $Z$ is a Baire space, we can also fix $n\in\omega$ and a non-empty $V\in\mathbf{\Delta}^0_1(Z)$ such that $V\subseteq V_n$.

Observe that $\bG\neq\{Z\}$ and $\bG\neq\{\varnothing\}$ because $\ell(\bG)=0$, hence it is possible to apply Lemma \ref{closureclopen}. In particular, one sees that $V\setminus B=B_n\cap V\in\bG$, hence $Z\setminus(B\cap V)=(Z\setminus V)\cup (V\setminus B)\in\bG$. In conclusion, we have $B\cap V\in\bG$ (again by Lemma \ref{closureclopen}) and $Z\setminus(B\cap V)\in\bG$. Since $\bG$ is non-selfdual, it follows that $B\cap V<B$.
\end{proof}

\section{The stretch operation and Radin's theorem}\label{sectionstretch}

As part of the proof of our main result (Theorem \ref{main}), we will need to show that every non-selfdual Wadge class $\bG$ such that $\ell(\bG)=\omega_1$ is the class associated to some Hausdorff operation. The purpose of this section is to give a proof of this fact (see Theorem \ref{vanweseplevelatleast2}).

First, we will deal with the case in which the ambient space is $\omega^\omega$. We will mostly follow the approach of \cite{vanwesept}. More precisely, Definition \ref{radindefinition} corresponds to \cite[Definitions 5.1.1 and 5.1.2]{vanwesept}, while Lemma \ref{radinlemma} is \cite[Lemma 5.1.3]{vanwesept}. The novelty here consists in observing that the assumption needed to apply Radin's Lemma \ref{radinlemma} (namely, that $A\equiv A^\sss$) will hold whenever $\ell(A\wc)\geq 2$. This is the content of Lemma \ref{levelatleast2}.

Set $\CC=\{S\subseteq\omega:\omega\setminus S\text{ is infinite}\}$. Given $S\in\CC$, let $\pi_S:\omega\setminus S\longrightarrow\omega$ denote the unique increasing bijection. Define $\phi:\CC\times\omega^\omega\longrightarrow\omega^\omega$ by setting
$$
\phi(S,x)(n)=\left\{
\begin{array}{ll} 0 & \textrm{if }n\in S,\\
x(\pi_S(n))+1 & \textrm{if }n\in\omega\setminus S.
\end{array}
\right.
$$

\begin{definition}[Radin]\label{radindefinition}
Given $A\subseteq\omega^\omega$, define
$$
A^\sss=\phi[\CC\times A].
$$
We will refer to $A^\sss$ as the \emph{stretch} of $A$.
\end{definition}

Informally, $A^\sss$ consists of all reals obtained by picking an element of $A$, inserting some number (finite or infinite) of zeros, and increasing by $1$ every other entry. The fact that this operation can be reversed will be used in the proof of the next lemma. Also notice that $A\leq A^\sss$ for every $A\subseteq\omega^\omega$, as witnessed by the function $f:\omega^\omega\longrightarrow\omega^\omega$ defined by setting $f(x)(n)=x(n)+1$ for $x\in\omega^\omega$ and $n\in\omega$.

\begin{lemma}\label{levelatleast2}
Let $\bS$ be a nice topological pointclass, and assume that $\Det(\bS(\omega^\omega))$ holds. Let $\bG=A\wc\in\NSDS(\omega^\omega)$. Assume that $\ell(\bG)\geq 2$. Then $A\equiv A^\sss$.
\end{lemma}
\begin{proof}
We have already observed that $A\leq A^\sss$. It remains to show that $A^\sss\leq A$.
Set
$$
W=\{x\in\omega^\omega:x(n)\neq 0\text{ for infinitely many }n\in\omega\}.
$$
Define $f:W\longrightarrow\omega^\omega$ by ``erasing all the zeros, and decreasing by $1$ all other entries''. More precisely, given $x\in W$, set $f(x)(n)=x(\pi_S^{-1}(n))-1$ for $n\in\omega$, where $S=\{n\in\omega:x(n)=0\}$. Observe that $f^{-1}[A]=A^\sss$. Furthermore, it is easy to realize that $f$ is continuous, hence $A^\sss=f^{-1}[A]\in\bG(W)$ by Lemma \ref{relativization}.\ref{relativizationpreimage}. By Lemma \ref{relativizationsubspace}, there exists $B\in\bG(\omega^\omega)=\bG$ such that $B\cap W=A^\sss$. Finally, since $W\in\bP^0_2(\omega^\omega)\subseteq\bD^0_3(\omega^\omega)$ and $\ell(\bG)\geq 2$, an application of Lemma \ref{levelintersection} shows that $A^\sss\in\bG$.
\end{proof}

\begin{lemma}[Radin]\label{radinlemma}
Let $\bG=A\wc\in\NSD(\omega^\omega)$. Assume that $A\equiv A^\sss$. Then $\bG\in\Ha(\omega^\omega)$.
\end{lemma}
\begin{proof}
Given an ambient set $Z$ and a sequence $(U_s:s\in\omega^{<\omega})$ consisting of subsets of $Z$, define the operation $\HH$ by declaring $z\in\HH(U_s:s\in\omega^{<\omega})$ if the following conditions are satisfied:
\begin{enumerate}
 \item\label{radinforall} For all $n\in\omega$ there exists a unique $s\in\omega^n$ such that $z\in U_s$,
 \item\label{radinexists} There exists $x\in A$ such that $z\in U_{x\re n}$ for all $n\in\omega$.
\end{enumerate}
After identifying $\omega^{<\omega}$ with $\omega$, one sees that $\HH$ is a Hausdorff operation. In fact, it is clear that $\HH=\HH_D$, where $D=\{\{x\re n:n\in\omega\}:x\in A\}\subseteq\PP(\omega^{<\omega})$. We claim that $\bG=\bG_D(\omega^\omega)$, which will conclude the proof. In particular, $Z=\omega^\omega$ from now on.

To see that $\bG\supseteq\bG_D(\omega^\omega)$, fix a sequence $(U_s:s\in\omega^{<\omega})$ consisting of open subsets of $\omega^\omega$. Given $x=(x_0,x_1,\ldots)\in\omega^\omega$, define $f(x)=y=(y_0,y_1,\ldots)\in\omega^\omega$ recursively as follows. Fix $j\in\omega$, and assume that $y_i$ has been defined for all $i<j$. Set $n=|\{i<j:y_i\neq 0\}|$.
\begin{itemize}
 \item If there exists a unique $s\in\omega^{n+1}$ such that $\Ne_{(x_0,
 \ldots,x_j)}\subseteq U_s$, then set $y_j=s(n)+1$.
 \item Otherwise, set $y_j=0$.
\end{itemize}
It is not hard to realize that $f:\omega^\omega\longrightarrow\omega^\omega$ witnesses that $\HH(U_s:s\in\omega^{<\omega})\leq A^\sss$.

To see that the inclusion $\bG\subseteq\bG_D(\omega^\omega)$ holds, pick $B\leq A$, and let $f:\omega^\omega\longrightarrow\omega^\omega$ witness this reduction. Define $U_s=f^{-1}[\Ne_s]$ for $s\in\omega^{<\omega}$. We claim that $B=\HH(U_s:s\in\omega^{<\omega})$. In order to prove the inclusion $\subseteq$, pick $z\in B$. Condition $(\ref{radinforall})$ is certainly verified, as $\{U_s:s\in\omega^n\}$ is a partition of $\omega^\omega$ for every $n\in\omega$. Furthermore, it is straightforward to check that $x=f(z)$ witnesses that condition $(\ref{radinexists})$ holds. In order to prove the inclusion $\supseteq$, pick $z\in\HH(U_s:s\in\omega^{<\omega})$. Fix $x\in A$ witnessing that condition $(\ref{radinexists})$ holds. Assume, in order to get a contradiction, that $z\notin B$. Observe that $f(z)\notin A$. On the other hand, $z\in U_{x\re n}$ for every $n\in\omega$, hence  $f(z)\in \Ne_{x\re n}$ for every $n\in\omega$. This clearly implies that $f(z)=x$, which contradicts the fact that $x\in A$.
\end{proof}

\begin{theorem}\label{radintheorem}
Let $\bS$ be a nice topological pointclass, and assume that $\Det(\bS(\omega^\omega))$ holds. Let $\bG\in\NSDS(\omega^\omega)$ be such that $\ell(\bG)\geq 2$. Then $\bG\in\HaS(\omega^\omega)$.
\end{theorem}
\begin{proof}
This follows from Lemmas \ref{levelatleast2} and \ref{radinlemma}.
\end{proof}

We conclude this section by generalizing Theorem \ref{radintheorem} from $\omega^\omega$ to an arbitrary zero-dimensional Polish space $Z$ (see Theorem \ref{vanweseplevelatleast2}). The transfer will be accomplished by exploiting once again the machinery of relativization.

\begin{lemma}\label{hausdorffbairehausdorffeverywhere}
Let $\bS$ be a nice topological pointclass, and assume that $\Det(\bS(\omega^\omega))$ holds. Let $\bG=\bG_D(\omega^\omega)$ for some $D\subseteq\PP(\omega)$, and let $Z$ be a zero-dimensional Polish space. Assume that $\bG\subseteq\bS(\omega^\omega)$. Then $\bG(Z)=\bG_D(Z)$.
\end{lemma}
\begin{proof}
Observe that $\bG\in\NSDS(\omega^\omega)$ by Theorem \ref{addisontheorem}. To see that the inclusion $\subseteq$ holds, pick $A\in\bG(Z)$. By condition $(\ref{existsemb})$ of Lemma \ref{relativizationprelim}, there exist an embedding $j:Z\longrightarrow\omega^\omega$ and $B\in\bG$ such that $A=j^{-1}[B]$. It follows from Lemma \ref{hausdorffrelativization}.\ref{hausdorffrelativizationpreimage} that $A\in\bG_D(Z)$.

To see that the inclusion $\supseteq$ holds, pick $A\in\bG_D(Z)$. By \cite[Theorem 7.8]{kechris} and Lemma \ref{hausdorffrelativization}.\ref{hausdorffrelativizationhomeomorphism}, we can assume without loss of generality that $Z$ is a closed subspace of $\omega^\omega$. By \cite[Proposition 2.8]{kechris}, we can fix a retraction $\rho:\omega^\omega\longrightarrow Z$. Observe that $\rho^{-1}[A]\in\bG_D(\omega^\omega)$ by Lemma \ref{hausdorffrelativization}.\ref{hausdorffrelativizationpreimage}. Since $\bG_D(\omega^\omega)=\bG=\bG(\omega^\omega)$, it follows from Lemma \ref{relativizationsubspace} that $A=\rho^{-1}[A]\cap Z\in\bG(Z)$.
\end{proof}

\begin{theorem}\label{vanweseplevelatleast2}
Let $\bS$ be a nice topological pointclass, and assume that $\Det(\bS(\omega^\omega))$ holds. Let $Z$ be a zero-dimensional Polish space. If $\bG\in\NSDS(Z)$ and $\ell(\bG)\geq 2$ then $\bG\in\HaS(Z)$.
\end{theorem}
\begin{proof}
Let $\bG(Z)\in\NSDS(Z)$, where $\bG\in\NSDS(\omega^\omega)$, and assume that $\ell(\bG(Z))\geq 2$. If $Z$ is countable, then it follows from Proposition \ref{levelcountable} that $\bG(Z)=\{\varnothing\}$ or $\bG(Z)=\{Z\}$, which clearly implies the desired result. So assume that $Z$ is uncountable. Notice that $\ell(\bG)\geq 2$ by Corollary \ref{levelinvariance}, hence $\bG\in\HaS(\omega^\omega)$ by Theorem \ref{radintheorem}. It follows from Lemma \ref{hausdorffbairehausdorffeverywhere} that $\bG(Z)\in\HaS(Z)$.
\end{proof}

\section{The main result}\label{sectionmain}

The following is the main result of this article. It states that every non-selfdual Wadge class can be obtained by starting with all classes of uncountable level, and then suitably iterating the operations of expansion and separated differences. As will become clear from the proof, classes of countable non-zero level are the expansion of some previously considered class, while classes of level zero are the separated differences of some previously considered classes.

At the same time, this yields a more explicit proof of Van Wesep's Theorem \ref{vanwesepmain}, in the sense that (aside from classes of uncountable level) it is made clear exactly which operations generate new classes from the old ones. This was first accomplished by Louveau in the Borel case (see \cite[Theorem 7.3.12]{louveaub}). The proof of the general case is essentially the same, except that the results from Section \ref{sectionstretch} are also needed.

The Borel version of following definition is due to Louveau (see \cite[Corollary 7.3.11]{louveaub}), which explains our choice of notation.
\begin{definition}\label{louveauhierarchy}
Given a space $Z$, define $\Lo(Z)$ as the smallest collection satisfying the following conditions:
\begin{itemize}
\item $\bG\in\Lo(Z)$ whenever $\bG\in\NSD(Z)$ and $\ell(\bG)=\omega_1$,
\item $\bG^{(\xi)}\in\Lo(Z)$ whenever $\bG\in\Lo(Z)$ and $\xi<\omega_1$,
\item $\SD_\eta(\bD,\bG)\in\Lo(Z)$, where $\bD=\bigcup_{n\in\omega}(\bL_n\cup\bLc_n)$, whenever $1\leq\eta<\omega_1$, $\bG\in\Lo(Z)$ and $\bL_n\in\Lo(Z)$ for $n\in\omega$ are such that $\bG\subseteq\bD$ and $\ell(\bL_n)\geq 1$ for each $n$.
\end{itemize}
Also set $\LoS(Z)=\{\bG\in\Lo(Z):\bG\subseteq\bS(Z)\}$ whenever $\bS$ is a topological pointclass.
\end{definition}

We remark that, when $\bS$ is a nice topological pointclass, an equivalent definition of $\LoS(Z)$ can be given by starting with the elements of $\NSDS(Z)$ of uncountable level, then closing under expansions and separated differences as in Definition \ref{louveauhierarchy}. 

\begin{theorem}\label{main}
Let $\bS$ be a nice topological pointclass, and assume that $\Det(\bS(\omega^\omega))$ holds. Let $Z$ be an uncountable zero-dimensional Polish space. Then
$$
\LoS(Z)=\HaS(Z)=\NSDS(Z).
$$
\end{theorem}
\begin{proof}
The inclusion $\LoS(Z)\subseteq\HaS(Z)$ follows from Theorem \ref{vanweseplevelatleast2}, Corollary \ref{expansionsarehausdorff}, and Proposition \ref{sdarehausdorff}. The inclusion $\HaS(Z)\subseteq\NSDS(Z)$ is given by Theorem \ref{addisontheorem}. Now assume, in order to get a contradiction, that $\NSDS(Z)\nsubseteq\LoS(Z)$. By Theorem \ref{wellfounded}, we can pick a $\subseteq$-minimal $\bG\in\NSDS(Z)\setminus\LoS(Z)$.

By Corollary \ref{everyclasshasalevel}, we can fix $\xi\leq\omega_1$ such that $\xi=\ell(\bG)$. Observe that $\xi<\omega_1$ by the definition of $\LoS(Z)$. First assume that $\xi\geq 1$. By Theorem \ref{expansiontheorem}, we can fix $\bL\in\NSDS(Z)$ such that $\bL^{(\xi)}=\bG$. Clearly $\bL\subseteq\bG$, and $\bL=\bG$ is impossible by Corollary \ref{expansionbigger}. Therefore $\bL\subsetneq\bG$, which implies $\bL\in\LoS(Z)$ by the minimality of $\bG$, contradicting the definition of $\LoS(Z)$. It follows that $\xi=0$. Since $\LoS(Z)\subseteq\HaS(Z)$ and $\bG$ is minimal, we can apply Theorem \ref{sdmain}, contradicting again the definition of $\LoS(Z)$.
\end{proof}

\section*{List of symbols and terminology}

The following is a list of most of the symbols and terminology used in this article, organized by the section in which they are defined.

\begin{itemize}
\item[Section \ref{sectionintroduction}:] space, power-set $\PP(Z)$, $\bGc$, selfdual class, Wadge class.
\item[Section \ref{sectionpreliminaries}:] image $f[A]$, inverse image $f^{-1}[B]$, Wadge-reduction $\leq$, strict Wadge-reduc\-tion $<$, Wadge-equivalence $\equiv$, Wadge class $A\wc$, continuously closed, Borel sets $\Borel(Z)$, embedding, differences $\Diff_\eta(A_\mu:\mu<\eta)$, class of differences $\Diff_\eta(\mathbf{\Sigma}^0_\xi(Z))$, game $\Ga(A,X)$, payoff set, determinacy assumption $\Det(\bS)$, Axiom of Determinacy $\AD$, principle of Dependent Choices $\DC$, partition, identity function $\id_Z$, basic clopen set $\Ne_s$, Boolean closure $\bool\bS$, clopen, base, zero-dimensional, Borel space, $\bS^0_\xi$-measurable function, Borel function.
\item[Section \ref{sectionpointclasses}:] topological pointclass $\bS$, nice topological pointclass $\bS$, Baire property assumption $\BP(\bS)$.
\item[Section \ref{sectionwadgebasic}:] collection $\NSD(Z)$ of all non-selfdual Wadge classes, collection $\NSDS(Z)$ of the non-selfdual Wadge classes of complexity $\bS$, retraction, Extended Wadge game $\EW(D,A_0,A_1)$.
\item[Section \ref{sectionselfdual}:] ideal $\II(A)$, $\sigma$-additive, flip-set.
\item[Section \ref{sectionrelativizationbasic}:] relativized class $\bG(Z)$.
\item[Section \ref{sectionhausdorffbasic}:] Hausdorff operation $\HH_D(A_0,A_1,\ldots)$.
\item[Section \ref{sectionhausdorffclasses}:] Hausdorff class $\bG_D(Z)$, collection $\Ha(Z)$ of all Hausdorff classes, collection $\HaS(Z)$ of the Hausdorff classes of complexity $\bS$.
\item[Section \ref{sectionhausdorffuniversal}:] $W$-universal set.
\item[Section \ref{sectionkuratowski}:] function $f^\ast$, $k$-th coordinate function $f_k$, $\xi$-refining function.
\item[Section \ref{sectionexpansionbasic}:] expansion $\bG^{(\xi)}$, Hausdorff expansion $\bG_D^{(\xi)}(Z)$.
\item[Section \ref{sectionlevelbasic}:] partitioned union $\PU_\xi(\bG)$, level $\ell(\bG)$. 
\item[Section \ref{sectionlevelhard}:] tree, infinite branch, terminal node, well-founded tree, rank function $\rho_T$ of a tree, rank $\rho(T)$ of a tree, $T/s$.
\item[Section \ref{sectionexpansioncomposition}:] closed under Borel preimages.
\item[Section \ref{sectionsdbasic}:] separated differences $\SD_\eta((V_{\xi,n}:\xi<\eta,n\in\omega),(A_{\xi,n}:\xi<\eta,n\in\omega),A^\ast)$, class of separated differences $\SD_\eta(\bD,\bG^\ast)$.
\item[Section \ref{sectionstretch}:] stretch $A^\sss$.
\item[Section \ref{sectionmain}:] Louveau hierarchy $\Lo(Z)$, Louveau hierarchy $\LoS(Z)$ of classes of complexity $\bS$.
\end{itemize}

\end{document}